\documentclass[a4paper,10pt,twoside]{article}

\usepackage{cmap} 

\usepackage[french,ngerman,english]{babel}
\usepackage[OT2,T1]{fontenc}
\usepackage[latin1]{inputenc}
\usepackage{mathtools,amsfonts,amssymb,amsthm,amsopn}
\usepackage{mathrsfs}
\usepackage{tikz-cd}
\usetikzlibrary{babel}
\tikzcdset{nodes={execute at begin node=\everymath{\displaystyle}}}
\usepackage{enumitem}

\usepackage[pdftex,
	hyperindex=true, bookmarks=true, bookmarksopen=true, bookmarksnumbered=true,
    pdfauthor={Timo Keller},
    pdftitle={On an analogue of the conjecture of Birch and Swinnerton-Dyer for Abelian schemes over higher dimensional bases over finite fields},
    pdfsubject={},
    pdfkeywords={Abelian varieties, L-functions, higher dimensional function fields, Birch-Swinnerton-Dyer conjecture},
    linkcolor=blue
]{hyperref}

\usepackage{relsize}
\usepackage{fancyhdr}

\usepackage{vmargin}
\setpapersize{A4}
\setmargins{2.1cm}{1.5cm} 
           {17cm}{23.5cm}   
           {2em}{2em}     
           {0pt}{0em}

\usepackage{ae,aecompl}
	\usepackage[step=1,stretch=30,shrink=30,selected=true]{microtype}

\DeclareTextFontCommand\textmathbf{\bfseries\boldmath} 

\newcommand{\N} {\ensuremath{\mathbf{N}}}
\newcommand{\Z} {\ensuremath{\mathbf{Z}}}
\newcommand{\Q} {\ensuremath{\mathbf{Q}}}
\newcommand{\R} {\ensuremath{\mathrm{R}}}
\newcommand{\RR} {\ensuremath{\mathbf{R}}}
\newcommand{\G} {\ensuremath{\mathbf{G}}}
\newcommand{\A} {\ensuremath{\mathbf{A}}}
\newcommand{\Ccal} {\ensuremath{\mathscr{C}}}
\newcommand{\Ocal} {\ensuremath{\mathscr{O}}}
\newcommand{\Ecal} {\ensuremath{\mathscr{E}}}
\newcommand{\Lcal} {\ensuremath{\mathscr{L}}}
\newcommand{\Pcal} {\ensuremath{\mathscr{P}}}
\newcommand{\F} {\ensuremath{\mathbf{F}}}

\newcommand{\PP} {\ensuremath{\mathbf{P}}}

\renewcommand{\H} {\ensuremath{\mathrm{H}}}
\newcommand{\PPic} {\ensuremath{\mathbf{Pic}}}

\renewcommand{\epsilon}{\varepsilon}
\renewcommand{\theta}{\vartheta}
\renewcommand{\phi}{\varphi}

\newcommand*\defn[1]{\textup{\textmathbf{#1}}}

\providecommand{\qu}[1]{\langle#1\rangle}

\providecommand{\pfister}[1]{\langle\kern-0.2em\langle#1\rangle\kern-0.2em\rangle}

\providecommand{\sep}{{\mathrm{sep}}}
\providecommand{\iso}{\ensuremath{\cong}}
\DeclareMathOperator{\ordrm}{ord}
\providecommand{\ord}{\mathop{\ordrm} \limits}

\providecommand{\isoto}{\stackrel{\sim}{\to}}
\DeclareMathOperator{\im}{im}

\providecommand{\card}[1]{\left\lvert#1\right\rvert}

\DeclareMathOperator{\Tr}{Tr}

\DeclareMathOperator{\Aut}{Aut}
\DeclareMathOperator{\End}{End}
\DeclareMathOperator{\Hom}{Hom}
\DeclareMathOperator{\Mor}{Mor}

\DeclareMathOperator{\id}{id}
\DeclareMathOperator{\Frob}{Frob}
\DeclareMathOperator{\CH}{CH}

\DeclareMathOperator{\colim}{\varinjlim}
\DeclareMathOperator{\res}{res}
\DeclareMathOperator{\cores}{cor}

\DeclareMathOperator{\Char}{char}
\DeclareMathOperator{\Br}{Br}
\DeclareMathOperator{\Ext}{Ext}

\DeclareMathOperator{\Spec}{Spec}

\DeclareMathOperator{\pr}{pr}

\DeclareMathOperator{\rk}{rk}
\DeclareMathOperator{\NS}{NS}

\DeclareMathOperator{\Pic}{Pic}
\DeclareMathOperator{\coker}{coker}
\DeclareMathOperator{\Quot}{Quot}
\DeclareMathOperator{\Alb}{Alb}

\providecommand{\divv}{\ensuremath{\mathrm{div}}}
\DeclareMathOperator{\supp}{supp}
\providecommand{\Mod}{\ensuremath{\mathbf{Mod}}}

\providecommand{\Acal}{\ensuremath{\mathscr{A}}}
\providecommand{\Bcal}{\ensuremath{\mathscr{B}}}
\providecommand{\Ecal}{\ensuremath{\mathscr{E}}}
\providecommand{\Fcal}{\ensuremath{\mathscr{F}}}
\providecommand{\Gcal}{\ensuremath{\mathscr{G}}}

\providecommand{\Lcal}{\ensuremath{\mathscr{L}}}

\providecommand{\Ocal}{\ensuremath{\mathscr{O}}}
\providecommand{\EExt}{\ensuremath{\mathscr{E}\mathrm{xt}}}
\providecommand{\HHom}{\ensuremath{\mathscr{H}\mathrm{om}}}

\providecommand{\Het}{\ensuremath{\H_\mathrm{\acute{e}t}}}

\providecommand{\piet}{\ensuremath{\pi_1^\mathrm{\acute{e}t}}}
\newcommand{\tors} {\ensuremath{\mathrm{tors}}}
\newcommand{\Tors} {\ensuremath{\mathrm{nt}}}
\newcommand{\Div} {\ensuremath{\mathrm{nd}}}
\newcommand{\cd} {\ensuremath{\mathrm{cd}}}
\newcommand{\FEt} {\ensuremath{\mathbf{F\acute{E}t}}}

\providecommand{\Zar}{\ensuremath{\mathrm{Zar}}}
\providecommand{\et}{\ensuremath{\mathrm{\acute{e}t}}}
\providecommand{\fppf}{\ensuremath{\mathrm{fppf}}}

\DeclareSymbolFont{cyrletters}{OT2}{wncyr}{m}{n}
\DeclareMathSymbol{\Sha}{\mathalpha}{cyrletters}{"58}

\newtheorem{maintheorem}{Theorem}

\newtheorem{theorem}{Theorem}[subsection]
\newtheorem{lemma}[theorem]{Lemma}
\newtheorem{corollary}[theorem]{Corollary}
\newtheorem{definition}[theorem]{Definition}
\newtheorem{proposition}[theorem]{Proposition}

\theoremstyle{definition}

\theoremstyle{remark}
\newtheorem{remark}[theorem]{Remark}
\newtheorem{example}[theorem]{Example}

\usepackage{cleveref}

\crefname{theorem}{Theorem}{Theorems}
\crefname{lemma}{Lemma}{Lemmata}
\crefname{corollary}{Corollary}{Corollaries}
\crefname{proposition}{Proposition}{Propositions}
\crefname{definition}{Definition}{Definitions}
\crefname{conjecture}{Conjecture}{Conjectures}
\crefname{example}{Example}{Examples}
\crefname{algorithm}{Algorithm}{Algorithms}
\crefname{remark}{Remark}{Remarks}

\numberwithin{equation}{section}

\addto\captionsngerman{%
	\renewcommand{\refname}{Bibliography}%
}

\allowdisplaybreaks[1]

\begin{document}
\nonfrenchspacing


\pagestyle{fancy}

\title{On an analogue of the conjecture of Birch and Swinnerton-Dyer for Abelian schemes over higher dimensional bases over finite fields}
\author{Timo Keller}
\maketitle\thispagestyle{empty}
\begin{abstract}
We formulate an analogue of the conjecture of Birch and Swinnerton-Dyer for Abelian schemes with everywhere good reduction over higher dimensional bases over finite fields of characteristic $p$.  We prove the prime-to-$p$ part conditionally on the finiteness of the $p$-primary part of the Tate-Shafarevich group or the equality of the analytic and the algebraic rank.  If the base is a product of curves, Abelian varieties and K3 surfaces, we prove the prime-to-$p$ part of the conjecture for constant or isoconstant Abelian schemes, in particular the prime-to-$p$ part for (1) relative elliptic curves with good reduction or (2) Abelian schemes with constant isomorphism type of $\Acal[p]$ or (3) Abelian schemes with supersingular generic fibre, and the full conjecture for relative elliptic curves with good reduction over curves and for constant Abelian schemes over arbitrary bases.  We also reduce the conjecture to the case of surfaces as the basis.
\\
\\
{\bf Keywords:} $L$-functions of varieties over global fields; Birch-Swinnerton-Dyer conjecture; Heights; Étale cohomology, higher regulators, zeta and $L$-functions; Abelian varieties of dimension $> 1$; Étale and other Grothendieck topologies and cohomologies; Arithmetic ground fields
\\
\\
{\bf MSC 2010:} 11G40, 11G50, 19F27, 11G10, 14F20, 14K15
\end{abstract}

\markright{}
\tableofcontents


\section{Introduction}
If $K$ is a global field, i.\,e.\ a finite extension of $\Q$ or of $\F_q(t)$, the conjecture of Birch and Swinnerton-Dyer for an Abelian variety $A/K$ relates global invariants, like the rank of the Mordell-Weil group $A(K)$, the order of the Tate-Shafarevich group $\Sha(A/K)$ (a group measuring the failure of the Hasse principle for principal homogeneous spaces of $A/K$) and the determinant of the height pairing $A(K) \times A^t(K) \to \RR$ with $A^t$ the dual Abelian variety to the vanishing order of the $L$-function $L(A/K,s)$ (built up from the number of points of the reduction of $A$ at the primes of $K$) at $s = 1$ and the special $L$-value at this point.  The aim of this article is to extend this setting from the classical situation of a \emph{curve} over a finite field to the case of a \emph{higher dimensional basis} over finite fields.

Even for elliptic curves over the rationals, this is a difficult problem.  The function field case is more accessible since the situation is more geometric as one has a ground field the algebraic closure of which one can pass to, but up to now, there have been only (mostly conditional) results over \emph{curves} over finite fields:  For Abelian varieties over global function fields, John Tate~\cite{Artin-Tate} considered the problem for Jacobians of curves, and the first result is due to James Milne~\cite{Milne-Tate-Shafarevic}:  He proved the conjecture of Birch and Swinnerton-Dyer for \emph{constant} Abelian schemes over global function fields, i.\,e. Abelian schemes of the form $\Acal = A \times_k X$ with $A/k$ an Abelian variety over a finite field $k$ and $X/k$ a smooth projective geometrically connected curve.  Later, Peter Schneider~\cite{Schneider} proved a conditional result for Abelian varieties over global function fields, namely that the prime-to-$p$ part of the conjecture of Birch and Swinnerton-Dyer ($p$ the characteristic of the ground field) holds iff for one $\ell \neq p$, the $\ell$-primary part of the Tate-Shafarevich group is finite.  In~\cite{Bauer}, Werner Bauer proved an analogue of Schneider's result for the prime-to-$p$ part of the conjecture, but only for Abelian varieties with good reduction; finally, Kazuya Kato and Fabien Trihan~\cite{Kato-Trihan} extended Bauer's result to the case of bad reduction.  Tate and Shafarevich~\cite{Tate-Shafarevich} gave examples of elliptic curves over $\F_q(t)$ of arbitrarily large rank and Douglas Ulmer~\cite{Ulmer} proved the conjecture for certain non-isoconstant elliptic curves over $\F_q(t)$ with arbitrarily large rank.

In section~\ref{section:BSD}, we proceed by generalising Schneider's arguments to the case of a higher dimensional basis $X$ over a finite field $k$.  A key point is to find the correct definition of the $L$-function in the higher dimensional setting.  Let $\Acal/X$ be an Abelian scheme.  The Kummer sequence for $\Acal/X$ on the small étale site of $X$ induces a short exact sequence
\[
    0 \to \Acal(X) \otimes_\Z \Z_\ell \to \H^1(X,T_\ell\Acal) \to T_\ell\Sha(\Acal/X) \to 0
\]
with $\H^1(X,T_\ell\Acal) = \varprojlim_n\H^1(X,\Acal[\ell^n])$.  Since $\Sha(\Acal/X)[\ell^\infty]$ is cofinitely generated, $T_\ell\Sha(\Acal/X) = 0$ iff $\Sha(\Acal/X)[\ell^\infty]$ is finite.  This gives us the link between the algebraic rank $\rk_{\Z} \Acal(X)$ and $\H^1(X,T_\ell\Acal)$.  Using the Hochschild-Serre spectral sequence $\H^p(G_k,\H^q(\overline{X},T_\ell\Acal)) \Rightarrow \H^{p+q}(X,T_\ell\Acal)$, one relates $\H^1(X,T_\ell\Acal)$ to $\H^1(\overline{X},T_\ell\Acal)^{G_k}$.  Then one uses~\cref{lemma:BayerNeukirch} to relate the vanishing order of the $L$-function to the algebraic rank and the special $L$-value at $s=1$ to orders of cohomology groups and determinants of cohomological pairings.  The proof is complicated by the fact that one has more non-vanishing cohomology groups than in the case of a curve as a basis.  For example, setting $d = \dim{X}$, if $d = 1$, Poincaré duality is a pairing between $\H^1(\overline{X},\Fcal) \times \H^1(\overline{X},\Fcal^\vee(1)) \to \Z_\ell$, whereas for general $d > 1$, it is a pairing $\H^1(\overline{X},\Fcal) \times \H^{2d-1}(\overline{X},\Fcal^\vee(d)) \to \Z_\ell$.

In section~\ref{subsection:pairingacutebrackets} and~\ref{subsection:pairingroundbrackets}, we study two cohomological pairings given by cup product in cohomology:
\begin{align*}
    \qu{\cdot,\cdot}_\ell&: \H^1(X, T_\ell\Acal)_\Tors \times \H^{2d-1}(X, T_\ell(\Acal^t)(d-1))_\Tors \to \H^{2d}(X,\Z_\ell(d)) \stackrel{\pr_1^*}{\to} \H^{2d}(\overline{X},\Z_\ell(d)) = \Z_\ell\\
    (\cdot,\cdot)_\ell&: \H^2(X, T_\ell\Acal)_\Tors \times \H^{2d-1}(X, T_\ell(\Acal^t)(d-1))_\Tors \to \H^{2d+1}(X,\Z_\ell(d)) = \Z_\ell
\end{align*}
If one $\ell$-primary component of the Tate-Shafarevich group of $\Acal/X$ is finite, we relate the pairing $\qu{\cdot,\cdot}_\ell$ to the Néron-Tate height pairing, and show that the determinant of the pairing $(\cdot,\cdot)_\ell$ equals $1$.  This is done by generalising Schneider's arguments comparing $\qu{\cdot,\cdot}_\ell$ with Bloch's height pairing from~\cite{BlochNote}.  Again, the higher dimensional case is more involved.

In section~\ref{sec:IsoconstantAbelianScheme}, we specialise to the case of an isoconstant Abelian scheme, and deduce in section~\ref{sec:Verification} from a descent theorem of our previous article~\cite[p.~238, Theorem~4.29]{KellerSha} our analogue of the conjecture of Birch and Swinnerton-Dyer for relative elliptic curves or Abelian schemes with constant isomorphism type of $\Acal[p]$ over products of curves and Abelian varieties by showing these are isoconstant since the moduli scheme $Y(N)$ is affine for $N \geq 3$ resp.\ since the Ekedahl-Oort stratification is quasi-affine.  We also prove the conjecture for supersingular Abelian schemes.

In section~\ref{sec:Reduction}, we reduce the conjecture to the case of a surface (and in special cases also of a curve) as a basis using Poonen's Bertini theorem for varieties over finite fields.

Our main results are as follows:

In section~\ref{section:BSD}, we first introduce a suitable $L$-function $L(\Acal/X,s)$ for Abelian schemes $\Acal$ over a smooth projective base scheme $X$ over a finite field of characteristic $p$ (see~\cref{rem:DefinitionOfLFunction} for a motivation):
\[
    L(\Acal/X,t) = \frac{\det(1 - t\Frob_q^{-1} \mid \H^1(\overline{X}, V_\ell\Acal))}{\det(1 - t\Frob_q^{-1} \mid \H^0(\overline{X}, V_\ell\Acal))}
\]
We then prove that an analogue of the conjecture of Birch and Swinnerton-Dyer holds for the prime-to-$p$ part, with two cohomological pairings $\qu{\cdot,\cdot}_\ell$ and $(\cdot,\cdot)_\ell$ in place of the height pairing, provided that for one $\ell \neq p$ the $\ell$-primary component of the Tate-Shafarevich group $\Sha(\Acal/X) := \Het^1(X,\Acal)$ is finite or, equivalently, if the analytic rank equals the algebraic rank.

The Tate-Shafarevich group is studied in a previous article~\cite[section~4]{KellerSha}, especially Theorem~4.4 and~4.5.  There, we show:
\[
	\Sha(\Acal/X) = \ker\Big(\H^1(K, \Acal) \to \prod_{x \in S} \H^1(K_x^{nr}, \Acal) \Big),
\]
where $K_x^{nr} = \Quot(\Ocal_{X,x}^{sh})$, and $S$ is either (a) the set of all points of $X$, or (b) the set $|X|$ of all closed points of $X$, or (c) the set $X^{(1)}$ of all codimension-$1$ points of $X$, and $\Acal = \PPic^0_{\Ccal/X}$ for a relative curve $\Ccal/X$ with everywhere good reduction admitting a section, and $X$ is a variety over a finitely generated field.  Here, one can replace $K_x^{nr}$ by $K_x^h = \Quot(\Ocal_{X,x}^h)$ if $\kappa(x)$ is finite, and $K_x^{nr}$ and $K_x^h$ by $\Quot(\hat{\Ocal}_{X,x}^{sh})$ and $\Quot(\hat{\Ocal}_{X,x}^h)$, respectively, if $x \in X^{(1)}$.

More precisely, we get the following first main result:
\begin{maintheorem}[\cref{thm:BSDI}]
Let $X/k$ be a smooth projective geometrically connected variety over a finite field $k = \F_q$ and $\Acal/X$ an Abelian scheme.  Set $\overline{X} = X \times_k \overline{k}$ and let $\ell \neq \Char{k}$ be a prime.  Let $\rho$ be the vanishing order of $L(\Acal/X,s)$ at $s=1$ and define the special $L$-value $c = L^*(\Acal/X,1)$ of $L(\Acal/X,s)$ at $s=1$ by
\begin{align*}
    L(\Acal/X,s) &\sim c\cdot(1-q^{1-s})^\rho \sim c\cdot(\log{q})^\rho(s-1)^\rho \quad\text{for $s \to 1$}.
\end{align*}
Then one has $\rho \geq \rk_{\Z}\Acal(X)$, and the following statements are equivalent:\\
\noindent (a) $\rho = \rk_{\Z}\Acal(X)$\\
\noindent (b) $\Sha(\Acal/X)[\ell^\infty]$ is finite\\
\noindent If these hold, one has for all $\ell \neq \Char{k}$ the equality
\[
    |c|_\ell^{-1} = \frac{\card{\Sha(\Acal/X)[\ell^\infty]} \cdot R_\ell(\Acal/X)}{\card{\Acal(X)[\ell^\infty]_\tors} \cdot \card{\H^2(\overline{X},T_\ell\Acal)^\Gamma}}
\]
and the prime-to-$p$ part of the Tate-Shafarevich group $\Sha(\Acal/X)[\text{non-$p$}]$ is finite.  Here $\Acal(X) = A(K)$ with $A$ the generic fibre of $\Acal/X$ and $K = k(X)$ the function field of $X$, and the regulator $R_\ell(\Acal/X)$ is the determinant of a cohomological pairing $\qu{\cdot,\cdot}_\ell$~\eqref{eq:Regulator spitze Klammer} divided by the determinant of a cohomological pairing $(\cdot,\cdot)_\ell$~\eqref{eq:Regulator runde Klammer}.
\end{maintheorem}
For example, (a) holds if $L(\Acal/X,1) \neq 0$ (\cref{rem:Lfn}\,(a)), and (b) holds under mild conditions if $\Acal/X$ is isoconstant (\cref{thm:MilnesMainTheorem}, \cref{rem:ShaConstant} and~\cref{thm:isotrivial}).

In section~\ref{subsection:pairingacutebrackets} and~\ref{subsection:pairingroundbrackets}, we construct a higher-dimensional analogue
\[
    \qu{\cdot,\cdot}: A(K) \times A^t(K) \to \log{q}\cdot\Z
\]
of the Néron-Tate canonical height pairing with $A^t$ the dual Abelian variety, and show the second main result, which identifies the cohomological regulator $R_\ell(\Acal/X)$ in Theorem~1 with a geometric one:
\begin{maintheorem}[\cref{thm:comparisonofpairings} and \cref{thm:rundeKlammerPaarungDeterminante1}]
Let $\ell$ be a prime different from $\Char{k}$.  Assume that $\Sha(\Acal/X)[\ell^\infty]$ is finite.\\
\noindent (a) The Néron-Tate canonical height pairing $\qu{\cdot,\cdot}$ gives the pairing $\qu{\cdot,\cdot}_\ell$ after tensoring with $\Z_\ell$ up to a known factor, the integral hard Lefschetz defect, see~\cref{def:integralhardLefschetzdefect}.\\
\noindent (b) The cohomological pairing $(\cdot,\cdot)_\ell$ has determinant $1$.
\end{maintheorem}
More precisely, the pairing $\qu{\cdot,\cdot}$ depends on the choice of a very ample line bundle on $X$, but the comparison isomorphism also, and the two choices cancel each other; see~\cref{rem:HeightPairingAndCohPairing}.  For (a), see~\cref{thm:comparisonofpairings}, and~\cref{thm:rundeKlammerPaarungDeterminante1} for (b).  In~\cref{thm:HeightAndTracePairing}, we identify the cohomological pairing $\qu{\cdot,\cdot}_\ell$ with a trace pairing in the case of $\Acal/X$ a constant Abelian variety, and in~\cref{thm:TraceAndHeightPairingForCurves} with another pairing if $X$ is a curve.

We prove our analogue conjecture of Birch and Swinnerton-Dyer for constant Abelian schemes unconditionally:
\begin{maintheorem}[\cref{thm:BSDII}]
Let $X/k$ be a smooth projective geometrically connected variety over a finite field $k = \F_q$ and $B/k$ an Abelian variety of dimension $d$.  Set $\overline{X} = X \times_k \overline{k}$ and $\Acal = B \times_k X$, and let $K = k(X)$ be the function field of $X$.  The $L$-function of $\Acal/X$ is defined in~\cref{def:LFunctionMilne}.  Assume\\
\noindent (a) the Néron-Severi group of $\overline{X}$ is torsion-free and\\
\noindent (b) the dimension of $\H^1_{\Zar}(\overline{X},\Ocal_{\overline{X}})$ as a vector space over $\overline{k}$ equals the dimension $g$ of the Albanese variety of $\overline{X}/\overline{k}$.

\noindent Then:

\noindent 1.  The Tate-Shafarevich group $\Sha(\Acal/X)$ is finite.

\noindent 2.  The vanishing order equals the Mordell-Weil rank $r$: $\ord\nolimits_{s=1}L(\Acal/X,s) = \rk \Acal(X) = \rk A(K)$.

\noindent 3.  There is the equality for the leading Taylor coefficient
    \[
        L^*(\Acal/X,1) = q^{(g-1)d} (\log q)^r\frac{\card{\Sha(\Acal/X)} \cdot R(\Acal/X)}{\card{\Acal(X)_\tors}}.
    \]
    Here, $R(\Acal/X)$ is the determinant of the trace pairing $\Hom(A,B) \times \Hom(B,A) \rightarrow \End(A) \stackrel{\mathrm{tr}}{\to} \Z$ with $A$ the Albanese variety of $X$, or, see~\cref{thm:HeightAndTracePairing}, the determinant of a cohomological pairing, and, if $X$ is a curve, the determinant of another pairing or the Néron-Tate canonical height pairing, see~\cref{thm:TraceAndHeightPairingForCurves}.
\end{maintheorem}

Combining the finiteness of $\Sha(\Acal/X)$ for constant $\Acal/X$~\cite[p.~98, Theorem~2]{Milne-Tate-Shafarevic} and the descent of finiteness of $\Sha$ under $\ell'$-alterations~\cite[p.~238, Theorem~4.29]{KellerSha} we obtain:
\begin{maintheorem}[\cref{thm:MilnesMainTheorem} and~\cref{thm:isotrivial}] \label[theorem]{mainthm:BSDforIsoconstant}
Let $X/k$ be a smooth projective geometrically connected variety over a finite field $k = \F_q$ and $\Acal/X$ an isoconstant Abelian scheme, i.\,e.\ such that there exists a proper, surjective, generically étale morphism $f: X' \to X$ such that $f^*\Acal := \Acal \times_X X'/X'$ is constant. Assume that (a) the Néron-Severi group of $\overline{X'}$ is torsion-free and (b) the dimension of $\H^1_{\Zar}(\overline{X'},\Ocal_{\overline{X'}})$ as a vector space over $\overline{k}$ equals the dimension of the Albanese variety of $\overline{X'}/\overline{k}$.  Then the prime-to-$p$ part of the conjecture of Birch and Swinnerton-Dyer holds for $\Acal/X$.
\end{maintheorem}
Note that we do not need $f$ to be of generical degree prime to $\ell$ since $\Acal/X$ is $\ell'$-isoconstant (isoconstant for a generically étale morphism $f: X' \to X$ of generical degree prime to $\ell$) for some $\ell$, and then we can use (a) $\implies$ (b) from~\cref{thm:BSDI} to get independence from $\ell$.  This also extends the known, classical results for Abelian varieties over one-dimensional global function fields, where the constant case had be settled by Milne~\cite[p.~100, Theorem~3]{Milne-Tate-Shafarevic}.  In~\cref{thm:EllipticCurvesIsoconstant}, we prove that relative \emph{elliptic curves} are isoconstant and conclude with
\begin{maintheorem}[\cref{cor:BSDoverCurves} and \cref{cor:BrauerGroupFinite}] \label[theorem]{mainthm:BSDforRelativeEllipticCurvesEtc}
Let $X$ be a product of smooth proper curves, Abelian varieties and K3 surfaces over a finite field of characteristic $p$.  Now let $\Acal$ be an Abelian $X$-scheme belonging to one of the following three classes:
\begin{enumerate}
	\item a relative elliptic curve
	\item an Abelian scheme such that the isomorphism type of $\Acal[p]$ is constant
	\item an Abelian scheme with supersingular generic fibre
\end{enumerate}
Then the prime-to-$p$ part of our analogue of the conjecture of Birch and Swinnerton-Dyer holds for $\Acal/X$ and, if $\Acal/X$ is a relative elliptic curve, $\Br(\Acal)[\text{non-$p$}]$ is finite.  If $X$ is a curve, the full conjecture of Birch and Swinnerton-Dyer holds for $\Acal/X$.  Furthermore, the Tate conjecture holds in dimension $1$ for $\Acal$.

Let $C/\F_q$ be a smooth proper geometrically connected curve and $\Ecal/C$ be a relative elliptic curve.  Then $\Br(\Ecal) = \Sha(\Ecal/C)$ is finite and of square order, and the Tate conjecture holds for $\Ecal$.
\end{maintheorem}

In the final section~\ref{sec:Reduction}, we reduce the conjecture to the case of a surface as a basis:
\begin{maintheorem}[\cref{thm:ReductionToSurface}]
	If the analogue of the conjecture of Birch-Swinnerton-Dyer holds for a prime $\ell$ invertible on the base and for all Abelian schemes over all smooth projective geometrically integral surfaces, then it holds over arbitrary dimensional bases.
	
	More precisely, if there is a sequence $S \hookrightarrow \ldots \hookrightarrow X$ of ample smooth projective geometrically integral hypersurface sections with a surface $S$ and the conjecture holds for $\Acal/S$, then it holds for $\Acal/X$.

	If there is a smooth projective ample geometrically integral curve $C \hookrightarrow S$ with $\rk \Acal(S) = \rk \Acal(C)$, the analogue of the conjecture of Birch and Swinnerton-Dyer for $\Acal/S$ is equivalent to the conjecture for $\Acal/C$.
\end{maintheorem}


\paragraph{Notation.}
Let $\N = \{0,1,2,\ldots\}$ be the set of natural numbers.  Canonical isomorphisms are often denoted by ``$=$''.

We denote Pontrjagin duality by $(-)^D$ (see~\cite[§\,1]{CohomologyOfNumberFields}), duals of $R$-modules or $\ell$-adic sheaves by $(-)^\vee$, and duals of Abelian schemes and Cartier duals by $(-)^t$.

The $\ell$-adic valuation $|\cdot|_\ell$ is taken to be normalised by $|\ell|_\ell = \ell^{-1}$.

If $\Gamma$ is a group acting on an Abelian group $A$, we denote by $A^\Gamma$ invariants and by $A_\Gamma$ coinvariants.  By $X^{(i)}$, we denote the set of codimension-$i$ points of a scheme $X$, and by $|X|$ the set of closed points.  For an Abelian variety $A$, we denote its Poincaré bundle by $\Pcal_A$.

For an Abelian group $A$, let $A_\tors$ be the torsion subgroup of $A$, and $A_\Tors = A/A_\tors$.  Let $A_\divv$ be the maximal divisible subgroup of $A$ (in general strictly contained in the subgroup of divisible elements of $A$, but see~\cref{lemma:Aellendlich} below) and $A_\Div = A/A_\divv$.  For an integer $n$ and an object $A$ of an Abelian category, denote the cokernel of $A \stackrel{n}{\to} A$ by $A/n$ and its kernel by $A[n]$, and for a prime $p$ the $p$-primary subgroup $\colim_n A[p^n]$ by $A[p^\infty]$.  Write $A[\text{non-}p]$ for $\varinjlim_{p \nmid n}A[n]$.  For a prime $\ell$, let the $\ell$-adic Tate module $T_\ell A$ be $\varprojlim_n A[\ell^n]$ and the rationalised $\ell$-adic Tate module $V_\ell A = T_\ell A \otimes_{\Z_\ell} \Q_\ell$.  The corank of $A[p^\infty]$ is the $\Z_p$-rank of $A[p^\infty]^D = T_pA$.

Denote the absolute Galois group of a field $k$ by $G_k$.

Varieties over a field $k$ are schemes of finite type over $\Spec{k}$.  For the class $[\Lcal]$ of a line bundle in the Picard group of a scheme, we write $\Lcal$.  If not stated otherwise, all cohomology groups are taken with respect to the étale topology.

An $\ell$-adic sheaf on a scheme $X$ is a projective system $(\Fcal_n)_{n \in \Z}$ of étale sheaves on $X$ such that all $\Fcal_n$ are constructible, $\Fcal_n = 0$ for $n < 0$, $\ell^{n+1}\Fcal_n = 0$ for $n \geq 0$ and $\Fcal_{n+1}/\ell^{n+1} \isoto \Fcal_n$ (see~\cite[p.~122, Definition~12.6]{Freitag-Kiehl}).    For example, the $\ell$-adic Tate module $T_\ell\Acal = (\Acal[\ell^n])_{n \in \N}$ is an $\ell$-adic sheaf on $X$ for $\Acal/X$ an Abelian scheme and $\ell$ invertible on $X$ (see~\cref{cor:TellAladicsheaf}).

\section{The $L$-function and the cohomological BSD formula} \label{section:BSD}
The main theorem~\cref{thm:BSDI} of this section is a conditional result on our analogue of the conjecture of Birch and Swinnerton-Dyer over higher dimensional bases over finite fields.

The results in this section are a generalisation of results of Schneider~\cite[p.~134--138]{SchneiderZeta} and~\cite[p.~496--498]{Schneider}.

Let $k = \F_q$ be a finite field with $q = p^n$ elements and let $\ell \neq p$ be a prime.  For a variety $X/k$ denote by $\overline{X}$ its base change to an algebraic closure $\overline{k} = k^{\mathrm{sep}}$ of $k$.

Denote by $\Frob_q$ the arithmetic Frobenius, the inverse of the geometric Frobenius as defined in~\cite[p.~5]{KiehlWeissauer} and by $\Gamma$ the absolute Galois group of the finite base field $k$.

Let $X/k$ be a smooth projective geometrically connected variety of dimension $d$, and let $\Acal/X$ be an Abelian scheme.

\subsection{Tate modules of Abelian groups}

We often use the following basic properties of the Tate module:

\begin{lemma} \label[lemma]{lemma:PropertiesOfTateModule}
	Let $A$ be an Abelian group and $\ell$ a prime.
	\begin{enumerate}[label=(\roman*)]
		\item There is a canonical isomorphism $\Hom(\Q_\ell/\Z_\ell, A) = T_\ell A$. \label{lemma:HomUndTateModul}
		\item If $A$ is finite, $T_\ell A$ is trivial. \label{lemma:Tate-module-of-finite-group-is-trivial}
		\item If $A$ is an $\ell$-primary torsion group such that $A[\ell]$ is finite, then $A$ is cofinitely generated and the maximal divisible subgroup $A_\divv$ of $A$ coincides with the subgroup of divisible elements of $A$. \label{lemma:Aellendlich}
		\item If $A$ is an $\ell$-primary torsion group and cofinitely generated, $T_\ell A = 0$ implies $A$ finite. \label{lemma:Tate-module-of-is-trivial-so-finite}
		\item The $\Z_\ell$-module $T_\ell A$ is torsion-free. \label{lemma:Tate-module-is-torsion-free}
	\end{enumerate}
\end{lemma}
\begin{proof}
	The statements~\ref{lemma:HomUndTateModul}, \ref{lemma:Tate-module-of-finite-group-is-trivial} and~\ref{lemma:Tate-module-is-torsion-free} are well-known.
	
	\ref{lemma:Aellendlich}: Equip $A$ with the discrete topology.  Applying Pontrjagin duality to $0 \to A[\ell] \to A \stackrel{\ell}{\to} A$ gives us that $A^D/\ell$ is finite, hence by~\cite[p.~179, Proposition~3.9.1]{CohomologyOfNumberFields} ($A^D$ being profinite as a dual of a discrete torsion group), $A^D$ is a finitely generated $\Z_\ell$-module, hence $A$ a cofinitely generated $\Z_\ell$-module.  For the second statement see~\cite[p.~30, Lemma~3.3.1]{JossenPhD}.
	
	\ref{lemma:Tate-module-of-is-trivial-so-finite}: Since $A$ is a cofinitely generated $\ell$-primary Abelian group, $A \iso B \oplus (\Q_\ell/\Z_\ell)^r$ with $B$ finite and $r \in \N$ (by the structure theorem of finitely generated modules over the principal ideal domain $\Z_p$ since the Pontrjagin dual of $A$ is a finitely generated $\Z_\ell$-module), so $T_\ell A \iso T_\ell B \oplus \Z_\ell^r = \Z_\ell^r$ by~\ref{lemma:Tate-module-of-finite-group-is-trivial}, hence $r = 0$ since $T_\ell A$ is finite, so $A \iso B$ is finite.
\end{proof}

\begin{remark}
	Note that, in contrast, for an $\ell$-adic sheaf $(\Fcal_n)_{n \in \N}$, $\varprojlim_n\H^i(X,\Fcal_n)$ need \emph{not} be torsion-free.
\end{remark}

\subsection{The yoga of weights}

\begin{definition}
A $\Q_\ell[\Gamma]$-module is said to be \defn{pure of weight $n$} if all eigenvalues $\alpha$ of the geometric Frobenius automorphism $\Frob_q^{-1}$ are algebraic integers which have absolute value $q^{n/2}$ under all embeddings $\iota: \Q(\alpha) \hookrightarrow \mathbf{C}$.
\end{definition}
For the definition of a smooth sheaf see~\cite[p.~7\,f., Definition~1.2]{KiehlWeissauer} and of a sheaf pure of weight $n$, see~\cite[p.~13, Definition~2.1\,(3)]{KiehlWeissauer}.  We often use the yoga of weights (without further mentioning):
\begin{theorem} \label[theorem]{thm:YogaOfWeights}
Let $f: X \to Y$ be a smooth proper morphism of schemes of finite type over $\F_q$ and $\Fcal$ a smooth sheaf pure of weight $n$.  Then $\R^if_*\Fcal$ is a smooth sheaf pure of weight $n+i$ for any $i$.
\end{theorem}
\begin{proof}
Apply Poincaré duality to~\cite[p.~138, Théorème~1]{DeligneWeilII}.
\end{proof}

\begin{definition} \label[definition]{def:Tatetwist}
Let $V$ be a $\Z_\ell[\Gamma]$-module.  Its \defn{$i$-th Tate twist} $V(i)$ is defined as $V(i) = V \otimes_{\Z_\ell} \Z_\ell(i)$ where $\Z_\ell(i) = \varprojlim_n\mu_{\ell^n}^{\otimes i}$ if $i \geq 0$ (let $\mu_{\ell^n}^{\otimes 0} = \Z/\ell^n$) and $\Z_\ell(i) = \Z_\ell(-i)^\vee$ if $i < 0$.
\end{definition}

\begin{lemma} \label[lemma]{lemma:WeightOfTateTwist}
Let $V$, $W$ be $\Q_\ell[\Gamma]$-modules pure of weight $m$ and $n$, respectively.

(a) The tensor product $V \otimes_{\Q_\ell} W$ is a $\Q_\ell[\Gamma]$-module pure of weight $m + n$.

(b) $\Hom_{\Q_\ell}(V,W)$ is a $\Q_\ell[\Gamma]$-module pure of weight $n - m$.  In particular, $V^\vee$ is pure of weight $-m$.

(c) The $i$-th Tate twist $V(i)$ is pure of weight $m-2i$.
\end{lemma}
\begin{proof}
This follows from~\cite[p.~154, (1.2.5)]{DeligneWeilII}.
\end{proof}

\begin{lemma} \label[lemma]{lemma:WeightNotEqualZeroTrivial}
If $V$ and $W$ are $\Q_\ell[\Gamma]$-modules pure of weights $m \neq n$, every $\Gamma$-morphism $V \to W$ is zero.

Consequently, if $V$ is a $\Q_\ell[\Gamma]$-module pure of weight $\neq 0$, $V_\Gamma = V^\Gamma = 0$.  
\end{lemma}
\begin{proof}
For the first statement, see~\cite[p.~4, Fact~2]{JannsenWeights}.

The second statement follows from the first one:  For $V$ pure of weight $m$, $V_\Gamma$ and $V^\Gamma$ are pure of weight $0$ since $\Gamma$ acts as the identity.  The inclusion $V^\Gamma \hookrightarrow V$ is a $\Gamma$-morphism and if the weight $m$ of $V$ is $\neq 0$, this morphism is zero and injective, so $V^\Gamma = 0$.  Analogously, consider the $\Gamma$-morphism $V \twoheadrightarrow V_\Gamma$.
\end{proof}

\subsection{Isogenies of commutative group schemes}

\begin{definition} \label[definition]{def:isogeny}
An \defn{isogeny} of commutative group schemes $G,H$ of finite type over an arbitrary base scheme $S$ is a group scheme homomorphism $f: G \to H$ such that for all $s \in S$, the induced homomorphism $f_s: G_s \to H_s$ on the fibres over $s$ is finite and surjective on identity components.
\end{definition}

\begin{remark}
	See~\cite[p.~180, Definition~4]{BLR}.  We will usually consider isogenies between Abelian schemes, for example the finite flat $n$-multiplication, which is étale iff $n$ is invertible on the base scheme.
\end{remark}

\begin{lemma} \label[lemma]{groupschemes-epi}
Let $G, G'$ be commutative group schemes over a scheme $S$ which are smooth and of finite type over $S$ with connected fibres and $\dim G = \dim G'$ and let $f: G' \to G$ be a morphism of commutative group schemes over $S$.

If $f$ is flat (respectively, étale) then $\ker(f)$ is a flat (respectively, étale) group scheme over $S$, $f$ is quasi-finite, surjective and defines an epimorphism in the category of flat (respectively, étale) sheaves over $S$.
\begin{proof}
(This is the (corrected) exercise 2.19 in~\cite[p.~67, II\,§\,2]{MilneEtaleCohomology}.)
Since $\ker(f) \to S$ is the base change of $f$
along the unit-section of $G$, it is flat (respectively, étale).
That $f$ is surjective and quasi-finite can be checked fibrewise for $s \in S$.
By the flatness and~\cite[p.~178, §\,7.3 Lemma~1]{BLR}, we have that $f_s$ is finite and flat. So the image of $f_s$ is open and closed in $G_s$. Since $G_s$ is connected by assumption, $f_s$ must be surjective.

Now let $T$ be an $S$-scheme, $g \in \Hom_S(T,G)$ and $T'$ the fibre product of $G'$ and $T$ along $f$ and $g$:
\[\begin{tikzcd}
	T' \arrow[d,"f'",twoheadrightarrow] \arrow[r,"g'"] & G' \arrow[d,"f",twoheadrightarrow] \\
	T  \arrow[r,"g"] & G
\end{tikzcd}\]
Then the base change $f'$ of $f$ is again flat (respectively, étale) and surjective, and so is a covering in the stated topology. Hence, then the base change
$g' \in \Hom_S(T',G')$ of $g$ is a local lift of $g$ in that topology. So the claim follows.
\end{proof}
\end{lemma}

\begin{lemma}
\label[lemma]{groupschemes-mult}
Let $S$ be a scheme and $f: G \to S$ be a smooth commutative group scheme over $S$ and $n$ an integer invertible on $S$.
Then the multiplication map $[n]: G \to G$ is étale and the $n$-torsion subgroup scheme $G[n]:=\ker([n]) \to S$ is an étale group scheme over $S$.

If, furthermore, $f$ is of finite type with connected fibres, then $[n]$ is surjective and induces an epimorphism in the category of étale sheaves over $S$.
\begin{proof}
For the first statement use~\cite[p.~179, §\,7.3 Lemma~2\,(b)]{BLR}. Note that the assumption ``of finite type'' is not needed here (see also~\cite[II\,3.9.4]{SGA3}). The morphism $\ker([n]) \to S$ is just the base change of $[n]$ along the unit-section.
For the second part apply~\cref{groupschemes-epi}.
\end{proof}
\end{lemma}

\begin{corollary}[Kummer sequence] \label[corollary]{cor:Kummersequence}
Let $\Acal/S$ be an Abelian scheme and let $\ell$ be invertible on $S$.  Then one has for every $n \geq 1$ a short exact sequence
\[
    0 \to \Acal[\ell^n] \to \Acal \stackrel{[\ell^n]}{\to} \Acal \to 0
\]
of étale sheaves on $S$.
\end{corollary}
\begin{proof}
This follows from~\cref{groupschemes-mult} since $[\ell^n]$ is étale by~\cite[p.~147, Proposition~20.7]{MilneAbelianVarieties} and since Abelian schemes have connected fibres.
\end{proof}

\subsection{Tate modules of Abelian schemes}

\begin{definition}
	Let $k$ be a field, $\ell \neq \Char{k}$ be a prime and $A/k$ be an Abelian variety.  The \defn{$\ell$-adic Tate module} $T_\ell A$ is the $\Z_\ell[G_k]$-module $\varprojlim_n A[\ell^n](k^\sep)$.
\end{definition}

Note that $A[\ell^n]/k$ is finite étale since $\ell$ is invertible in $k$, and hence $A[\ell^n](k^\sep) = A[\ell^n](\overline{k})$.

\begin{proposition} \label[proposition]{thm:Tate module and etale cohomology}
Let $K$ be an arbitrary field, $\ell \neq \Char{K}$ be prime and $A/K$ an Abelian variety.  Let $\overline{A} = A \times_K \overline{K}$.  Then we have an isomorphism of ($\ell$-adic discrete) $G_K$-modules, equivalently, by~\cite[p.~53, Theorem~II.1.9]{MilneEtaleCohomology}, of ($\ell$-adic) étale sheaves on $\Spec{K}$,
\[
    T_\ell(A) = \H^1(\overline{A}, \Z_\ell)^\vee.
\]
In particular, $T_\ell(A)$ is pure of weight $-1$.
\end{proposition}
\begin{proof}
Consider the Kummer sequence
\[
    1 \to \mu_{\ell^n} \to \G_m \stackrel{\ell^n}{\to} \G_m \to 1
\]
on $\overline{A}$.  Taking étale cohomology, one gets an exact sequence of $G_K$-modules
\[
    0 \to \G_m(\overline{A})/\ell^n \to \H^1(\overline{A},\mu_{\ell^n}) \to \H^1(\overline{A}, \G_m)[\ell^n] \to 0.
\]
Since $\Gamma(\overline{A},\Ocal_{\overline{A}}) = \overline{K}$ is separably closed and $\ell \neq \Char{K}$, $\G_m(\overline{A})$ is $\ell$-divisible (one can extract $\ell$-th roots), and hence
\[
    \H^1(\overline{A},\mu_{\ell^n}) \isoto \H^1(\overline{A}, \G_m)[\ell^n] = \Pic(\overline{A})[\ell^n] = \Pic^0(\overline{A})[\ell^n],
\]
the latter equality since $\NS(\overline{A})$ is torsion-free by~\cite[p.~178, Corollary~2]{MumfordAbelianVarieties}.  Taking Tate modules $\varprojlim_n$ yields
\begin{equation} \label{eq:H1Zl1isotoTellPic0}
    \H^1(\overline{A},\Z_\ell(1)) \isoto T_\ell\Pic^0(\overline{A}),
\end{equation}
so (the first equality coming from the perfect Weil pairing~\eqref{eq:HomTellAZell})
\[
    \Hom(T_\ell A, \Z_\ell(1)) = T_\ell(A^t) = \H^1(\overline{A},\Z_\ell(1)),
\]
so
\[
    (T_\ell A)^\vee = \Hom(T_\ell A, \Z_\ell) = \H^1(\overline{A},\Z_\ell),
\]
so
\[
    T_\ell A = \H^1(\overline{A},\Z_\ell)^\vee.
\]

Alternatively, $\pi_1(A, 0) = \prod_\ell T_\ell(A)$ by~\cite[p.~171]{MumfordAbelianVarieties}, and $\H^1(\overline{A}, \Z_\ell) = \Hom(\pi_1(A, 0), \Z_\ell)$ by~\cite[p.~231, Proposition~4.14]{KellerSha}.
\end{proof}

\begin{remark}
Note that both $T_\ell(-)$ and $\H^1(-,\Z_\ell)^\vee$ are covariant functors.
\end{remark}

\begin{proposition} \label[proposition]{prop:RiundTateModul}
Let $S$ be a locally Noetherian scheme, $\pi: \Acal \to S$ be a projective Abelian scheme over $S$. Let $\ell$ be a prime number invertible on $S$.  Then we have a canonical isomorphism $\R^1\pi_*\Z_\ell(1) = T_\ell\Acal^t$ as $\ell$-adic étale sheaves on $S$.  In particular, $T_\ell\Acal$ has weight $-1$.
\end{proposition}
\begin{proof}
Applying the functor $\pi_*$ on the exact Kummer sequence
\[
    1 \to \mu_{\ell^n} \to \G_{m,\Acal} \stackrel{\ell^n}{\to} \G_{m,\Acal} \to 1
\]
of étale sheaves on $\Acal$, we get an exact sequence
   \[ 1 \to \pi_*\G_{m,\Acal}/\ell^n \to \R^1\pi_*\mu_{\ell^n} \to \R^1\pi_*\G_{m,\Acal}[\ell^n] \to 0. \]
of étale sheaves on $S$.
The first term will vanish by following arguments.
Since $\pi: \Acal \to S$ is proper and its geometric fibres are integral by definition, we get the isomorphism
 $\Ocal_S = \pi_*\Ocal_\Acal$ by the Stein factorization (cf.~\cite[p.~348, Theorem~12.68]{Goertz-Wedhorn}). Hence we have $\G_{m,S} = \pi_*\G_{m,\Acal}$.
But since $\ell$ is invertible on $S$, the map $\ell^n: \G_{m,S} \to \G_{m,S}$ is an epimorphism and we get
\[ \pi_*\G_{m,\Acal}/\ell^n = \G_{m,S}/\ell^n = 1.\]
For the last term in the above sequence by~\cite[p.~203, §\,8.1]{BLR} we get the canonical isomorphism $\R^1\pi_*\G_{m,\Acal} = \PPic_{\Acal/S}$
since $\pi$ is smooth and proper.
Note that since $\Acal \to S$ is projective and flat with integral fibres, the Picard scheme exists by~\cite[p.~263, Theorem~9.4.8]{FGAExplained}. 
Let $\NS_{\Acal/S}$ be defined by the short exact sequence of étale sheaves:
\[
    0 \to \PPic^0_{\Acal/S} \to \PPic_{\Acal/S} \to \NS_{\Acal/S} \to 0.
\]
Here $\PPic^0_{\Acal/S}$ is the identity component of $\PPic_{\Acal/S}$ and coincides with the dual Abelian scheme $\Acal^t$ by~\cite[p.~234, §\,8.4 Theorem~5]{BLR}. This implies that $\PPic_{\Acal/S}$ is a smooth commutative group scheme over $S$. Taking $\ell^n$-torsion, which is left exact, we get a short exact sequence:
\[
    0 \to \PPic^0_{\Acal/S}[\ell^n] \to \PPic_{\Acal/S}[\ell^n]  \to \NS_{\Acal/S}[\ell^n].
\]
We now prove that $\PPic^0_{\Acal/S}[\ell^n] \to \PPic_{\Acal/S}[\ell^n]$ is an isomorphism by looking at the stalks.
Since the first two groups are étale over $S$, by~\cref{groupschemes-mult}
it suffices to look at the sequence over the geometric points $\bar s$ of $S$ by~\cite[p.~34, Proposition~I.4.4]{MilneEtaleCohomology}. But by~\cite[p.~165, IV\,§\,19 Theorem~3, Corollary~2]{MumfordAbelianVarieties}, the group
$\NS_{\Acal_{\overline{s}}}(\bar s)$ is a finitely generated free Abelian group (since we are over a field) and its torsion part vanishes.
So, all together, we have the isomorphisms:
\[\R^1\pi_*\mu_{\ell^n} = \R^1\pi_*\G_{m,\Acal}[\ell^n] = \PPic_{\Acal/S}[\ell^n] = \PPic^0_{\Acal/S}[\ell^n] = \Acal^t[\ell^n]. \]
By taking the projective limit over all $n$ we then get the claim:
$\R^1\pi_*\Z_\ell(1) = T_\ell\Acal^t$.

The statement on the weight follows from~\cref{lemma:WeightOfTateTwist} and~\cref{thm:YogaOfWeights}:  $\Z_\ell(1)$ has weight $-2$ and $1-2 = -1$.
\end{proof}

\begin{lemma} \label[lemma]{lemma:isogeny induces iso on Tate modules}
Let $f: A \to B$ be an isogeny (not necessarily étale) of Abelian varieties over a field $k$ and $\ell \neq \Char{k}$.  Then $f$ induces an Galois equivariant isomorphism $V_\ell A \isoto V_\ell B$ of rational Tate modules.
\end{lemma}
\begin{proof}
There is an exact sequence of $\ell$-divisible groups
\[
    0 \to \ker(f)[\ell^\infty] \to A[\ell^\infty] \to B[\ell^\infty] \to 0
\]
with $A[\ell^\infty]$ and $B[\ell^\infty]$ étale since $\ell$ is invertible in $k$ and $\ker(f)[\ell^\infty]$ a finite étale group scheme.  Since for an Abelian group $M$, one has $T_\ell M = \Hom(\Q_\ell/\Z_\ell, M)$ by~\cref{lemma:PropertiesOfTateModule}\,\ref{lemma:HomUndTateModul}, applying $\Hom(\Q_\ell/\Z_\ell, -)$ to the above exact sequence yields an exact sequence
\[
    0 \to T_\ell \ker(f) \to T_\ell A \to T_\ell B \to \Ext^1(\Q_\ell/\Z_\ell, \ker(f)[\ell^\infty]).
\]
Since $\ker(f)$ is a finite group scheme, we have $T_\ell \ker(f) = 0$ by~\cref{lemma:PropertiesOfTateModule}\,\ref{lemma:Tate-module-of-finite-group-is-trivial}.  Since $T_\ell A$ and $T_\ell B$ have the same rank as $f$ is an isogeny (or since $\Ext^1(\Q_\ell/\Z_\ell, \ker(f)[\ell^\infty])$ is finite), tensoring with $\Q_\ell$ yields the desired isomorphism.
\end{proof}

\begin{corollary} \label[corollary]{prop:TateModulesIsogeny}
Let $f: \Acal \to \Bcal$ be an isogeny (not necessarily étale) of Abelian schemes over $S$ and $\ell$ invertible on $S$.  Then $f$ induces an isomorphism $V_\ell\Acal \isoto V_\ell\Bcal$ of $\ell$-adic sheaves.
\end{corollary}
\begin{proof}
We check the isomorphism $V_\ell\Acal \to V_\ell\Bcal$ on stalks.  Let $\pi: \Acal^t \to S$ and $\pi': \Bcal^t \to S$ be the structure morphisms of the dual Abelian schemes and $\pi_x,\pi_x'$ the base changes of $\pi,\pi'$ by $\{x\} \to S$.  By~\cref{prop:RiundTateModul}, we have $V_\ell\Acal = \R^1\pi_*\Q_\ell(1)$ and $V_\ell\Bcal = \R^1\pi'_*\Q_\ell(1)$.  Since $\pi$ and $\pi'$ are proper, by proper base change~\cite[p.~224, Corollary~VI.2.5]{MilneEtaleCohomology}, ($\Z_\ell(1)$ is an inverse limit of the torsion sheaves $\mu_{\ell^n}$), $(V_\ell\Acal)_x = \R^1\pi_*\Q_\ell(1)_x = \R^1\pi_{x,*}\Q_\ell(1) = V_\ell(\Acal_x)$ and analogously for $\Bcal$.  So one can assume $S$ is the spectrum of a field. Then the statement is just~\cref{lemma:isogeny induces iso on Tate modules}.
\end{proof}

\begin{lemma} \label[lemma]{lemma:TateTwistherausziehen}
Let $\pi: X \to Y$ be a morphism of schemes and $\Fcal$ an $\ell$-adic sheaf on $X$.  Then $\R^i\pi_*(\Fcal(n)) = (\R^i\pi_*\Fcal)(n)$.
\end{lemma}
\begin{proof}
We have
\begin{align*}
    \R^i\pi_*(\Fcal(n)) &= \R^i\pi_*(\Fcal \otimes \Z_\ell(n))\\
                        &= \R^i\pi_*(\Fcal \otimes \pi^*\Z_\ell(n)) \quad\text{since $\pi^*\mu_{\ell^n} = \mu_{\ell^n}$}\\
                        &= \R^i\pi_*(\Fcal) \otimes \Z_\ell(n) \quad\text{by the projection formula~\cite{MilneEtaleCohomology}, p.~260, Lemma~VI.8.8 since $\Z_\ell(n)$ is flat}\\
                        &= (\R^i\pi_*\Fcal)(n). \qedhere
\end{align*}
\end{proof}

\begin{definition}
	Let $S$ be an arbitrary base scheme and $\Acal/S$ be an Abelian scheme.  A \defn{polarisation} of $\Acal/S$ is an $S$-group scheme homomorphism $\lambda: \Acal \to \Acal^t$ such that for all $s \in S$, the induced homomorphism $\lambda_{\overline{s}}: \Acal_{\overline{s}} \to \Acal^t_{\overline{s}}$ on geometric fibres is a polarisation in the classical sense, i.\,e.\ it is of the form $a \mapsto t_a^*\Lcal \otimes \Lcal^{-1}$ for $\Lcal \in \Pic(A_{\overline{s}})$ ample.
	
	A polarisation is called \defn{principal} if it is an isomorphism.
\end{definition}

\begin{remark}
See~\cite[p.~126, §\,13]{MilneAbelianVarieties} for the definition of a polarisation for Abelian varieties and~\cite[p.~120, Definition~6.3]{MumfordGIT} for the definition of a polarisation over a general base scheme.
	
Since a polarisation is fibrewise an isogeny, it is globally an isogeny in the sense of~\cref{def:isogeny}.
\end{remark}

\begin{proposition} \label[proposition]{prop:AbelianSchemeanddualisogenous}
Let $X$ be a normal Noetherian integral scheme and $\Acal/X$ an Abelian scheme.  Then there is a polarisation $\Acal \to \Acal^t$.
\end{proposition}
\begin{proof}
Since being an isogeny is defined fibrewise, we have to show that there exists a relatively ample line bundle for $\Acal/X$ since ample line bundles induce polarisations (see~\cite[p.~126, §\,13]{MilneAbelianVarieties}).  This follows from~\cite[p.~170, Théorème~XI.1.13]{RaynaudSLN} and by property~(A) in~\cite[p.~159, Definition~XI.1.2]{RaynaudSLN} and by the existence of an ample line bundle on the generic fibre~\cite[p.~114, Corollary~7.2]{MilneAbelianVarieties}.
\end{proof}

\begin{remark}
Note that
\begin{align*}
    P_i(\Acal/X,q^{-s}) &= \det(1 - q^{-s}\Frob_q^{-1} \mid \H^i(\overline{X}, \R^1\pi_*\Q_\ell)) \\
    &= \det(1 - q^{-s}\Frob_q^{-1} \mid \H^i(\overline{X}, V_\ell(\Acal^t)(-1))) \quad\text{by~\cref{prop:RiundTateModul}}\\
                        &= \det(1 - q^{-s}\Frob_q^{-1} \mid \H^i(\overline{X}, V_\ell(\Acal^t))(-1)) \quad\text{by~\cref{lemma:TateTwistherausziehen}}\\
                        &= \det(1 - q^{-s}q\Frob_q^{-1} \mid \H^i(\overline{X}, V_\ell\Acal^t))\\
                        &= \det(1 - q^{-s}q\Frob_q^{-1} \mid \H^i(\overline{X}, V_\ell\Acal)) \quad\text{by~\cref{prop:TateModulesIsogeny} and~\cref{prop:AbelianSchemeanddualisogenous}}\\
                        &= L_i(\Acal/X, q^{-s+1}),
\end{align*}
so the vanishing order of $P_i(\Acal/X,q^{-s})$ at $s=1$ is equal to the vanishing order of $L_i(\Acal/X,t)$ at $t = q^{-1+1} = q^0 = 1$, and the respective leading coefficients agree.
\end{remark}

The following is a generalisation of~\cite[p.~134--138]{SchneiderZeta} and~\cite[p.~496--498]{Schneider}.

\begin{lemma} \label[lemma]{lemma:BarsottiTateIsEllAdicSheaf}
Let $(G_n)_{n \in \N}$ be a Barsotti-Tate group consisting of finite \emph{étale} group schemes.  Then it is an $\ell$-adic sheaf.
\end{lemma}
\begin{proof}
By~\cite[p.~161, (2)]{Tate-p-divisible},
\[
    0 \to \ker[\ell] \to G_{n+1} \stackrel{[\ell]}{\to}G_n \to 0
\]
is exact.  But $\ker[\ell] = \ell^nG_{n+1}$.  Furthermore, $G_n = 0$ for $n < 0$ and $\ell^{n+1}G_n = 0$ by~\cite[p.~161, (ii)]{Tate-p-divisible}.  Finally, the $G_n$ are constructible since they are finite \emph{étale} group schemes.
\end{proof}

\begin{corollary} \label[corollary]{cor:TellAladicsheaf}
For $\ell$ invertible on $X$, $T_\ell\Acal = (\Acal[\ell^n])_{n \in \N}$ is an $\ell$-adic sheaf.
\end{corollary}
\begin{proof}
This follows from~\cref{lemma:BarsottiTateIsEllAdicSheaf} since $\ell$ is invertible on $X$, so $\Acal[\ell^n]/X$ is finite étale by~\cite[p.~147, Proposition~20.7]{MilneAbelianVarieties}.
\end{proof}

\begin{theorem} \label[theorem]{thm:WeilPairing}
	Let $f: \Acal \to \Acal'$ be an $X$-isogeny of Abelian schemes with dual isogeny $f^t: \Acal'^t \to \Acal^t$.  The \defn{Weil pairing}
	\[
	\qu{\cdot,\cdot}_f: \ker(f) \times_X \ker(f^t) \to \G_m
	\]
	is a non-degenerate and biadditive pairing of finite flat $X$-group schemes, i.\,e.\ it defines a canonical $X$-isomorphism
	\[
	\ker(f^t) \isoto (\ker(f))^t.
	\]
	Moreover, it is functorial in $f$.
	
	If $X = \Spec{k}$, it induces a perfect pairing of torsion-free finitely generated $\Z_\ell[\Gamma]$-modules
	\begin{align}
	T_\ell\Acal \times T_\ell(\Acal^t)   &\to \Z_\ell(1) \label{eq:WeilPairing}
	\end{align}
	and this a canonical isomorphism of $\Z_\ell[\Gamma]$-modules
	\begin{align}
		\Hom_{\Z_\ell}(T_\ell\Acal, \Z_\ell)              &= T_\ell(\Acal^t)(-1). \label{eq:HomTellAZell}
	\end{align}	
\end{theorem}
\begin{proof}
	See~\cite[p.~186]{MumfordAbelianVarieties} (for Abelian varieties) and~\cite[p.~66\,f., Theorem~1.1]{OdaDeRham} (for Abelian schemes).  Note that it is \emph{not} assumed that $f$ is étale.
\end{proof}

\subsection{Étale cohomology of varieties over finite fields}

\begin{lemma} \label[lemma]{lemma:isogenous Abelian varieties}
	Let $k$ be a finite field.  Then $k$-isogenous Abelian varieties have the same number of $k$-rational points.
\end{lemma}
\begin{proof}
	Let $f: A\to B$ be a $k$-isogeny.  Note that the finite field $k$ is perfect.  Take Galois invariants of
	\[
	0 \to (\ker f)(\overline{k}) \to A(\overline{k}) \to B(\overline{k}) \to 0
	\]
	and using Lang-Steinberg~\cite[p.~205, Theorem~3]{MumfordAbelianVarieties} in the form $\H^1(k,A) = 0 = \H^1(k,B)$ and the Herbrand quotient $h((\ker f)(\overline{k})) = 1$ (since $(\ker f)(\overline{k})$ is finite) yields $\card{A(k)} = \card{B(k)}$.
	
	(Alternatively, use that $A(\F_{q^n}) = \ker(1-\Frob_q^n)$ and $f(1-\Frob_q^n) = (1-\Frob_q^n)f$ and $\deg{f} \neq 0$ is finite, take degrees and cancel $\deg{f}$.)
\end{proof}

\begin{remark}
	For the much harder converse:  By~\cite[p.~139, Theorem~1\,(c1)$\iff$(c4)]{TateEndomorphisms}, two Abelian varieties over a finite field $k$ are $k$-isogenous \emph{iff} they have the same number of $k'$-rational points for every finite extension $k'$ of $k$.  For the question how many $k'$ suffice, see~\cite{282315}.
\end{remark}

\begin{theorem} \label[theorem]{thm:ProperConstructibleFiniteCohomology}
Let $X$ be a proper scheme over a separably closed or finite field $K$ and $\Fcal$ be a constructible étale sheaf on $X$.  Then $\H^q(X,\Fcal)$ is finite for all $q \geq 0$.
\end{theorem}
Note that~\cite[p.~224, Corollary~VI.2.8]{MilneEtaleCohomology} does not hold in general (consider $X = \Spec\Q$ with $\H^1(\Spec\Q,\mu_n) = \Q^\times/n$)!
\begin{proof}
By the proper base change theorem~\cite[p.~223, Theorem~VI.2.1]{MilneEtaleCohomology}, the claim follows for separably closed fields. For a finite field $K$ with absolute Galois group $\Gamma$, the claim follows by passing to a separable closure $\overline K$ of $K$ and the usage of Hochschild-Serre spectral sequence $\H^p(\Gamma,\H^q(\overline{X},\Fcal)) \Rightarrow \H^{p+q}(X,\Fcal)$ with $\overline{X} := X \times_K \overline{K}$, which degenerates by~\cite[p.~124, Exercise~5.2.1]{Weibel} because of $\mathrm{cd}(\Gamma) = 1$ by~\cite[p.~69, (1.6.13)\,(ii)]{CohomologyOfNumberFields} as $\Gamma = \hat{\Z}$ into short exact sequences
\[
    0 \to \H^{i-1}(\overline{X},\Fcal)_\Gamma \to \H^i(X,\Fcal) \to \H^i(\overline{X},\Fcal)^\Gamma \to 0
\]
with the outer groups being finite by the case of a separably closed ground field.
\end{proof}

\begin{lemma} \label[lemma]{lemma:HochschildSerreSS}
Let $X$ be a variety over a finite field $k$ with absolute Galois group $\Gamma$.  Let $(\Fcal_n)_{n\in\N}$, $\Fcal = \varprojlim_n\Fcal_n$ be an $\ell$-adic sheaf.  For every $i$, there is a short exact sequence
\begin{align}
    0 \to \H^{i-1}(\overline{X}, \Fcal)_\Gamma \to \H^i(X, \Fcal) \to \H^i(\overline{X}, \Fcal)^\Gamma \to 0 \label{eq:sesHS}
\end{align}
with $\H^i(\overline{X}, \Fcal)$ and $\H^i(X, \Fcal)$ finitely generated $\Z_\ell$-modules.
\end{lemma}
The following argument is a generalisation of~\cite[p.~78, Lemma~3.4]{Milne-1988b}.
\begin{proof}
Since $\Gamma = \hat{\Z}$ has cohomological dimension $1$ by~\cite[p.~69, (1.6.13)\,(ii)]{CohomologyOfNumberFields}, we get from the Hochschild-Serre spectral sequence for $\overline{X}/X$ (see~\cite[p.~106, Remark~III.2.21\,(b)]{MilneEtaleCohomology}) by~\cite[p.~124, Exercise~5.2.1]{Weibel} short exact sequences for every $n$ and $i$
\[
    0 \to \H^{i-1}(\overline{X}, \Fcal_n)_\Gamma \to \H^i(X, \Fcal_n) \to \H^i(\overline{X}, \Fcal_n)^\Gamma \to 0.
\]
Since all involved groups are finite (because the two outer groups are finite by~\cref{thm:ProperConstructibleFiniteCohomology} since $\overline{X}/\overline{k}$ is proper over $\overline{k}$ separably closed and $\Fcal_n$ is constructible by definition of an $\ell$-adic sheaf), the system satisfies the Mittag-Leffler condition, so taking the projective limit yields an exact sequence
\[
    0 \to \varprojlim_n(\H^{i-1}(\overline{X}, \Fcal_n)_\Gamma) \to \H^i(X,\Fcal) \to \varprojlim_n(\H^i(\overline{X}, \Fcal_n)^\Gamma) \to 0.
\]
Write $M_{(n)}$ for $\H^i(\overline{X}, \Fcal_n)$.  Breaking the exact sequence
\[
    0 \to M_{(n)}^\Gamma \to M_{(n)} \stackrel{\Frob-1}{\longrightarrow} M_{(n)} \to (M_{(n)})_\Gamma \to 0
\]
into two short exact sequences and applying $\varprojlim_n$, one obtains, setting $Q_{(n)} = (\Frob-1)M_{(n)}$, exact sequences
\begin{align}
    0 \to \varprojlim_nM_{(n)}^\Gamma \to \varprojlim_nM_{(n)} \stackrel{\Frob-1}{\longrightarrow} \varprojlim_nQ_{(n)} \to \varprojlim_n{}^1M_{(n)}^\Gamma\label{eq:Mn} \\
    0 \to \varprojlim_nQ_{(n)}        \to \varprojlim_nM_{(n)} \to \varprojlim_n(M_{(n)})_\Gamma \to \varprojlim_n{}^1Q_{(n)}.  \label{eq:Qn}
\end{align}

Since the $M_{(n)}$, and hence the $M_{(n)}^\Gamma$ are finite (argument as above), they form a Mittag-Leffler system, and hence one gets from~\eqref{eq:Mn} exact sequences
\[
    0 \to \varprojlim_nM_{(n)}^\Gamma \to \varprojlim_nM_{(n)} \stackrel{\Frob-1}{\longrightarrow} \varprojlim_nQ_{(n)} \to 0.
\]
Similarly, the $Q_{(n)} \subseteq M_{(n)}$ are finite, and hence
\[
    0 \to \varprojlim_nQ_{(n)}        \to \varprojlim_nM_{(n)} \to \varprojlim_n(M_{(n)})_\Gamma \to 0
\]
is exact from~\eqref{eq:Qn}.  Combining the above two short exact sequences, one gets the exactness of
\[
    0 \to \varprojlim_nM_{(n)}^\Gamma \to \varprojlim_nM_{(n)} \stackrel{\Frob-1}{\longrightarrow} \varprojlim_nM_{(n)} \to \varprojlim_n(M_{(n)})_\Gamma \to 0,
\]
which shows that for all $i$
\begin{align*}
    \varprojlim_n(\H^{i}(\overline{X}, \Fcal_n)^\Gamma) = \varprojlim_nM_{(n)}^\Gamma = \ker\Big(\varprojlim_nM_{(n)} \stackrel{\Frob-1}{\longrightarrow} \varprojlim_nM_{(n)}\Big) = \H^i(\overline{X}, \Fcal)^\Gamma \\
    \varprojlim_n(\H^{i}(\overline{X}, \Fcal_n)_\Gamma) = \varprojlim_n(M_{(n)})_\Gamma = \coker\Big(\varprojlim_nM_{(n)} \stackrel{\Frob-1}{\longrightarrow} \varprojlim_nM_{(n)}\Big) = \H^{i}(\overline{X}, \Fcal)_\Gamma,
\end{align*}
which is what we wanted.
\end{proof}

\cref{lemma:HochschildSerreSS} implies
\begin{align} \label{eq:H2dH2dplus1}
    \H^{2d}(\overline{X}, T_\ell\Acal)_\Gamma &\isoto \H^{2d+1}(X, T_\ell\Acal)
\end{align}
since $\H^{2d+1}(\overline{X}, T_\ell\Acal) = 0$ by~\cite[p.~221, Theorem~VI.1.1]{MilneEtaleCohomology} as $\dim{\overline{X}} = d$.  Because of $\H^i(\overline{X}, T_\ell\Acal) = 0$ for $i > 2d$ for the same reason, it follows from~\cref{lemma:HochschildSerreSS} that $\H^i(X, T_\ell\Acal) = 0$ for $i > 2d + 1$.  Furthermore, one has
\begin{align} \label{eq:ZellH2dbarX}
    \Z_\ell = (\Z_\ell)_\Gamma = \H^{2d}(\overline{X}, \Z_\ell(d))_\Gamma \isoto \H^{2d+1}(X, \Z_\ell(d)),
\end{align}
the second equality by Poincaré duality~\cite[p.~276, Theorem~VI.11.1\,(a)]{MilneEtaleCohomology} and the isomorphism by~\cref{lemma:HochschildSerreSS} since $\H^{2d+1}(\overline{X}, \Z_\ell(d))^\Gamma = 0$ by~\cite[p.~221, Theorem~VI.1.1]{MilneEtaleCohomology} as $\dim{\overline{X}} = d$.

\begin{definition}
Let $f: A \to B$ be a homomorphism of Abelian groups.  If $\ker(f)$ and $\coker(f)$ are finite, $f$ is called an \defn{isomorphism up to finite groups}, in which case we define
\[
    q(f) = \frac{\card{\coker(f)}}{\card{\ker(f)}}.
\]
\end{definition}

\begin{remark}
	An isomorphism up to finite groups is called \emph{quasi-isomorphism} in~\cite[p.~433]{Artin-Tate}, but we avoid this term because one may confuse it with a quasi-isomorphism of complexes.
\end{remark}

The following lemma is crucial for relating special values of $L$-functions and orders of cohomology groups.

\begin{lemma} \label[lemma]{lemma:BayerNeukirch}
Let $\Frob$ be a topological generator of $\Gamma$ and $M$ be a finitely generated $\Z_\ell$-module with continuous $\Gamma$-action.  Then the following are equivalent:

1. $\det(1-\Frob \mid M \otimes_{\Z_\ell} \Q_\ell) \neq 0$.

2. $\H^0(\Gamma, M) = M^\Gamma$ is finite.

3. $\H^1(\Gamma, M)$ is finite.

If one of these holds, we have $\H^1(\Gamma, M) = M_\Gamma$ and
\[
    |\det(1-\Frob \mid M \otimes_{\Z_\ell} \Q_\ell)|_\ell = \frac{\card{\H^0(\Gamma, M)}}{\card{\H^1(\Gamma, M)}} = \frac{|M^\Gamma|}{|M_\Gamma|} = \frac{\card{\ker{(1-\Frob)}}}{\card{\coker{(1-\Frob)}}} = q(1-\Frob)^{-1}.
\]
\end{lemma}
\begin{proof}
See~\cite[p.~42, Lemma~(3.2)]{BayerNeukirch}.  If $M$ is torsion, $\H^1(\Gamma, M) = M_\Gamma$ by~\cite[p.~69, (1.6.13)~Proposition\,(i)]{CohomologyOfNumberFields}. 
\end{proof}

\begin{corollary} \label[corollary]{cor:weightnotequalzerofinite}
Let $\Frob$ be a topological generator of $\Gamma$ and $M$ be a finitely generated $\Z_\ell$-module with continuous $\Gamma$-action.  If $M \otimes_{\Z_\ell} \Q_\ell$ has weight $\neq 0$, $M^\Gamma$ and $M_\Gamma$ are finite.
\end{corollary}
\begin{proof}
Since $M \otimes_{\Z_\ell} \Q_\ell$ has weight $\neq 0$, $\Frob$ has all eigenvalues $\neq 1 = q^{0/2}$, hence $\det(1-\Frob \mid M \otimes_{\Z_\ell} \Q_\ell)$ is $\neq 0$, so the corollary follows from~\cref{lemma:BayerNeukirch}.
\end{proof}

\subsection{$L$-functions of Abelian schemes}

\begin{definition} \label[definition]{def:L-series of an Abelian variety}
	Let $\pi: \Acal \to X$ be an Abelian scheme.  Then for $\ell$ invertible on $X$ let
	\begin{align*}
	\mathcal{L}(\Acal/X, s)   &= \prod_{x \in |X|}\det\big(1 - q^{-s\deg(x)}\Frob_{x}^{-1} \mid (\R^1\pi_*\Q_\ell)_{\overline{x}}\big)^{-1}
	\end{align*}
	as a power series in $q^{-s}$ with coefficients in $\Q_\ell$.  Here, $\Frob_x^{-1}$ is the geometric Frobenius of the finite field $k(x)$.
\end{definition}

\begin{theorem} \label[theorem]{prop:Grothendieck trace formula}
	Let $\pi: \Acal \to X$ be an Abelian scheme of (relative) dimension $d$.  One has
	\begin{align*}
	\mathcal{L}(\Acal/X, s)   = \prod_{i=0}^{2d}\det\big(1 - q^{-s}\Frob_q^{-1} \mid \H^i(\overline{X}, \R^1\pi_*\Q_\ell)\big)^{(-1)^{i+1}},
	\end{align*}
	where the Frobenius acts via functoriality on the second factor of $\overline{X} = X \times_k \overline{k}$.
\end{theorem}
\begin{proof}
	This follows from the Grothendieck-Lefschetz trace formula~\cite[p.~7, Theorem~1.1]{KiehlWeissauer}.
\end{proof}

\begin{corollary} \label[corollary]{cor:indepencenceofell}
	The power series $\mathcal{L}(\Acal/X, s)$ is a rational function in $q^{-s}$ with coefficients in $\Q$ independent of $\ell \neq p$.
	
	The factors in~\cref{prop:Grothendieck trace formula} for different $i$ are polynomials with coefficients in $\Q$ independent of $\ell \neq p$.
\end{corollary}
\begin{proof}
	The right hand side in~\cref{prop:Grothendieck trace formula} is a polynomial in $q^{-s}$ with coefficients in $\Q_\ell$.  These are contained in $\Q$ and independent of $\ell$:   Using~\cref{def:L-series of an Abelian variety}, by~\cref{prop:RiundTateModul}, \cref{thm:Tate module and etale cohomology} and $(\R^1\pi_*\Q_\ell)_{\overline{x}} = \H^1(\Acal_{\overline{x}}, \Q_\ell)$ by proper base change~\cite[p.~224, Corollary~VI.2.5]{MilneEtaleCohomology}
	\begin{align*}
	\det(1 - t\Frob_x^{-1} \mid \H^1(\Acal_{\overline{x}},\Q_\ell)) &= \det(1 - t\Frob_x^{-1} \mid (V_\ell\Acal_{\overline{x}})^\vee)
	\end{align*}
	is in $\Q[t]$ independent of $\ell \neq p$ since this is true for the $\ell$-adic Tate module by~\cite[p.~167, Theorem~4]{MumfordAbelianVarieties}.
	
	By the yoga of weights, the characteristic polynomials in~\cref{prop:Grothendieck trace formula} for different $i$ do not cancel since their roots have different absolute values for all complex embeddings.  Since their alternating product is in $\Q(t)$ independent of $\ell \neq p$, this holds for all factors individually.
\end{proof}

\begin{definition}  \label[definition]{def:LFunctionSchneider}
	For an Abelian scheme $\pi: \Acal \to X$ let
	\[
	P_i(\Acal/X,t) = \det\big(1 - t\Frob_q^{-1} \mid \H^i(\overline{X}, \R^1\pi_*\Q_\ell)\big).
	\]
	for $\ell$ invertible on $X$ and define the \defn{relative $L$-function} of an Abelian scheme $\Acal/X$ by
	\[
	L(\Acal/X,s) = \frac{P_1(\Acal/X,q^{-s})}{P_0(\Acal/X,q^{-s})}.
	\]
\end{definition}

For our purposes, it is better to consider the following $L$-function:

\begin{definition} \label[definition]{def:Li}
	Let
	\[
	L_i(\Acal/X,t) = \det(1 - t\Frob_q^{-1} \mid \H^i(\overline{X}, V_\ell\Acal))
	\]
	for $\ell$ invertible on $X$.
\end{definition}

\begin{remark} \label[remark]{rem:DefinitionOfLFunction}
	This definition is motivated in~\cref{rem:MotivationForLfunction} below.  It is explained there why we omit the cohomology in degrees $> 1$ in contrast to~\cref{prop:Grothendieck trace formula} coming from the usual~\cref{def:L-series of an Abelian variety}.
	
	The $P_i(\Acal/X,t)$ are polynomials with rational coefficients independent of $\ell$ by~\cref{cor:indepencenceofell}.  Using~\cref{prop:RiundTateModul}, the proof of~\cref{cor:indepencenceofell} also shows this for the $L_i(\Acal/X,t)$.
\end{remark}

\subsection{The cohomological formula for the special $L$-value $L^*(\Acal/X,1)$}

\begin{corollary} \label[corollary]{cor:Liell}
If $i \neq 1$, $\H^i(\overline{X}, T_\ell\Acal)^\Gamma$ and $\H^i(\overline{X}, T_\ell\Acal)_\Gamma$ are finite, and one has
\[
    |L_i(\Acal/X,1)|_\ell = \frac{\card{\H^i(\overline{X}, T_\ell\Acal)^\Gamma}}{\card{\H^i(\overline{X}, T_\ell\Acal)_\Gamma}}.
\]
\end{corollary}
\begin{proof}
This follows from~\cref{lemma:BayerNeukirch}\,2 and~\cref{lemma:WeightNotEqualZeroTrivial} since $\H^i(\overline{X}, V_\ell\Acal)$ has weight $i - 1$ by~\cref{thm:YogaOfWeights} and~\cref{thm:Tate module and etale cohomology}, which is $\neq 0$ if $i \neq 1$.
\end{proof}

After~\cref{cor:Liell}, one can concentrate on $i = 1$.

\begin{lemma} \label[lemma]{lemma:alphafandbeta}
Infinite groups in the short exact sequences in~\eqref{eq:sesHS} can only occur in the following two sequences:
\begin{equation}
\begin{gathered}\begin{tikzcd}
	0 \arrow[r] &\H^1(\overline{X}, T_\ell\Acal)_\Gamma \arrow[r,"\beta"] &\H^2(X, T_\ell\Acal) \arrow[r] &\H^2(\overline{X}, T_\ell\Acal)^\Gamma \arrow[r] &0\\
	0 \arrow[r] &\H^0(\overline{X}, T_\ell\Acal)_\Gamma \arrow[r] &\H^1(X, T_\ell\Acal) \arrow[r,"\alpha"] &\H^1(\overline{X}, T_\ell\Acal)^\Gamma \arrow[llu,"f",dashrightarrow] \arrow[r] &0
\end{tikzcd}\label{eq:sesHS1}\end{gathered}\end{equation}
Here, $f$ is induced by the identity on $\H^1(\overline{X}, T_\ell\Acal)$.  The morphisms $\alpha$ and $\beta$ are isomorphisms up to finite groups, and $\alpha$ is surjective and $\beta$ is injective.
\end{lemma}
\begin{proof}
Since $T_\ell\Acal$ has weight $-1$ by~\cref{prop:RiundTateModul}, $\H^i(\overline{X}, T_\ell\Acal)$ has weight $i-1$ by~\cref{thm:YogaOfWeights}.  So the conditions of~\cref{lemma:BayerNeukirch} are fulfilled for the $\Gamma$-module $M = \H^i(\overline{X},T_\ell\Acal)$ and $i \neq 1$.  Therefore infinite groups in the short exact sequences in~\eqref{eq:sesHS} can only occur in the two sequences of diagram~\eqref{eq:sesHS1}.  Since $\H^2(\overline{X}, T_\ell\Acal)^\Gamma$ and $\H^0(\overline{X}, T_\ell\Acal)_\Gamma$ are finite (having weight $2-1 \neq 0$ and $0-1 \neq 0$, so~\cref{cor:weightnotequalzerofinite} applies), $\alpha$ and $\beta$ are isomorphisms up to finite groups, and $\alpha$ is surjective and $\beta$ is injective.
\end{proof}

Recall from~\cref{def:Li} that
\[
    L_1(\Acal/X,t) = \det(1-t\Frob_q^{-1} \mid \H^1(\overline{X}, V_\ell\Acal)).
\]
Define $\tilde{L}_1(\Acal/X,t)$ and the \defn{analytic rank} $\rho$ by
\begin{align}
    \rho &= \ord_{t=1}L_1(\Acal/X,t) \in \N, \label{eq:rho} \\
    L_1(\Acal/X,t) &= (t-1)^\rho \cdot \tilde{L}_1(\Acal/X,t). \label{eq:L1tilde}
\end{align}
Note that $\tilde{L}_1(\Acal/X,1) \neq 0$ and $\tilde{L}_1(\Acal/X,t) \in \Q_\ell[t]$.

The idea is that for infinite cohomology groups $\H^1(\overline{X}, T_\ell\Acal)$, one should insert a regulator term $q(f)$ or $q((\beta f \alpha)_\Tors)$ with $(\beta f \alpha)_\Tors$ induced by $\beta f \alpha$ by modding out torsion.
\begin{lemma} \label[lemma]{lemma:L1tilda}
Let $\rho$ be as in~\eqref{eq:rho}.  One always has $\rho \geq \rk_{\Z_\ell}\H^1(X, T_\ell\Acal)$ with equality iff $f$ in~\eqref{eq:sesHS1} is an isomorphism up to finite groups.  In this case,
\begin{align*}
    |\tilde{L}_1(\Acal/X,1)|_\ell^{-1} &= q(f) = \frac{|\coker{f}|}{|\ker{f}|} \quad\text{and} \\
    |\tilde{L}_1(\Acal/X,1)|_\ell^{-1} &= q((\beta f \alpha)_\Tors) \cdot \frac{\card{\H^0(\overline{X},T_\ell\Acal)_\Gamma}}{\card{\H^2(\overline{X},T_\ell\Acal)^\Gamma}} \cdot \frac{\card{\H^2(X,T_\ell\Acal)_\tors}}{\card{\H^1(X,T_\ell\Acal)_\tors}}
\end{align*}
with $\tilde{L}_1(\Acal/X,t)$ from~\eqref{eq:L1tilde}.
\end{lemma}
\begin{proof}
By writing $\Frob_q^{-1}$ in Jordan normal form, one sees that $\rho$ is equal to
\[
    \dim_{\Q_\ell}\bigcup_{n^\geq1}\ker(1-\Frob_q^{-1})^n \geq \dim_{\Q_\ell}\ker(1-\Frob_q^{-1}) = \dim_{\Q_\ell}\H^1(\overline{X}, V_\ell\Acal)^\Gamma,
\]
i.\,e.\
\[
    \rho \geq \dim_{\Q_\ell}\H^1(\overline{X}, V_\ell\Acal)^\Gamma,
\]
 and that equality holds iff the operation of the Frobenius on $\H^1(\overline{X}, V_\ell\Acal)$ is semi-simple at $1$, i.\,e.\
\[
    \dim_{\Q_\ell}\bigcup_{n^\geq1}\ker(1-\Frob_q^{-1})^n = \dim_{\Q_\ell}\ker(1-\Frob_q^{-1}),
\]
i.\,e.\ the generalised eigenspace at $1$ equals the eigenspace, which is equivalent to $f_{\Q_\ell}$ in~\eqref{eq:sesHS1} being an isomorphism, i.\,e.\ $f$ being an isomorphism up to finite groups.

From~\eqref{eq:sesHS1}, since $\H^0(\overline{X}, T_\ell\Acal)_\Gamma$ is finite, one sees that
\[
    \dim_{\Q_\ell}\H^1(\overline{X}, V_\ell\Acal)^\Gamma = \rk_{\Z_\ell}\H^1(\overline{X}, T_\ell\Acal)^\Gamma = \rk_{\Z_\ell}\H^1(X, T_\ell\Acal).
\]
Hence, the inequality $\rho \geq \rk_{\Z_\ell}\H^1(X, T_\ell\Acal)$ and the first statement follows.

Assuming $f$ being an isomorphism up to finite groups, one has by~\cref{lemma:BayerNeukirch} and arguing as in~\cite[p.~136, proof of Lemma~3]{SchneiderZeta}
\begin{align*}
    |\tilde{L}_1(\Acal/X,1)|        &= \frac{|[(\Frob_q-1)\H^1(\overline{X}, T_\ell\Acal)]^\Gamma|}{|[(\Frob_q-1)\H^1(\overline{X}, T_\ell\Acal)]_\Gamma|} \\
                                    &= \frac{|[(\Frob_q-1)\H^1(\overline{X}, T_\ell\Acal)]^\Gamma|}{|(\Frob_q-1)\H^1(\overline{X}, T_\ell\Acal):(\Frob_q-1)^2\H^1(\overline{X}, T_\ell\Acal)|} \\
                                    &= \frac{|\ker{f}|}{|\coker{f}|} = q(f)^{-1}.
\end{align*}
For the second equation,
\begin{align*}
    q(f) &= \frac{q((\beta f)_\tors)}{q(\beta)} \cdot q((\beta f)_\Tors) \quad\text{by~\cite[p.~306-19--306-20, Lemma~z.1--z.4]{Artin-Tate}} \\
         &= \frac{1}{\card{\H^2(\overline{X},T_\ell\Acal)^\Gamma}} \cdot \frac{\card{(\H^2(X,T_\ell\Acal)_\Gamma)_\tors}}{\card{(\H^1(\overline{X},T_\ell\Acal)^\Gamma)_\tors}} \cdot q((\beta f)_\Tors) \\
         &= q((\beta f \alpha)_\Tors) \cdot \frac{\card{\H^0(\overline{X},T_\ell\Acal)_\Gamma}}{\card{\H^2(\overline{X},T_\ell\Acal)^\Gamma}} \cdot \frac{\card{\H^2(X,T_\ell\Acal)_\tors}}{\card{\H^1(X,T_\ell\Acal)_\tors}} \quad\text{since $\coker(\alpha) = 0$.} \qedhere
\end{align*}
\end{proof}

\begin{lemma} \label[lemma]{lemma:AcalExactSequence}
Let $\ell \neq p$ be invertible on $X$.  Then there is an exact sequence
\[
    0 \to \Acal(X) \otimes \Q_\ell/\Z_\ell \to \H^1(X,\Acal[\ell^\infty]) \to \H^1(X,\Acal)[\ell^\infty] \to 0.
\]
\end{lemma}
\begin{proof}
Since $\ell$ is invertible on $X$, one has the short exact Kummer sequence of étale sheaves $0 \to \Acal[\ell^n] \to \Acal \to \Acal \to 0$, which induces
\begin{equation} \label{eq:AcalSes}
    0 \to \Acal(X)/\ell^n \to \H^1(X,\Acal[\ell^n]) \to \H^1(X,\Acal)[\ell^n] \to 0.
\end{equation}
Passing to the colimit $\varinjlim_n$ yields the result.
\end{proof}

\begin{remark}
	This reminds us of the exact sequence
	\[
		0 \to A(K)/n \to \mathrm{Sel}^{(n)}(A/K) \to \Sha(A/K)[n] \to 0
	\]
	for an Abelian variety $A$ over a global field $K$ with $n$ invertible in $K$.
\end{remark}

Recall our definition of the Tate-Shafarevich group, $\Sha(\Acal/X) = \Het^1(X,\Acal)$, from~\cite[p.~225, Definition~4.2]{KellerSha}.
\begin{lemma} \label[lemma]{lemma:ShaEndlichenEllKorang}
Let $\ell$ be invertible on $X$.  Then the $\Z_\ell$-corank of $\Sha(\Acal/X)[\ell^\infty]$ is finite.
\end{lemma}
\begin{proof}
From~\eqref{eq:AcalSes}, one sees that $\H^1(X,\Acal)[\ell]$ is finite as it is a quotient of $\H^1(X,\Acal[\ell])$ and $\Acal[\ell]/X$ is constructible, and the cohomology of a constructible sheaf on a proper variety over a finite field is finite by~\cref{thm:ProperConstructibleFiniteCohomology}.  Hence $\Sha(\Acal/X)[\ell^\infty]$ is cofinitely generated by~\cref{lemma:PropertiesOfTateModule}\,\ref{lemma:Aellendlich}.
\end{proof}

For an $\ell$-adic sheaf $\Fcal$, denote by $\Fcal(n)$ the $n$-th Tate twist of $\Fcal$, $\Fcal(n) := \Fcal \otimes_{\Z_\ell} \Z_\ell(n)$, see~\cref{def:Tatetwist}.
\begin{lemma} \label[lemma]{lemma:lesGarbensequenz}
Let $X/k$ be proper over $k$ separably closed or finite, and let $\ell$ be invertible on $X$.  There is a long exact sequence 
\[
    \ldots \to \H^i(X,T_\ell\Acal(n)) \to \H^i(X,T_\ell\Acal(n)) \otimes_{\Z_\ell} \Q_\ell \to \H^i(X,\Acal[\ell^\infty](n)) \to \ldots
\]
which induces isomorphisms
\[
    \H^{i-1}(X,\Acal[\ell^\infty](n))_\Div \isoto \H^i(X,T_\ell\Acal(n))_\tors
\]
and short exact sequences
\[
    0 \to \H^i(X,T_\ell\Acal(n))_\Tors \to \H^i(X,T_\ell\Acal(n)) \otimes_{\Z_\ell} \Q_\ell \to \H^i(X,\Acal[\ell^\infty](n))_\divv \to 0.
\]
\end{lemma}
\begin{proof}
Consider for $m,m' \in \N$ invertible on $X$ the short exact sequence of étale sheaves
\[
    0 \to \Acal[m](n) \hookrightarrow \Acal[mm'](n) \stackrel{\cdot m}{\to} \Acal[m'](n) \to 0.
\]
Setting $m = \ell^\mu, m' = \ell^\nu$, the associated long exact sequence is
\[
    \ldots \to \H^i(X,\Acal[\ell^\mu](n)) \to \H^i(X,\Acal[\ell^{\mu+\nu}](n)) \to \H^i(X,\Acal[\ell^\nu](n)) \to \ldots.
\]
Passing to the projective limit $\varprojlim_{\mu}$ and then to the inductive limit $\varinjlim_\nu$ yields the desired long exact sequence since all involved cohomology groups are finite by~\cref{thm:ProperConstructibleFiniteCohomology} since $X/k$ is proper over a separably closed or finite field and our sheaves are constructible.  Here, we use that $\varprojlim$ is exact on finite groups, see~\cite[p.~83, Proposition~3.5.7 and Exercise~3.5.2]{Weibel}.

For the second statement, consider the exact sequence
\[
    \H^{i-1}(X,T_\ell\Acal(n)) \otimes_{\Z_\ell} \Q_\ell \stackrel{f}{\to} \H^{i-1}(X,\Acal[\ell^\infty](n)) \stackrel{d}{\to} \H^i(X,T_\ell\Acal(n)) \stackrel{g}{\to} \H^i(X,T_\ell\Acal(n)) \otimes_{\Z_\ell} \Q_\ell.
\]
Since $\H^i(X,T_\ell\Acal(n))$ is a finitely generated $\Z_\ell$-module (since $(\Acal[\ell^n])_{n\in\N}$ is an $\ell$-adic sheaf) and $g$ is induced by the identity, we have $\ker{g} = \H^i(X,T_\ell\Acal(n))_\tors$; note that $\H^i(X,T_\ell\Acal(n)) \iso \Z_\ell^{\rk} \oplus \H^i(X,T_\ell\Acal(n))_\tors$ and that the codomain of $g$ is isomorphic to $\Q_\ell^{\rk}$.  Since $\H^{i-1}(X,T_\ell\Acal(n))$ is a finitely generated $\Z_\ell$-module and $\H^{i-1}(X,\Acal[\ell^\infty](n))$ is a cofinitely generated $\ell$-torsion module isomorphic to $(\Q_\ell/\Z_\ell)^{\rk} \oplus \H^i(X,T_\ell\Acal(n))_\tors$, so the divisible part is $(\Q_\ell/\Z_\ell)^{\rk}$, we have $\im{f} = \H^{i-1}(X,\Acal[\ell^\infty](n))_\divv$.  The claim follows from the exactness of the sequence.

For the third statement, consider the exact sequence
\[
    \H^i(X,T_\ell\Acal(n)) \stackrel{g}{\to} \H^i(X,T_\ell\Acal(n)) \otimes_{\Z_\ell} \Q_\ell \stackrel{f}{\to} \H^i(X,\Acal[\ell^\infty](n)).
\]
Since $\H^i(X,T_\ell\Acal(n))$ is a finitely generated $\Z_\ell$-module (since $(\Acal[\ell^n](n))_{n\in\N}$ is an $\ell$-adic sheaf) and $g$ is induced by the identity, we have $\ker{g} = \H^i(X,T_\ell\Acal(n))_\tors$; note that $\H^i(X,T_\ell\Acal(n)) \iso \Z_\ell^{\rk} \oplus \H^i(X,T_\ell\Acal(n))_\tors$ and that the codomain of $g$ is isomorphic to $\Q_\ell^{\rk}$.  Since $\H^i(X,T_\ell\Acal(n))$ is a finitely generated $\Z_\ell$-module and $\H^i(X,\Acal[\ell^\infty](n))$ is a cofinitely generated $\ell$-torsion module isomorphic to $(\Q_\ell/\Z_\ell)^{\rk} \oplus \H^i(X,T_\ell\Acal(n))_\tors$, so the divisible part is $(\Q_\ell/\Z_\ell)^{\rk}$, we have $\im{f} = \H^i(X,\Acal[\ell^\infty](n))_\divv$.  The claim follows from the exactness of the sequence.
\end{proof}

\begin{theorem}[Mordell-Weil(-Lang-Néron)] \label[theorem]{thm:MordellWeil}
Let $K$ be a field finitely generated over its prime field and $A/K$ an Abelian variety.  Then the Mordell-Weil group $A(K)$ is a finitely generated Abelian group.
\end{theorem}
\begin{proof}
See~\cite[p.~42, Theorem~2.1]{Conrad-trace}.
\end{proof}

Note that $\Acal(X) = A(K)$ by the Néron mapping property:
\begin{theorem}[Néron mapping property] \label[theorem]{thm:NeronModel}
Let $S$ be a regular, Noetherian, integral, separated scheme with $g: \{\eta\} \hookrightarrow S$ the inclusion of the generic point.  Let $\Acal/S$ be an Abelian scheme.  Then
\[
    \Acal \isoto g_*g^*\Acal
\]
as sheaves on the smooth site $S_{\mathrm{sm}}$ of $S$.
\end{theorem}
\begin{proof}
See~\cite[p.~222, Theorem~3.3]{KellerSha}.
\end{proof}

\begin{lemma} \label[lemma]{lemma:cohTerms}
Assume $\ell$ is invertible on $X$.  Then one has the following identities for the étale cohomology groups of $X$:
\begin{align}
    \H^i(X, T_\ell\Acal) &= 0 \quad\text{for $i \not\in \{1, 2, \ldots, 2d+1$\}} \label{eq:HiVanishing}\\
    \H^1(X, T_\ell\Acal)_\tors &= \H^0(X, \Acal[\ell^\infty])_\Div = {\H^0(X, \Acal)[\ell^\infty]} \label{eq:H1}\\
    \H^2(X, T_\ell\Acal)_\tors &= \H^1(X, \Acal[\ell^\infty])_\Div \label{eq:H2} \\
    \H^1(X, \Acal[\ell^\infty])_\Div &= \Sha(\Acal/X)[\ell^\infty] \quad\text{if $\Sha(\Acal/X)[\ell^\infty]$ is finite} \label{eq:H1andSha}
\end{align}
\end{lemma}
\begin{proof}
\eqref{eq:HiVanishing}:  For $i > 2d+1$ this follows from~\eqref{eq:sesHS} (using the fact that $\H^i(\overline{X},T_\ell\Acal) = 0$ for $i > 2d$, as noted earlier below~\eqref{eq:H2dH2dplus1}), and it holds for $i = 0$ since $\H^0(X, \Acal[\ell^n]) \subseteq \Acal(X)_\tors$ is finite (since $\Acal(X)$ is a finitely generated Abelian group by the Mordell-Weil theorem~\cref{thm:MordellWeil} and the Néron mapping property~\cref{thm:NeronModel}) hence its Tate-module is trivial.

\eqref{eq:H1} and~\eqref{eq:H2}:
From~\cref{lemma:lesGarbensequenz}, we get 
\begin{align*}
    \H^i(X, T_\ell\Acal)_\tors &= \H^{i-1}(X, \Acal[\ell^\infty])_\Div 
\end{align*}
The desired equalities follow by plugging in $i = 1,2$.

Further, one has $\H^0(X, \Acal[\ell^\infty])_\Div = {\H^0(X, \Acal)[\ell^\infty]}_\tors$ in~\eqref{eq:H1} because $\H^0(X, \Acal[\ell^\infty])$ is cofinitely generated by the Mordell-Weil theorem and the Néron mapping property~\cref{thm:NeronModel}.

Finally, \eqref{eq:H1andSha} holds since by~\cref{lemma:AcalExactSequence}, $\H^1(X, \Acal[\ell^\infty])_\Div = \H^1(X,\Acal)[\ell^\infty]$ if the latter is finite, and this equals $\Sha(\Acal/X)[\ell^\infty]$.
\end{proof}

Now we have two pairings given by cup product in cohomology
\begin{align}
    \qu{\cdot,\cdot}_\ell&: \H^1(X, T_\ell\Acal)_\Tors \times \H^{2d-1}(X, T_\ell(\Acal^t)(d-1))_\Tors \to \H^{2d}(X,\Z_\ell(d)) \stackrel{\pr_1^*}{\to} \H^{2d}(\overline{X},\Z_\ell(d)) = \Z_\ell, \label{eq:Regulator spitze Klammer}\\
    (\cdot,\cdot)_\ell&: \H^2(X, T_\ell\Acal)_\Tors \times \H^{2d-1}(X, T_\ell(\Acal^t)(d-1))_\Tors \to \H^{2d+1}(X,\Z_\ell(d)) = \Z_\ell. \label{eq:Regulator runde Klammer}
\end{align}

\begin{lemma} \label[lemma]{lemma:AAstrichBpairings}
Let $A, A'$ and $B$ finitely generated free $\Z_\ell$-modules.  Consider the commutative diagram
\[\begin{tikzcd}
	A  \arrow[r,"\times",phantom] \arrow[d,"f"] & B \arrow[d,equal] \arrow[r,"\qu{\cdot{,}\cdot}"] & \Z_\ell \arrow[d,equal]\\
	A' \arrow[r,"\times",phantom] & B \arrow[r,"(\cdot{,}\cdot)"] & \Z_\ell,
\end{tikzcd}\]
where $\qu{\cdot,\cdot}$ is a non-degenerate pairing.

Then $f$ is an isomorphism up to finite groups iff $(\cdot,\cdot)$ is non-degenerate, and in this case one has
\[
    q(f) = \left|\frac{\det\qu{\cdot,\cdot}}{\det(\cdot,\cdot)}\right|_\ell^{-1}.
\]
\end{lemma}
\begin{proof}
Since the $\Z_\ell$-modules are finitely generated free, the pairings are non-degenerate iff they are perfect after tensoring with $\Q_\ell$.  So $f$ is an isomorphism up to finite groups iff $f \otimes_{\Z_\ell} \Q_\ell$ is an isomorphism iff $f_{\Q_\ell}^*: \Hom_{\Q_\ell}(A'_{\Q_\ell},\Q_\ell) \isoto \Hom_{\Q_\ell}(A_{\Q_\ell},\Q_\ell) = B_{\Q_\ell}$ is an isomorphism iff $(\cdot,\cdot)_{\Q_\ell}$ is perfect (the last equality coming from the following facts: (a) $\qu{\cdot,\cdot}_{\Q_\ell}$ is perfect, (b) a non-degenerate pairing of finite dimensional vector spaces is perfect and (c) the non-degeneracy of a pairing is preserved by localisation).

The statement on $q(f)$ follows by considering the dual diagram
\[\begin{tikzcd}
    A \arrow[r,hookrightarrow] \arrow[d,"f"] & B^\vee \arrow[d,equal]\\
    A' \arrow[r,hookrightarrow] & B^\vee.
\end{tikzcd}\]
from~\cite[p.~433\,f., Lemma~z.1 and Lemma~z.2]{Artin-Tate}.
\end{proof}


\begin{lemma} \label[lemma]{lemma:TwoPairings}
Recall the maps $\alpha,\beta,f$ from diagram~\eqref{eq:sesHS1}.  The pairing $(\cdot,\cdot)_\ell$~\eqref{eq:Regulator runde Klammer} is non-degenerate.  The regulator term $q((\beta f \alpha)_\Tors)$ is defined iff $f$ is an isomorphism up to finite groups, and then it equals
\[
    \left|\frac{\det\qu{\cdot,\cdot}_\ell}{\det(\cdot,\cdot)_\ell}\right|_\ell^{-1}
\]
where both pairings are non-degenerate.  Conversely, if the pairing $\qu{\cdot,\cdot}_\ell$~\eqref{eq:Regulator spitze Klammer} is non-degenerate, $f$ is an isomorphism up to finite groups.
\end{lemma}
\begin{proof}
Using $\H^{2d+1}(X,\Z_\ell(d)) = \Z_\ell$ and $\H^{2d}(\overline{X},\Z_\ell(d)) = \Z_\ell$ by~\eqref{eq:ZellH2dbarX}, there is a commutative diagram of pairings
\[\begin{tikzcd}
    \H^2(X,T_\ell\Acal)_\Tors \arrow[right,"\times" near end,phantom] & \H^{2d-1}(X,T_\ell(\Acal^t)(d-1))_\Tors \arrow[d,"\iso"] \arrow[r,"(\cdot{,}\cdot)_\ell"] & \Z_\ell \arrow[d,equal]\\
    (\H^1(\overline{X},T_\ell\Acal)_\Gamma)_\Tors \arrow[u,"\beta_\Tors",hookrightarrow] \arrow[right,"\times",phantom] & (\H^{2d-1}(\overline{X},T_\ell(\Acal^t)(d-1))^\Gamma)_\Tors \arrow[d,equal] \ar[r,"\cup"] & \Z_\ell \arrow[d,equal]\\
    (\H^1(\overline{X},T_\ell\Acal)^\Gamma)_\Tors \arrow[u,"f_\Tors"] \arrow[right,"\times",phantom] & (\H^{2d-1}(\overline{X},T_\ell(\Acal^t)(d-1))^\Gamma)_\Tors \arrow[r,"\cup"] & \Z_\ell \arrow[d,equal]\\
    \H^1(X,T_\ell\Acal)_\Tors \arrow[u,"\alpha_\Tors" left,"\iso" right] \arrow[right,"\times",phantom] & \H^{2d-1}(X,T_\ell(\Acal^t)(d-1))_\Tors \arrow[u,"\iso" right] \ar[r,"\qu{\cdot{,}\cdot}_\ell"] & \Z_\ell,
\end{tikzcd}\]
where the maps $\alpha_\Tors$, $\beta_\Tors$ and $f_\Tors$ are induced by the maps $\alpha$ resp.\ $\beta$ resp.\ $f$ in diagram~\eqref{eq:sesHS1}.  Note that by~\cref{lemma:alphafandbeta}, $\alpha_\Tors$ is an isomorphism and $\beta_\Tors$ is injective with finite cokernel.

As in~\cite[p.~137, (5)]{SchneiderZeta}, this diagram is commutative with the pairing in the second line non-degenerate:  By Poincaré duality~\cite[p.~276, Theorem~VI.11.1]{MilneEtaleCohomology} (using that $T_\ell\Acal$ is a smooth sheaf since the $\Acal[\ell^n]$ are étale), the pairing in the second line is non-degenerate, hence the pairing in the first line is also non-degenerate since $\beta$ is an isomorphism up to finite groups.  The upper right and the lower right arrows (note that these are the same morphism) are isomorphisms since their kernel is $(\H^{2d-2}(\overline{X},T_\ell(\Acal^t)(d-1))_\Tors$ by~\cref{lemma:HochschildSerreSS}, which is $0$ for weight reasons by~\cref{cor:weightnotequalzerofinite}: $(2d-2) + (-1) - 2(d-1) = -1 \neq 0$; the lower left arrow $\alpha_\Tors$ is an isomorphism since $\alpha$ is an isomorphism up to finite groups since it is surjective by~\eqref{eq:sesHS1} with finite kernel again by~\eqref{eq:sesHS1} (the kernel $\H^0(\overline{X},T_\ell\Acal)^\Gamma$ is finite since $\H^0(\overline{X},T_\ell\Acal)$ has weight $-1+0 \neq 0$ by~\cref{prop:RiundTateModul}, so the $\Gamma$-invariants are finite by~\cref{cor:weightnotequalzerofinite}).

Hence if $f$ is an isomorphism up to finite groups, the pairing $\qu{\cdot,\cdot}_\ell$ is non-degenerate by~\cref{lemma:AAstrichBpairings}, and then the claimed equality for the regulator $q((\beta f \alpha)_\Tors) = \card{\coker(\beta f \alpha)_\Tors}$ follows since $\ker(\beta f \alpha)_\Tors = 0$ by~\cref{lemma:AAstrichBpairings}.

Conversely, if $\qu{\cdot,\cdot}_\ell$ is non-degenerate, $f$ is an isomorphism up to finite groups by~\cref{lemma:AAstrichBpairings}.
\end{proof}

\begin{lemma} \label[lemma]{lemma:SesAcal}
Let $\ell$ be invertible on $X$.  Then one has a short exact sequence
\[
    0 \to \Acal(X) \otimes_\Z \Z_\ell \stackrel{\delta}{\to} \H^1(X,T_\ell\Acal) \to \varprojlim_n(\H^1(X,\Acal)[\ell^n]) \to 0.
\]
If $\Sha(\Acal/X)[\ell^\infty] = \H^1(X,\Acal)[\ell^\infty]$ is finite, $\delta$ induces an isomorphism
\[
    \Acal(X) \otimes_\Z \Z_\ell \isoto \H^1(X,T_\ell\Acal).
\]
\end{lemma}
\begin{proof}
Since $\ell$ is invertible on $X$, the short exact Kummer sequence of étale sheaves
\[
    0 \to \Acal[\ell^n] \to \Acal \stackrel{\ell^n}{\to} \Acal \to 0
\]
induces a short exact sequence
\[
    0 \to \Acal(X)/\ell^n \stackrel{\delta}{\to} \H^1(X,\Acal[\ell^n]) \to \H^1(X,\Acal)[\ell^n] \to 0
\]
in cohomology, and passing to the limit $\varprojlim_n$ gives us the desired short exact sequence since the $\Acal(X)/\ell^n$ are finite by the Mordell-Weil theorem~\cref{thm:MordellWeil} and the Néron mapping property~\cref{thm:NeronModel}, so they satisfy the Mittag-Leffler condition and $\varprojlim_n^1\Acal(X)/\ell^n = 0$.

The second claim follows from the short exact sequence and since the Tate module of a finite group is trivial by~\cref{lemma:PropertiesOfTateModule}\,\ref{lemma:Tate-module-of-finite-group-is-trivial}.
\end{proof}

\begin{lemma} \label[lemma]{lemma:ShaFinite}
Consider the following statements:

(1) $\qu{\cdot,\cdot}_\ell$ is non-degenerate.

(2) The morphism $f$, where $f$ is as in~\eqref{eq:sesHS1}, is an isomorphism up to finite groups.

(3) In the inequality $\rho \geq \rk_{\Z_\ell}\H^1(X,T_\ell\Acal)$ from~\cref{lemma:L1tilda}, equality holds: $\rho = \rk_{\Z_\ell}\H^1(X,T_\ell\Acal)$.

(4) The canonical injection $\Acal(X) \otimes_\Z \Z_\ell \stackrel{\delta}{\to} \H^1(X,T_\ell\Acal)$ is surjective.

(5) The $\ell$-primary part of the Tate-Shafarevich group $\Sha(\Acal/X)[\ell^\infty]$ is finite.

Then $(1) \iff (2) \iff (3)$ and $(4) \iff (5)$; further $(3) \iff (4)$ assuming $\rho = \rk_\Z\Acal(X)$.

Furthermore, the following are equivalent:

(a) $\rho = \rk_\Z\Acal(X)$ (weak Birch-Swinnerton-Dyer conjecture)

(b) $\qu{\cdot,\cdot}_\ell$ is non-degenerate and the $\ell$-primary part of the Tate-Shafarevich group $\Sha(\Acal/X)[\ell^\infty]$ is finite.
\end{lemma}
\begin{proof}
$(1) \iff (2)$: See~\cref{lemma:TwoPairings}.

$(2) \iff (3)$: This is~\cref{lemma:L1tilda}.

$(3) \iff (4)$ assuming $\rho = \rk_\Z\Acal(X)$: One has $\rho = \rk_{\Z}\Acal(X) = \rk_{\Z_\ell}(\Acal(X) \otimes_\Z \Z_\ell)$ and by~\cref{lemma:SesAcal} $\rk_{\Z_\ell}(\Acal(X) \otimes_\Z \Z_\ell) \leq \rk_{\Z_\ell}\H^1(X,T_\ell\Acal)$, so this is an equality iff $\delta$ in 4.\ is onto.

$(4) \iff (5)$:  By~\cref{lemma:PropertiesOfTateModule}\,\ref{lemma:Tate-module-of-is-trivial-so-finite}, $\varprojlim_n(\H^1(X,\Acal)[\ell^n]) = T_\ell(\H^1(X,\Acal))$ is trivial iff $\H^1(X,\Acal)[\ell^\infty] = \Sha(\Acal/X)[\ell^\infty]$ is finite since $\Sha(\Acal/X)[\ell^\infty]$ is a cofinitely generated $\Z_\ell$-module by~\cref{lemma:ShaEndlichenEllKorang}.  

(a)$\implies$(b):  Since $\delta$ in (4) is injective, one has $\rk_\Z\Acal(X) \leq \rk_{\Z_\ell}\H^1(X,T_\ell\Acal) \leq \rho$.  Therefore, $\rho = \rk_\Z\Acal(X)$ implies equality, and (3) and (4) follow, so (1)--(5) hold.

(b)$\implies$(a):  from (b) follows $(5) \implies (4)$ and $(1) \implies (2) \implies (3)$, so from (4) one gets $\Acal(X) \otimes_\Z \Z_\ell \isoto \H^1(X,T_\ell\Acal)$, but by (3), $\rho = \rk_{\Z_\ell}\H^1(X,T_\ell\Acal) = \rk_\Z\Acal(X)$.
\end{proof}

\begin{remark} \label[remark]{rem:exp}
We have
\begin{align*}
    1 - q^{1-s} = 1 - \exp(-(s-1)\log{q}) = (\log{q})(s-1) + O\big((s-1)^2\big) \quad\text{for $s \to 1$}
\end{align*}
using the Taylor expansion of $\exp$.
\end{remark}

\begin{definition}
	Define $c$ by
	\begin{align}
		L(\Acal/X,s) &\sim c\cdot(1-q^{1-s})^\rho \nonumber \\
	&\sim c\cdot(\log{q})^\rho(s-1)^\rho \quad\text{for $s \to 1$}, \label{eq:definitionofc}
	\end{align}
	see~\cref{rem:exp}.
\end{definition}

\begin{remark}
 	Note that $c \in \Q$ since $L(\Acal/X,s)$ is a rational function with $\Q$-coefficients in $q^{-s}$, and $c \neq 0$ since $\rho$ is the vanishing order of the $L$-function at $s=1$ by definition of $\rho$ and the Riemann hypothesis.
\end{remark}

\begin{corollary} \label[corollary]{cor:cell}
If $\rho = \rk_{\Z_\ell}\H^1(X, T_\ell\Acal)$, then
\[
    |c|_\ell^{-1} = q((\beta f \alpha)_\Tors) \cdot \frac{\card{\H^2(X,T_\ell\Acal)_\tors}}{\card{\H^1(X,T_\ell\Acal)_\tors} \cdot \card{\H^2(\overline{X},T_\ell\Acal)^\Gamma}}.
\]
\end{corollary}
\begin{proof}
Using~\cref{lemma:L1tilda} for $\tilde{L}_1(\Acal/X,t)$ and~\cref{cor:Liell} for $L_0(\Acal/X,t)$, one gets
\begin{align*}
    |c|_\ell^{-1} &= q((\beta f \alpha)_\Tors) \cdot \frac{\card{\H^0(\overline{X},T_\ell\Acal)_\Gamma}}{\card{\H^2(\overline{X},T_\ell\Acal)^\Gamma}} \cdot \frac{\card{\H^2(X,T_\ell\Acal)_\tors}}{\card{\H^1(X,T_\ell\Acal)_\tors}} \cdot \frac{\card{\H^0(\overline{X},T_\ell\Acal)^\Gamma}}{\card{\H^0(\overline{X},T_\ell\Acal)_\Gamma}} \\
            &= q((\beta f \alpha)_\Tors) \cdot \frac{1}{\card{\H^2(\overline{X},T_\ell\Acal)^\Gamma}} \cdot \frac{\card{\H^2(X,T_\ell\Acal)_\tors}}{\card{\H^1(X,T_\ell\Acal)_\tors}} \cdot \frac{\card{\H^0(\overline{X},T_\ell\Acal)^\Gamma}}{1}.
\end{align*}
For $0 = \H^0(X,T_\ell\Acal) \isoto \H^0(\overline{X},T_\ell\Acal)^\Gamma$ use~\eqref{eq:sesHS} with $i=0$ and~\eqref{eq:HiVanishing}.
\end{proof}

\begin{theorem}[analogue of the conjecture of Birch and Swinnerton-Dyer for Abelian schemes over higher dimensional bases, cohomological version] \label[theorem]{thm:BSDI}
Recall that $k$ is a finite field of characteristic $p$, $X/k$ is a smooth projective geometrically connected variety and $\Acal/X$ is an Abelian scheme with analytic rank $\rho = \ord\nolimits_{t=1}L(\Acal/X,t)$.  One has $\rho \geq \rk_{\Z_\ell}\H^1(X, T_\ell\Acal) \geq \rk_\Z \Acal(X)$ and the following are equivalent:

(a) $\rho = \rk_{\Z}\Acal(X) = \rk_{\Z}A(K)$

(b) For some $\ell \neq p = \Char{k}$, $\qu{\cdot,\cdot}_\ell$ is non-degenerate and $\Sha(\Acal/X)[\ell^\infty]$ is finite.

If these hold, we have for all $\ell \neq p$
\[
    |c|_\ell^{-1} = \left|\frac{\det\qu{\cdot,\cdot}_\ell}{\det(\cdot,\cdot)_\ell}\right|_\ell^{-1} \cdot \frac{\card{\Sha(\Acal/X)[\ell^\infty]}}{\card{\Acal(X)[\ell^\infty]_\tors} \cdot \card{\H^2(\overline{X},T_\ell\Acal)^\Gamma}},
\]
where the special $L$-value $c$ is defined by~\eqref{eq:definitionofc}, and the prime-to-$p$ torsion $\Sha(\Acal/X)[\text{non-$p$}]$ is finite.
\end{theorem}
\begin{proof}
Note that $\rho = \ord\nolimits_{t=1}L_1(\Acal/X,t) = \ord\nolimits_{t=1}L(\Acal/X,t)$ by~\cref{rem:DefinitionOfLFunction}, and that $\rk_{\Z_\ell}\H^1(X, T_\ell\Acal) \geq \rk_\Z \Acal(X)$ by the injection from~\cref{lemma:ShaFinite}\,(4).  The first statement is~\cref{lemma:ShaFinite} (a)$\iff$(b).  Now identify the terms in~\cref{cor:cell} using~\cref{lemma:cohTerms} (cohomology groups) and~\cref{lemma:TwoPairings} (regulator).

By~\cref{thm:BSDI}\,(b) for $\ell$ $\implies$ (a) independent of $\ell$ $\implies$ (b) for $\ell'$, $\Sha(\Acal/X)[\ell'^\infty]$ is finite for every $\ell' \neq p$.  But since $c \neq 0$, and by the relation of $|c|_{\ell'}^{-1}$ and $|\Sha(\Acal/X)[\ell'^\infty]|$, the prime-to-$p$ torsion is finite.
\end{proof}

\begin{remark} \label[remark]{rem:Lfn}
(a) For example, \cref{thm:BSDI} holds unconditionally if $L(\Acal/X,1) \neq 0$ since one then has $0 = \rho \geq \rk_{\Z}\Acal(X) \geq 0$.  For examples when $\Sha(\Acal/X)[\ell^\infty]$ is finite, see section~\ref{sec:IsoconstantAbelianScheme}.

(b) The (determinants of the) pairings $\qu{\cdot,\cdot}_\ell$ and $(\cdot,\cdot)_\ell$ are identified below:  One has $\det(\cdot,\cdot)_\ell = 1$ and $\det\qu{\cdot,\cdot}_\ell$ is the regulator, see especially~\cref{rem:HeightPairingAndCohPairing}.

(c) For the vanishing of $\H^2(\overline{X},T_\ell\Acal)^\Gamma$ see~\cref{rem:H2vanishing} below.

(d) Assume $\Sha(\Acal/X)[\ell^\infty]$ finite.  Then, the pairing $(\cdot,\cdot)_\ell$ has determinant $1$, see the discussion in subsection~\ref{subsection:pairingroundbrackets} below, and is thus non-degenerate.  The pairing $\qu{\cdot,\cdot}_\ell$ equals the height pairing, see subsection~\ref{subsection:pairingacutebrackets} below, and is therefore non-degenerate by~\cite[p.~98, Theorem~9.15]{Conrad-trace}.
\end{remark}

\begin{remark}
	This remark is about the independence of the analytic rank $\ord\nolimits_{s=1}L(\Acal/X,s)$ on the model $\Acal/X$ of $A/K$.
	
	Note that the vanishing order of $L(\Acal/X,s)$ at $s = 1$, the analytic rank, only depends on $L_1(\Acal/X,s)$ since $L_0(\Acal/X,1) = \det(1-\Frob_q^{-1} \mid \H^0(\overline{X},V_\ell\Acal)) \neq 0$ by~\cref{lemma:BayerNeukirch} ``$2 \implies 1$'' since $\H^0(\overline{X}, V_\ell\Acal)$ is pure of weight $0-1 \neq 0$ by~\cref{prop:RiundTateModul} below, so its invariants $\H^0(\overline{X}, V_\ell\Acal)^\Gamma$ are finite by~\cref{cor:weightnotequalzerofinite}.  Furthermore, the vanishing order of $L_1$ at $s = 1$ only depends on the generic fibre $A/K$ (and not on the model $X$) assuming the conjecture of Birch and Swinnerton-Dyer for $\Acal/X$ by an a posteriori argument:  If the conjecture holds, by~\cref{thm:BSDI}, the vanishing order at $s = 1$ of $L(\Acal/X,s)$ equals the (algebraic) rank of $A(K)$.
	
	For $\dim{X} = 1$, there is a canonical model $X$ of $K$.  In contrast, for higher dimensional $X$, there is no canonical model (one can e.\,g.\ blow up smooth centres), and the special $L$-value depends on the model.  If every birational morphism of smooth projective $k$-varieties of dimension $d$ is given by a sequence of monoidal transformations (e.\,g.\ for surfaces, see~\cite[p.~412, Theorem~V.5.5]{HartshorneAG} over algebraically closed fields), the vanishing order of $L_1(\Acal/X,1) = \det(1-\Frob_q^{-1} \mid \H^1(\overline{X},V_\ell\Acal))$ is independent of the model of $\overline{X}$ by calculation of the étale cohomology of blow-ups of torsion sheaves~\cite[section~0EW3]{stacks-project}:  If $\overline{X}'$ is the blow-up of $\overline{X}$ along a closed point $\overline{Z}$ with exceptional divisor $\overline{E} \iso \mathbf{P}^{c-1}_{\overline{k}}$, then there is an exact sequence of proper varieties over $\overline{k}$
	\[
	\H^0(\overline{E},V_\ell\Acal) \to \H^1(\overline{X},V_\ell\Acal) \to \H^1(\overline{X}',V_\ell\Acal) \oplus \H^1(\overline{Z},V_\ell\Acal) \to \H^1(\overline{E},V_\ell\Acal).
	\]
	Here, $\H^1(\overline{Z},V_\ell\Acal) = 0$ since $\cd_\ell\overline{Z} = 0$, $\H^0(\overline{E},V_\ell\Acal)$ is pure of weight $0-1 \neq 0$ and $\Acal[\ell^n]|_{\overline{E}} \iso \mu_{\ell^n}^{2g}$ since $\piet(\mathbf{P}^{c-1}_{\overline{k}}) = 0$ and $\Acal[\ell^n]/X$ is finite étale, so $\H^1(\overline{E},V_\ell\Acal) = \H^1(\mathbf{P}^{c-1}_{\overline{k}},\Q_\ell(1)) = 0$.
\end{remark}

\section{Comparison of the cohomological pairing $\qu{\cdot,\cdot}_\ell$ with geometric height pairings} \label{subsection:pairingacutebrackets}
The objective of this section is to show that the cohomological pairing $\qu{\cdot,\cdot}_\ell$ discussed above coincides, up to multiplication by a certain \emph{integral hard Lefschetz defect} (see~\cref{def:integralhardLefschetzdefect}), with certain other geometric pairings defined in subsection~\ref{subsection:DefinitionOfGeometricPairings} below, namely the \emph{generalised Bloch} and the \emph{generalised Néron-Tate canonical height} pairings (see~\cref{def:BlochPairing} and~\cref{def:generalisedNeronTatecanonicalheightpairing}, respectively).  The equality of these three pairings (up to the indicated integral hard Lefschetz defect) is the content of the main~\cref{thm:comparisonofpairings} of this section, which is essentially equivalent to the commutativity of diagram~\eqref{eq:BigDigram} in~\cref{thm:bigdiagramPairings}.  To establish the commutativity of diagram~\eqref{eq:BigDigram}, we first compare (in subsection~\ref{subsect:Yoneda}) $\qu{\cdot,\cdot}_\ell$ with a certain Yoneda pairing.  The main results in this subsection are~\cref{lemma:commutativityof2firstpart} and~\cref{lemma:cupproductandintersectionproduct}, which establish the commutativity of subdiagrams (1) and (2) in diagram~\eqref{eq:BigDigram}.  In subsection~\ref{subsec:Bloch}, the Yoneda pairing is compared with the generalised Bloch pairing and, in subsection~\ref{subsec:NTheight}, the generalised Bloch pairing is shown to coincide with the generalised Néron-Tate canonical height pairing (see~\cref{cor:BlochpairingandNeronTatepairing}).  The developments in subsection~\ref{subsec:Bloch} and~\ref{subsec:NTheight} yield the commutativity of subdiagram~(3) in~\eqref{eq:BigDigram}, thereby establishing the commutativity of the full diagram and thereby proving~\cref{thm:comparisonofpairings}.

\subsection{Definition of the geometric pairings}\label{subsection:DefinitionOfGeometricPairings}

We wish to define a \defn{generalised Bloch pairing}
\[
	\qu{\cdot,\cdot}: A(K) \times A^t(K) \to \RR,
\]
see~\cref{def:BlochPairing} below.  To this end, we need some preparations.

Let $X$ be a $k$-variety and $K = k(X)$ be the function field of $X$.  Define for $S \subset X^{(1)}$ finite the $S$-adele ring of $X$ as the restricted product
\begin{align*}
\A_{K,S} &= \prod_{x \in S}K_x \times \prod_{v \in X^{(1)} \setminus S}\Ocal_{X,x}
\end{align*}
where $K_x$ is the quotient field of the discrete valuation ring $\Ocal_{X,x}$ with discrete valuation $v_x: K_x^\times \to \Z$ and absolute value $|\cdot|_{\iota,x} = q^{-\deg_\iota\overline{\{x\}}\cdot v_x(\cdot)}$, and the \defn{adele ring} of $X$ as
\begin{align*}
	\A_K & = \varinjlim_{S \subset X\text{ finite}} \A_{K,S}. 
\end{align*}

\begin{proposition}[adele valued points] \label[proposition]{prop:AdeleValuedPoints}
	Let $X$ be a $k$-variety and $S \subset X^{(1)}$ be a finite set of places.  Then
	\[
	\varinjlim_{S'} X_{S'}(\A_{K,S'}) = X_S(\A_K) = X(\A_K)
	\]
	and
	\[
	X_S(\A_{K,S}) = \prod_{v \in S}X_v(K_v) \times \prod_{v \not\in S}X_{S,v}(\Ocal_{X,v})
	\]
	as sets with the notation from~\cite[p.~70]{Conrad-Adelic}.  This bijection is used to define a topology on $X_S(\A_{K,S})$.
\end{proposition}
\begin{proof}
	See~\cite[p.~70\,f., (3.1) and Theorem~3.6]{Conrad-Adelic}.
\end{proof}

\begin{lemma} \label[lemma]{lemma:PicardGroupOfInfiniteProduct}
	Let $(R_i)_{i \in I}$ be a family of rings with $\Pic(\Spec{R_i}) = 0$ for every $i \in I$.  Then $\Pic(\prod_{i \in I}\Spec R_i) = 0$.  (Note that the infinite fibre product exists and is affine by~\cite[\href{https://stacks.math.columbia.edu/tag/0CNH}{Tag 0CNH}]{stacks-project}.)
\end{lemma}
\begin{proof}
	Line bundles correspond to $\G_m$-torsors, and a torsor is trivial iff it has a section.  So let $\Lcal$ be a line bundle on $\prod_{i \in I}\Spec R_i$.  Since line bundles on affine schemes are affine, $\Lcal$ is represented by an affine scheme $X$.  One has
	\[
		X(\prod_{i \in I}\Spec R_i) = \prod_{i\in I}X(R_i),
	\]
	by~\cite[p.~72\,ff., proof after~(3.3)]{Conrad-Adelic}.  Each of the factors has a non-trivial element by $\Pic(R_i) = 0$.  Hence, the product is non-empty by the axiom of choice.
\end{proof}

\begin{corollary} \label[corollary]{cor:PicAdeleRingTrivial}
	The Picard group of the adele ring $\A_K$ is trivial.
\end{corollary}
\begin{proof}
	By the previous~\cref{lemma:PicardGroupOfInfiniteProduct}, $\Pic(\A_{K,S}) = 0$ since line bundles on the local rings $K_x$ and $\Ocal_{X,x}$ are trivial, and $\Pic(\A_K) = \varinjlim_S \Pic(\A_{K,S})$ by the compatibility of étale cohomology ($\Pic(X) = \H^1(X,\G_m)$) with limits, see~\cite[p.~88\,f., Lemma~III.1.16]{MilneEtaleCohomology}.
\end{proof}

Let $a \in A(K)$ and $a^t = (1 \to \G_m \to \mathscr{X} \to \Acal \to 0) \in \EExt^1_X(\Acal,\G_m) = \Acal^t(X) = A^t(K)$.  By descent theory, $\mathscr{X}$ is a smooth commutative $X$-group scheme, and by Hilbert's theorem~90, the sequence
\begin{equation} \label{eq:GmKXKAcalK}
1 \to \G_m(K) \to \mathscr{X}(K) \to \Acal(K) \to \H^1(K,\G_m) = 0
\end{equation}
and, by~\cref{cor:PicAdeleRingTrivial},
\begin{equation} \label{eq:GmKXKAcalAK}
1 \to \G_m(\A_K) \to \mathscr{X}(\A_K) \to \Acal(\A_K) \to \H^1(\A_K,\G_m) = \Pic(\A_K) = 0
\end{equation}
are still exact.

Fix a closed immersion $\iota: X \hookrightarrow \PP^N_k$ with the very ample sheaf $\Ocal_X(1) := \iota^*\Ocal_{\PP^N_k}(1)$.  There is a natural homomorphism, the \defn{logarithmic modulus map},
\begin{equation} \label{eq:homomorphisml}
l: \G_m(\A_K) \to \log{q}\cdot\Z \subseteq \RR, (a_x) \mapsto \sum_{x \in X^{(1)}}\log|a_x|_{\iota,x} = - \log{q} \cdot \sum_{x \in X^{(1)}}\deg_\iota\overline{\{x\}}\cdot v_x(a_x).
\end{equation}
By the product formula (see~\cite[p.~146, Exercise~II.6.2\,(d)]{HartshorneAG}), $l(\G_m(K)) = l(K^\times) = 0$.  Scale the image of $l$ such that $l$ is surjective.

\begin{lemma} \label[lemma]{lemma:laextends}
	The homomorphism $l: \G_m(\A_K) \to \log{q}\cdot\Z \subseteq \RR$~\eqref{eq:homomorphisml} has a unique extension $l_{a^t}: \mathscr{X}(\A_K) \to \RR$, which induces by restriction to $\mathscr{X}(K)$ a homomorphism
	\[
	l_{a^t}: A(K) \to \RR.
	\]
\end{lemma}
\begin{proof}
	Define $\G_m^1$ as the kernel of $l$, and $\mathscr{X}^1$ as
	\[
	\mathscr{X}^1 = \{a \in \mathscr{X}(\A_K) : \exists n \in \Z_{\geq 1}, na \in \widetilde{\mathscr{X}}^1\}
	\]
	the rational saturation of $\widetilde{\mathscr{X}}^1$ with
	\begin{align*}
	\widetilde{\mathscr{X}}^1 &= \G_m^1\cdot\prod_{v \in X^{(1)}}\mathscr{X}(\Ocal_{X,v}) \subseteq \mathscr{X}(\A_K).
	\end{align*}
	
	Consider the following commutative diagram (at first without the dashed arrows) with exact rows by~\eqref{eq:GmKXKAcalK} and~\eqref{eq:GmKXKAcalAK} and injective upper vertical morphisms and exact left column:
	\begin{equation} \label{eq:commdiaglat}
	\begin{gathered}
	\begin{tikzcd}
	& 1 \arrow[d] & 0 \arrow[d,dashed] & & \\
	1 \arrow[r] & \G_m^1 \arrow[r] \arrow[d] & \mathscr{X}^1 \arrow[r] \arrow[d] & \Acal(\A_K) \arrow[r] \arrow[d,equal] & 0\\
	1 \arrow[r] & \G_m(\A_K) \arrow[r] \arrow[d,"l"] & \mathscr{X}(\A_K) \arrow[r] \arrow[d,"l_{a^t}",dashed] & \Acal(\A_K) \arrow[r] & 0\\
	& \log{q}\cdot\Z \arrow[d] \arrow[r,dashed,equal] & \log{q}\cdot\Z \arrow[d] & & \\
	& 0 & 0 & &
	\end{tikzcd}
	\end{gathered}
	\end{equation}
	For the commutativity of the diagram, it suffices to show that (1) $\G_m^1 = \mathscr{X}^1 \cap \G_m(\A_K) \subseteq \mathscr{X}(\A_K)$ and (2) $\mathscr{X}^1 \twoheadrightarrow \Acal(\A_K)$.
	
	Assertion (1) is true because of the following:  One has $\G_m^1 \subseteq \G_m(\A_K) \cap \mathscr{X}^1$ by definition of $\mathscr{X}^1$ and $\G_m^1$.  For the other inclusion $\mathscr{X}^1 \cap \G_m(\A_K) \subseteq \G_m^1$, note that
	\[
	l\Big(\prod_{v \in X^{(1)}}\G_m(\Ocal_{X,v})\Big) = 0
	\]
	since $\log|a_v|_{\iota,v} = 0$ for $a_v \in \Ocal_{X,v}^\times$, hence
	\[
	\G_m^1\cdot\prod_{v \in X^{(1)}}\G_m(\Ocal_{X,v}) \subseteq \G_m^1,
	\]
	so
	\[
	\widetilde{\mathscr{X}}^1 \cap \G_m(\A_K) = \G_m^1\cdot\prod_{v \in X^{(1)}}\G_m(\Ocal_{X,v}) \subseteq \G_m^1,
	\]
	but $\G_m(\A_K)/\G_m^1 \hookrightarrow \RR$ is torsion-free, hence the inclusion $\mathscr{X}^1 \cap \G_m(\A_K) \subseteq \G_m^1$.
	
	Assertion (2) is true because of the following:  By the long exact sequence associated to the short exact sequence $1 \to \G_m \to \mathscr{X} \to \Acal \to 0$ and~\cref{lemma:PicardGroupOfInfiniteProduct}, there is a surjection
	\[
	\mathscr{X}\Big(\prod_{x \in X^{(1)}}\Ocal_{X,x}\Big) \twoheadrightarrow \Acal\Big(\prod_{x \in X^{(1)}}\Ocal_{X,x}\Big) = \Acal(\A_K),
	\]
	the latter equality by~\cref{prop:AdeleValuedPoints} and the valuative criterion for properness.  But obviously
	\begin{equation} \label{eq:XprodsubseteqX1}
	\mathscr{X}\Big(\prod_{x \in X^{(1)}}\Ocal_{X,x}\Big) \subseteq \mathscr{X}^1.
	\end{equation}
	
	By the snake lemma, the diagram completed with the dashed arrows is also exact and there exists the sought-for extension $l_{a^t}: \mathscr{X}(\A_K) \to \log{q}\cdot\Z$ of $l: \G_m(\A_K) \to \log{q}\cdot\Z$.  The homomorphism $l_{a^t}$ induces by restriction to $\mathscr{X}(K)$ a homomorphism
	\[
	l_{a^t}: A(K) \to \RR,
	\]
	since $l(\G_m(K)) = 0$ by the product formula.  
\end{proof}

\begin{definition} \label[definition]{def:BlochPairing}
	Define the \defn{generalised Bloch pairing} $\qu{\cdot,\cdot}: A(K) \times A^t(K) \to \log{q}\cdot\Z$ as follows:  Let $a \in A(K)$ and $a^t \in A^t(K)$.  Let $\qu{a,a^t}$ be the image of $a$ under the composition of the maps
	\[
	A(K) = \mathscr{X}(K)/\G_m(K) \to \mathscr{X}(\A_K)/\G_m(K) \stackrel{l_{a^t}}{\to} \log{q}\cdot\Z
	\]
	with $l_{a^t}$ coming from~\cref{lemma:laextends}.
\end{definition}
(The first identity comes from~\eqref{eq:GmKXKAcalK}.  Note that $\G_m(K) \subseteq \G_m^1$ by the product formula.)\\

Now we wish to define the \defn{generalised Néron-Tate canonical height pairing} $\hat{h}_{K,\iota}$.

For the definition of a generalised global field see~\cite[p.~83, Definition~8.1]{Conrad-trace}.  Let us recall \textbf{Conrad's height pairing for generalised global fields} from~\cite[p.~82\,ff., section~8]{Conrad-trace}.  Let $X$ be a smooth projective geometrically connected variety over a finite ground field $k = \F_q$ and $K = k(X)$ the function field of $X$.  Choose a closed $k$-immersion $\iota: X \hookrightarrow \PP^N_k$.  For $x \in X^{(1)}$ let
\[
c_{x,\iota} = q^{-\deg_{k,\iota} \overline{\{x\}}} \in \Q \cap (0,1).
\]
For the definition of the degree of a closed subscheme of projective space see~\cite[p.~52]{HartshorneAG}.  Then the absolute values
\[
\|\cdot\|_{x,\iota} = c_{x,\iota}^{\mathrm{ord}_x(\cdot)}
\]
on $K$, where $x \in X^{(1)}$, satisfy the product formula by~\cite[p.~146, Exercise~II.6.2\,(d)]{HartshorneAG}.  Conrad calls the system of these valuations the structure of a \defn{generalised global field} on $K$.  This induces a height function
\[
h_{K,n,\iota}: \PP^n_K(\overline{K}) \to \RR_{\geq 0},\quad h_{K,n,\iota}([t_0:\ldots:t_n]) = \frac{1}{[K':K]}\sum_{v'}\max_{i=0}^n \log\|t_i\|_{v',\iota}
\]
on projective space over $K$, see~\cite[p.~86\,ff.]{Conrad-trace}, with the finitely many lifts $v'$ of $v = x \in X^{(1)}$ to $K'/K$ finite, where $K \subseteq K' \subseteq \overline{K}$ is a finite subextension over $K$ that contains the $t_i$ and we canonically endow $K'$ with a structure of generalised global field via the algebraic method as in~\cite[p.~86]{Conrad-trace}; this is independent of the choice of $K'$, see~\cite[p.~87\,ff.]{Conrad-trace}.

Now, we construct a canonical height pairing
\[
\hat{h}_{\iota} := \hat{h}_{K,\iota}: \Acal(X) \times \Acal^t(X) \to \RR
\]
as follows.  If $\Lcal$ is a very ample line bundle on a smooth projective geometrically connected $K$-variety $X$, the induced closed immersion
\[
\iota: X \hookrightarrow \PP(\H^0(X,\Lcal))
\]
defines a height function
\[
h_{K,\Lcal,\iota} := h_{K,\H^0(X,\Lcal),\iota} \circ \iota_\Lcal: X(\overline{K}) \to \RR,
\]
where $h_{K,\H^0(X,\Lcal),\iota} = h_{K,n,\iota}$ if $\dim_K\H^0(X,\Lcal) = n$.  By linearity, since one can write any line bundle on $X$ as a difference of two very ample line bundles (see~\cite[p.~186, l.~8]{HindrySilverman}), this extends as in~\cite[p.~184, Theorem~B.3.2]{HindrySilverman} \emph{(Weil's height machine)} to a homomorphism
\[
\Pic(X) \to \RR^{X(\overline{K})}/O(1),
\]
where $O(1) := \{f: X(\overline{K}) \to \RR : \text{$f$ is bounded}\} \subset \RR^{X(\overline{K})}$ is the vector subspace of bounded functions.

Our $K$-variety $X$ will now be an Abelian variety $A/K$ arising as the generic fibre of an Abelian scheme.
\begin{definition}[generalised Néron-Tate canonical height pairing] \label[definition]{def:generalisedNeronTatecanonicalheightpairing}
Now let $\Acal/X$ be an Abelian scheme.  In this case, one can, by the Tate limit argument, define a \emph{canonical} height pairing, taking values in $\RR^{A(\overline{K})}$ (not modulo bounded functions)
\[
\hat{h}_{K,\Lcal,\iota}: \Acal(X) = A(K) \to \RR
\]
or
\[
\hat{h}_{K,\iota}: \Acal(X) \times \Acal^t(X) = A(K) \times A^t(K) \to \RR,
\]
respectively as in~\cite[p.~284\,ff.]{BombieriGubler}.  One has $\hat{h}_{K,\iota}(x,\Lcal) = \hat{h}_{K,\Lcal,\iota}(x)$.
\end{definition}

\begin{proposition} \label[proposition]{prop:NeronTateHeightAndPoincareBundle}
	Let $K$ be a generalised global field, $A/K$ an Abelian variety and $\Pcal \in \Pic(A \times_K A^t)$ the Poincaré bundle.  Then
	\[
	\hat{h}_{\Lcal}(x) = \hat{h}_{\Pcal}(x,\Lcal)
	\]
	for $x \in A(K)$ and $\Lcal \in A^t(K)$.
\end{proposition}
\begin{proof}
	See~\cite[p.~292, Corollary~9.3.7]{BombieriGubler}.
\end{proof}

\begin{lemma} \label[lemma]{lemma:xLcal}
	Let $\Acal/X$ be a projective Abelian scheme over a locally Noetherian scheme $X$.  Let $x \in \Acal(X)$ and $\Lcal \in \PPic^0_{\Acal/X}(X) = \Acal^t(X)$.  By the universal property of the Poincaré bundle~\cite[p.~262\,f., Exercise~9.4.3]{FGAExplained}, there is a unique $X$-morphism $h: X \to \Acal^t$ such that $\Lcal = (\id_\Acal \times h)^*\Pcal_\Acal$.  Then $x^*\Lcal = (x,h)^*\Pcal_\Acal$.
\end{lemma}
\begin{proof}
	Note that the map $(x,h): X \to \Acal \times_X \Acal^t$ factors as
	\[\begin{tikzcd}
	X \arrow[r,"x"] & \Acal \times_X X \arrow[r,"\id_\Acal \times h"] & \Acal \times_X \Acal^t.
	\end{tikzcd}\]
	Consequently, $x^*\Lcal = x^*(\id_\Acal \times h)^*\Pcal_\Acal = (x,h)^*\Pcal_\Acal$.
\end{proof}

Now we need to define the \defn{integral hard Lefschetz defect}.

\begin{theorem}[hard Lefschetz for finite ground fields] \label[theorem]{thm:hardLefschetzFiniteGroundField}
Let $k$ be a finite field, $\ell \neq \Char{k}$ be prime and $X/k$ be a smooth projective variety of pure dimension $d$.  Let $\eta \in \H^2(X,\Z_\ell(1))$ be the first Chern class of $\Ocal_X(1) \in \Pic(X)$ (the image of $\Ocal_X(1)$ under the homomorphism $\Pic(X) \to \H^2(X,\Z_\ell(1))$ from~\cite[p.~271, Proposition~VI.10.1]{MilneEtaleCohomology}) and $\Acal/X$ be an Abelian scheme.  Then the iterated cup products
\[
    (\cup\eta)^{i}: \H^{d-i}(X, V_\ell\Acal) \to \H^{d+i}(X, V_\ell\Acal(i))
\]
are isomorphisms.
\end{theorem}
\begin{proof}
This follows from the hard Lefschetz theorem~\cite[p.~144, Théorème~5.4.10]{BBD} for the projective morphism $f: X \to \Spec{k}$ since $\Fcal := V_\ell\Acal[d]$ is a pure perverse sheaf:  The sheaf $\Fcal$ is pure of weight $-1$ by~\cref{prop:RiundTateModul}.  It is perverse:  The sheaf $V_\ell\Acal = \R^1\pi_*\Q_\ell(1)$ with the smooth projective morphism $\pi: \Acal^t \to X$ is smooth by proper and smooth base change~\cite[p.~223, Corollary~VI.2.2 and p.~230, Corollary~VI.4.2]{MilneEtaleCohomology}.  If $\Fcal$ is a smooth sheaf on a smooth pure $d$-dimensional variety, then $\Fcal[d]$ is perverse by~\cite[p.~149, Corollary~III.5.5]{KiehlWeissauer}.
\end{proof}

\begin{corollary}[integral hard Lefschetz for finite ground fields] \label[corollary]{thm:integralhardLefschetz}
The \emph{integral} hard Lefschetz morphism
\[
    (\cup \eta)^{d-1}: \H^1(X, T_\ell(\Acal^t))_\Tors \to \H^{2d-1}(X, T_\ell(\Acal^t)(d-1))_\Tors
\]
is injective with finite cokernel.
\end{corollary}
\begin{proof}
By the hard Lefschetz theorem~\cref{thm:hardLefschetzFiniteGroundField}, it follows that the kernel and the cokernel tensored with $\Q_\ell$ are trivial, hence torsion, hence finite.  Now note that all groups are taken modulo their torsion subgroup, so the kernel is trivial.
\end{proof}

\begin{definition} \label[definition]{def:integralhardLefschetzdefect}
We call the order of the cokernel of the integral hard Lefschetz morphism from~\cref{thm:integralhardLefschetz} the \defn{integral hard Lefschetz defect}.
\end{definition}

\subsection{Comparison of the cohomological pairing with a Yoneda pairing} \label{subsect:Yoneda}
In the rest of this section, if we deal with $\Ext$-groups or $\EExt$-sheaves, we always mean them with respect to the fppf topology in order to have the \emph{Barsotti-Weil formula} $\EExt_{X_\fppf}^1(\Acal,\G_m) \isoto \Acal^t$ (\cite[p.~121, l.~--11]{MilneAbelianVarieties} or~\cite[p.~III.18--1, Theorem~III.18.1]{OortCommutativeGroupSchemes}).  Although we are also dealing with étale cohomology, there is no problem since by~\cite[p.~116, Remark~3.11\,(b)]{MilneEtaleCohomology} the étale and fppf cohomology of sheaves represented by smooth group schemes (we are using $\Acal$, $\Acal[\ell^n]$, $\G_m$ and $\mu_{\ell^n}$ with $\ell$ invertible on $X$) agree.

Note that one has a Yoneda Ext-pairing
\[
    \vee: \Ext^r(A,B) \times \Ext^s(B,C) \to \Ext^{r+s}(A,C),
\]
in Abelian categories with enough injectives, see~\cite[p.~167]{MilneEtaleCohomology}; we will use this several times below.  This induces pairings
\[
    \vee: \H^r(X,\Fcal) \times \Ext_X^s(\Fcal,\Gcal) \to \H^{r+s}(X,\Gcal).
\]
See also~\cite[p.~166\,f.]{GelfandManin}.

\begin{lemma} \label[lemma]{prop:YonedaExtAndHeightPairing}
Let $\Acal/X$ be a projective Abelian scheme over a locally Noetherian scheme $X$.  Then the following diagram commutes:
\[\begin{tikzcd}
    \Acal(X) \arrow[d,equal] \arrow[r,"\times",phantom] & \Acal^t(X) \arrow[r] & \Pic(X) \\
    \H^0(X,\Acal) \arrow[r,"\times",phantom] & \Ext^1_X(\Acal,\G_m) \arrow[u,"\iso"] \arrow[r,"\vee"] & \H^1(X,\G_m) \arrow[u,"\iso"]
\end{tikzcd}\]
Here, the upper pairing is given by $(x,\Lcal) \mapsto x^*\Lcal = (x,\Lcal)^*\Pcal_\Acal$ (the equality by~\cref{lemma:xLcal}) for $x \in \Acal(X)$ and $\Lcal \in \Acal^t(X) = \PPic^0_{\Acal/X}(X)$ with $(x,\Lcal): X \to \Acal \times_X \Acal^t$, and the lower pairing is the Yoneda pairing.
\end{lemma}
\begin{proof}
The morphism $\Ext^1_X(\Acal,\G_m) \to \Acal^t(X)$ is an isomorphism by the Barsotti-Weil formula.

Given $x \in \Acal(X)$, i.\,e.\ $x: X \to \Acal$, and $e: (1 \to \G_m \to G \to \Acal \to 0) \in \Ext^1_X(\Acal,\G_m)$, $(x,e)$ maps to $G \times_\Acal X$ under the Yoneda pairing (composition in the lower row).  This is a $\G_m$-torsor on $X$, namely $x^*\Lcal$ if $G = \Lcal \setminus 0$ ($0$ the zero-section), which is the composition in the upper row.
\end{proof}

\begin{lemma} \label[lemma]{lemma:HHomAcalGmGleich0}
Let $n$ be invertible on $X$.  Then one has
\begin{align*} 
    \HHom_X(\Acal,\G_m) = 0 \text{ and } \HHom_X(\Acal,\mu_n) = 0.
\end{align*}
\end{lemma}
\begin{proof}
This holds since $\G_m$ and $\mu_n$ are affine over $X$ and $\Acal/X$ is proper and has geometrically integral fibres using the Stein factorisation.
\end{proof}


\begin{corollary}
Let $\ell$ be invertible on $X$.  Then the local-to-global Ext spectral sequence $\H^p(X,\EExt_X^q(\Acal,\mu_{\ell^n})) \Rightarrow \Ext_X^{p+q}(\Acal,\mu_{\ell^n})$ gives an injection
\begin{align} \label{eq:H1EExt1Ext2}
    \H^1(X,\EExt_X^1(\Acal,\mu_{\ell^n})) \hookrightarrow \Ext_X^2(\Acal,\mu_{\ell^n}).
\end{align}
\end{corollary}
\begin{proof}
This follows since $\HHom_X(\Acal,\mu_{\ell^n}) = 0$ by~\cref{lemma:HHomAcalGmGleich0}, so $E_2^{p,0} = 0$ for all $p$ in the Ext spectral sequence.
\end{proof}

\begin{lemma} \label[lemma]{lemma:EExt1AcalGmelln}
Let $\ell$ be invertible on $X$.  Then one has
\begin{align*}
    \EExt^1_X(\Acal,\mu_{\ell^n}) = \EExt^1_X(\Acal,\G_m)[\ell^n] = \Acal^t[\ell^n].
\end{align*}
\end{lemma}
\begin{proof}
One has a short exact sequence of sheaves 
\begin{align} \label{eq:EExtGm}
    0 \to \EExt^1_X(\Acal,\G_m)[\ell^n] \to \EExt^1_X(\Acal,\G_m) \stackrel{\ell^n}{\to} \EExt^1_X(\Acal,\G_m) \to 0
\end{align}
since one can check $\EExt^1_X(\Acal,\G_m)/\ell^n = \Acal^t/\ell^n = 0$ on stalks by the exactness of the Kummer sequence.

The short exact Kummer sequence yields by~\cref{lemma:HHomAcalGmGleich0} a short exact sequence
\begin{align} \label{eq:EExtGm2}
    \HHom_X(\Acal,\G_m) = 0 \to \EExt^1_X(\Acal,\mu_{\ell^n}) \to \EExt^1_X(\Acal,\G_m) \stackrel{\ell^n}{\to} \EExt^1_X(\Acal,\G_m) \to 0,
\end{align}
the $0$ at the right hand side by~\eqref{eq:EExtGm}.

Combining~\eqref{eq:EExtGm} and~\eqref{eq:EExtGm2}, one gets the first equation in~\cref{lemma:EExt1AcalGmelln}.  The second equation follows from the Barsotti-Weil formula.
\end{proof}

\begin{lemma} \label[lemma]{lemma:deltaisaniso}
Let $\ell$ be invertible on $X$.  Then one has an isomorphism
\[
    \delta: \HHom_X(\Acal[\ell^n],\mu_{\ell^n}) \isoto \EExt_X^1(\Acal,\mu_{\ell^n}) = \Acal^t[\ell^n].
\]
\end{lemma}
\begin{proof}
Applying the functor $\HHom_X(-,\mu_{\ell^n})$ to the short exact Kummer sequence $0 \to \Acal[\ell^n] \to \Acal \to \Acal \to 0$ gives an exact sequence
\[
    0 = \HHom_X(\Acal,\mu_{\ell^n}) \to \HHom_X(\Acal[\ell^n],\mu_{\ell^n}) \stackrel{\delta}{\to} \EExt_X^1(\Acal,\mu_{\ell^n}) \stackrel{\ell^n}{\to} \EExt_X^1(\Acal,\mu_{\ell^n}),
\]
the first equality by~\cref{lemma:HHomAcalGmGleich0}.  But multiplication by $\ell^n$ kills $\EExt_X^1(\Acal,\mu_{\ell^n})$, so the last arrow is zero.  Hence $\delta$ is an isomorphism.

The equality $\EExt_X^1(\Acal,\mu_{\ell^n}) = \Acal^t[\ell^n]$ is~\cref{lemma:EExt1AcalGmelln}.
\end{proof}

The following is commutativity of part~(1) of diagram~\eqref{eq:BigDigram}.

\begin{proposition} \label[proposition]{lemma:commutativityof2firstpart}
Note that under the assumption $\Sha(\Acal/X)[\ell^\infty]$ finite, one has from~\cref{lemma:SesAcal} an isomorphism
\begin{equation} \label{eq:deltaboundarymapabelianscheme}
	\delta: \Acal(X) \otimes_\Z \Z_\ell \isoto \H^1(X,T_\ell\Acal).
\end{equation}
induced by the boundary map of the long exact sequence induced by the short exact Kummer sequence~\cref{cor:Kummersequence}.  Denote the analogous map for $\Acal^t$ by $\delta^t: \Acal^t(X) \otimes_\Z \Z_\ell \isoto \H^1(X,T_\ell\Acal^t)$.
	
Then the diagram
\[\begin{tikzcd}
    \H^1(X,T_\ell\Acal)_\Tors \times \H^1(X,T_\ell\Acal^t)_\Tors \arrow[r,"\cup"] & \H^2(X,\Z_\ell(1))_\Tors \\
    \Acal(X)_\Tors \otimes_\Z \Z_\ell \times \Acal^t(X)_\Tors \otimes_\Z \Z_\ell \arrow[u,"(\delta{,}\delta^t)" left,"\iso" right] \arrow[r] & \Pic(X)_\Tors \otimes_\Z \Z_\ell \arrow[u,"\delta_2",hookrightarrow]
\end{tikzcd}\]
commutes.
\end{proposition}
\begin{proof}
The pairing in the lower row identifies with $\H^0(X,\Acal) \times \Ext_X^1(\Acal,\G_m) \to \H^1(X,\G_m)$ by~\cref{prop:YonedaExtAndHeightPairing}.  

In the rest of the proof, we show that the following diagram commutes:
\begin{equation} \label{eq:DiagramMitWeilPaarungUndYonedaPaarung}
\begin{gathered}\begin{tikzcd}
    \H^1(X,\Acal[\ell^n]) \arrow[r,"\times",phantom] & \H^1(X,\Acal^t[\ell^n]) \arrow[r,"\cup"] & \H^2(X,\mu_{\ell^n}) \\
    \H^0(X,\Acal) \arrow[u,"\delta"] \arrow[r,"\times",phantom] & \H^0(X,\Acal^t) \arrow[u,"\delta^t"] \arrow[r] & \H^1(X,\G_m) \arrow[u,"\delta'"]
\end{tikzcd}
\end{gathered}
\end{equation}
Here, the pairing in the upper line is induced by the Weil pairing, and the pairing in the lower line is given by~\cref{prop:YonedaExtAndHeightPairing}.  The morphism $\delta'$ is the connecting morphism of the Kummer sequence.  Since $\H^1(X,\Acal[\ell^n])$ is killed by $\ell^n$, $\delta$ factors through $\delta_{\ell^n}$, and analogously for $\delta^t$ and $\delta'$.

By~\cref{prop:YonedaExtAndHeightPairing}, the pairing $\Acal(X) \times \Acal^t(X) \to \Pic(X)$ identifies with
\[
    \H^0(X,\Acal) \times \Ext^1_X(\Acal,\G_m) \to \H^1(X,\G_m).
\]

The diagram
\[\begin{tikzcd}
    \H^0(X,\Acal) \arrow[d,equal]\arrow[r,"\times",phantom] & \Ext^1_X(\Acal,\G_m) \arrow[d,"\delta_{\Ext}"] \arrow[r] & \H^1(X,\G_m) \arrow[d,"\delta'"] \\
    \H^0(X,\Acal) \arrow[r,"\times",phantom] & \Ext^2_X(\Acal,\mu_{\ell^n}) \arrow[r] & \H^2(X,\mu_{\ell^n})
\end{tikzcd}\]
commutes, where the horizontal maps are Yoneda Ext-pairings, by the $\delta$-functoriality~\cite[p.~67, Theorem~1.1]{AltmanKleiman}, so we are left with proving that the lower pairing of this diagram and the upper pairing of the diagram~\eqref{eq:DiagramMitWeilPaarungUndYonedaPaarung} are equal.  In order to show this, we prove the commutativity of 
\begin{equation} \label{eq:DiagrammExt2etc}
\begin{gathered}
\begin{tikzcd}
    \H^0(X,\Acal) \arrow[d,"\delta"] \arrow[r,"\times",phantom] & \Ext^2_X(\Acal,\mu_{\ell^n}) \arrow[r] & \H^2(X,\mu_{\ell^n}) \arrow[d,equal] \\
    \H^1(X,\Acal[\ell^n]) \arrow[r,"\times",phantom] & \H^1(X,\EExt_X^1(\Acal,\mu_{\ell^n})) \arrow[u,hookrightarrow] \arrow[r] & \H^2(X,\mu_{\ell^n}) \\
    \H^1(X,\Acal[\ell^n]) \arrow[r,"\times",phantom] \arrow[u,equal] & \H^1(X,\Acal^t[\ell^n]) \arrow[u,equal] \arrow[r] & \H^2(X,\mu_{\ell^n}) \arrow[u,equal];
\end{tikzcd}
\end{gathered}
\end{equation}
note that $\EExt^1_X(\Acal,\mu_{\ell^n}) = \Acal^t[\ell^n]$ by~\cref{lemma:EExt1AcalGmelln} and use the injection~\eqref{eq:H1EExt1Ext2}.  

By adjunction, rewrite the two upper rows of the diagram~\eqref{eq:DiagrammExt2etc} as
\[\begin{tikzcd}
    \Ext^2_X(\Acal,\mu_{\ell^n}) \arrow[r] & \Hom(\H^0(X,\Acal), \H^2(X,\mu_{\ell^n})) \\
    \H^1(X,\EExt_X^1(\Acal,\mu_{\ell^n})) \arrow[u,hookrightarrow] \arrow[r] & \Hom(\H^1(X,\Acal[\ell^n]), \H^2(X,\mu_{\ell^n})) \arrow[u,"\delta^*"]
\end{tikzcd}\]
with the injectivity by~\eqref{eq:H1EExt1Ext2}.  Now, the low term exact sequence associated to the local-to-global Ext spectral sequence gives an embedding $\H^1(X,\HHom_X(\Acal[\ell^n],\mu_{\ell^n})) \hookrightarrow \Ext^1_X(\Acal[\ell^n],\mu_{\ell^n})$.  But by~\cref{lemma:deltaisaniso}, one has an isomorphism $\delta_1: \H^1(X,\HHom_X(\Acal[\ell^n],\mu_{\ell^n})) \isoto \H^1(X,\EExt_X^1(\Acal,\mu_{\ell^n}))$.
Now, the square in the diagram
\[
\begin{tikzcd}
    & \Ext_X^2(\Acal,\mu_{\ell^n}) \arrow[r] & \Hom(\H^0(X,\Acal), \H^2(X,\mu_{\ell^n})) \\
    & \Ext_X^1(\Acal[\ell^n],\mu_{\ell^n}) \arrow[r] \arrow[u,"\delta"] & \Hom(\H^1(X,\Acal[\ell^n]), \H^2(X,\mu_{\ell^n})) \arrow[u,"\delta^*"] \\
    \H^1(X,\EExt_X^1(\Acal,\mu_{\ell^n})) \arrow[ur,hookrightarrow] \arrow[uur,hookrightarrow] \ar[rru] & &
\end{tikzcd}
\]
commutes by $\delta$-functoriality~\cite[p.~67, Theorem~1.1]{AltmanKleiman} with the injection~\eqref{eq:H1EExt1Ext2}. 
The lower triangle commutes by definition and the upper left triangle by functoriality of the Grothendieck spectral sequence and its low term exact sequence applied to the special case of the local-to-global Ext spectral sequences
\begin{align*}
	E_2^{p,q} = \H^p(X,\EExt_X^q(\Acal,\mu_{\ell^n})) \Rightarrow &E^{p+q} = \Ext_X^{p+q}(\Acal,\mu_{\ell^n})\\
	'{E_2^{p,q}} = \H^p(X,\EExt_X^q(\Acal[\ell^n],\mu_{\ell^n})) \Rightarrow &'E^{p+q} = \Ext_X^{p+q}(\Acal[\ell^n],\mu_{\ell^n})
\end{align*}
defined on derived categories with edge maps $\kappa_1, '\kappa_1$ and the exact triangle $\Acal[\ell^n] \to \Acal \to \Acal \stackrel{+1}{\to} \Acal[\ell^n][1]$ inducing $\HHom_X(\Acal,\mu_{\ell^n}) \isoto \HHom_X(\Acal[\ell^n][1],\mu_{\ell^n})$:
\[
\begin{tikzcd}
	E_2^{1,1} \arrow[r,"\kappa_1"] & E^2\\
	'E_2^{1,0} \arrow[u,"\delta_1 \iso"] \arrow[r,"'\kappa_1"] & 'E^1 \arrow[u,"\delta"] 
\end{tikzcd}
\] 
\end{proof}

The following is commutativity of part~(2) of diagram~\eqref{eq:BigDigram}.

\begin{proposition} \label[proposition]{lemma:cupproductandintersectionproduct}
The diagram
\[\begin{tikzcd}
    \H^2(X,\Z_\ell(1))_\Tors \arrow[rr,"\cup\eta^{d-1}"] && \H^{2d}(X,\Z_\ell(d))_\Tors \arrow[r,"\pr_1^*"] & \H^{2d}(\overline{X},\Z_\ell(d)) \arrow[r,"\iso"] & \Z_\ell \arrow[d,equal]\\
    \CH^1(X)_\Tors \otimes_\Z \Z_\ell \arrow[u,"\delta_2 = \mathrm{cl}_X^1",hookrightarrow] \arrow[rr,"\cap\Ocal_X(1)^{d-1}"] && \CH^d(X)_\Tors \otimes_\Z \Z_\ell \arrow[u,"\mathrm{cl}_X^d"]
    \arrow[r,"\pr_1^*"] & \CH^d(\overline{X})_\Tors \otimes_\Z \Z_\ell \arrow[u,"\mathrm{cl}_{\overline{X}}^d"] \arrow[r,"\deg"] & \Z \otimes_\Z \Z_\ell
\end{tikzcd}\]
commutes.
\end{proposition}
\begin{proof}
Since the category of $\Z_\ell$-modules modulo the Serre subcategory of torsion $\Z_\ell$-modules is equivalent to the category of $\Q_\ell$-modules, see~\cite[Tag 0B0K]{stacks-project}, we can prove the statement after tensoring with $\Q_\ell$.  There is a ring homomorphism
\[
    \mathrm{cl}_{\overline{X}}: \bigoplus_{i=0}^d\CH^i(\overline{X}) \to \bigoplus_{i=0}^d\H^{2i}(\overline{X},\Q_\ell(i)),
\]
see~\cite[p.~270, Proposition~VI.9.5]{MilneEtaleCohomology} (intersection product on the Chow ring and cup product on the cohomology ring) or~\cite[p.~243, Lemma~(6.14)]{JannsenContinuous}, and $\CH^d(\overline{X}) \to \H^{2d}(\overline{X},\Q_\ell(d)) \isoto \Q_\ell$ maps the class of a point to $1$, see~\cite[p.~276, Theorem~VI.11.1\,(a)]{MilneEtaleCohomology}.
\end{proof}

\subsection{Comparison of a Yoneda pairing with the generalised Bloch pairing} \label{subsec:Bloch}
This is a generalisation of~\cite{Schneider} and~\cite{BlochNote}.

Recall that $d = \dim{X}$.  We want to show that the pairing
\begin{equation} \label{eq:AcalXExt1X}
\begin{tikzcd}
    \qu{\cdot,\cdot}: \Acal(X) \times \Ext^1_{X_{\mathrm{fppf}}}(\Acal,\G_m) \arrow[r,"\vee"] & \H^1(X,\G_m) \arrow[rr,"\cap \Ocal_X(1)^{d-1}"] && \CH^d(X) \arrow[r,"\deg"] & \Z
\end{tikzcd}
\end{equation}
(note that the Yoneda pairing $\vee$ identifies with $\Acal(X) \times \Acal^t(X) \to \Pic(X)$ by~\cref{prop:YonedaExtAndHeightPairing}) coincides up to a factor $-\log{q}$ with the generalised Bloch pairing
\begin{equation} \label{eq:BlochPairing}
    h: A(K) \times A^t(K) \to \log{q}\cdot\Z \subset \RR
\end{equation}
from~\cref{def:BlochPairing}:

\begin{proposition} \label[proposition]{lemma:YonedaandBlochpairing}
The diagram
\[\begin{tikzcd}
    A(K) \arrow[r,"\times",phantom] & A^t(K) \arrow[r,"h"] & \RR \\
    A(K) \arrow[r,"\times",phantom] \arrow[u,equal] & \Ext^1_{X_{\mathrm{fppf}}}(\Acal,\G_m) \arrow[u,"\iso"] \arrow[r,"\qu{\cdot{,}\cdot}"] & \arrow[u,"\cdot(-\log{q})",hookrightarrow]\Z
\end{tikzcd}\]
commutes.
\end{proposition}

Note that $\Acal(X) = \Hom_{X_{\mathrm{fppf}}}(\Z,\Acal)$. The Yoneda pairing $\vee: \Hom_{X_{\mathrm{fppf}}}(\Z,\Acal) \times \Ext^1_{X_{\mathrm{fppf}}}(\Acal,\G_m) \to \H^1(X,\G_m)$ maps $(a,a^t)$ to the extension $a \vee a^t$ defined by
\begin{equation}\begin{gathered}\label{eq:aveeavee}\begin{tikzcd}
    a \vee a^t:& 1 \arrow[r] & \G_m \arrow[d,equal] \arrow[r] & \mathscr{Y} \arrow[r] \arrow[d] & \Z \arrow[r] \arrow[d,"a"] & 0\\
           a^t:& 1 \arrow[r] & \G_m            \arrow[r] & \mathscr{X} \arrow[r] & \Acal \arrow[r] & 0.
\end{tikzcd}\end{gathered}\end{equation}
By composition, one gets an extension
\[
    l_{a \vee a^t}: \mathscr{Y}(\A_K) \to \mathscr{X}(\A_K) \stackrel{l_{a^t}}{\to} \RR
\]
of $l: \G_m(\A_K) \to \RR$ to $\mathscr{Y}(\A_K)$, which induces because of $l(\G_m(K)) = 0$ in the exact sequence $a \vee a^t$ by restriction to $\mathscr{Y}(K)$ a homomorphism
\[
    l_{a \vee a^t}: \Z \stackrel{a}{\to} A(K) \stackrel{l_{a^t}}{\to} \RR,
\]
so one obviously has
\begin{equation} \label{eq:haatlaveeat1}
    h(a,a^t) = l_{a^t}(a) = l_{a \vee a^t}(1).
\end{equation}

By~\eqref{eq:commdiaglat} and~\eqref{eq:XprodsubseteqX1}
\[
    l_{a^t}\Big(\prod_{x\in X^{(1)}}\mathscr{X}(\Ocal_{X,x})\Big) = 0, \quad\text{hence } l_{a \vee a^t}\Big(\prod_{x\in X^{(1)}}\mathscr{Y}(\Ocal_{X,x})\Big) = 0,
\]
by the diagram~\eqref{eq:aveeavee} defining $a \vee a^t$.

\begin{lemma} \label[lemma]{lemma:le1}
Let $(1 \to \G_m \to \mathscr{Y} \to \Z \to 0) = e \in \Ext^1_{X_{\mathrm{fppf}}}(\Z, \G_m) = \H^1(X,\G_m) = \Pic{X}$ be a torsor representing $\Lcal \in \Pic{X}$, and let $l_e: \mathscr{Y}(\A_K) \to \RR$ be an extension of $l$ which vanishes on $\prod_{x\in X^{(1)}}\mathscr{Y}(\Ocal_{X,x})$.  Then one has for the homomorphism $l_e: \Z \to \RR$ (since $l_e(\G_m(K)) = l(\G_m(K)) = 0$) defined by restriction to $\mathscr{Y}(K)$:
\[
    l_e(1) = -\log{q} \cdot \deg(\Lcal \cap \Ocal_X(1)^{d-1}),
\]
where $\Ocal_X(1)^{d-1}$ denotes the $(d-1)$-fold self-intersection of $\Ocal_X(1)$.
\end{lemma}
(Note that for every $e$ there is a extension $l_e$ as in the Lemma using the diagram~\eqref{eq:aveeavee} and~\cref{lemma:laextends}.)
\begin{proof}
Considering $e$ as a class of a line bundle $\Lcal$ on $X$, write $Y(\Lcal) := V(\Lcal) \setminus \{\text{$0$-section}\}$ for the $\G_m$-torsor on $X$ defined by $\Lcal$.  Then $e$ is isomorphic to the extension
\begin{align*}
    1 \to \G_m \to \coprod_{n\in\Z} Y(\Lcal^{\otimes n}) \to \Z \to 0.
\end{align*}

For every $x \in |X|$, choose an open neighbourhood $U_x \subseteq X$ such that $1 \in \Z$ has a preimage $s_x \in \mathscr{Y}(U_x)$ (these exist by exactness of the short exact sequence $e$ of sheaves; note that $\Pic(X) = \H^1_\Zar(X,\G_m) = \H^1_{\mathrm{fppf}}(X,\G_m)$ by~\cite[p.~124, Proposition~III.4.9]{MilneEtaleCohomology}, so there is indeed such a \emph{Zariski} neighbourhood, not just an fppf one).  Let further $s \in \mathscr{Y}(K)$ be a preimage of $1 \in \Z$ (note that $0 \to \G_m(K) \to \mathscr{Y}(K) \to \Z(K) \to 0$ is exact by Hilbert 90).  Then one has $s_x^{-1}\cdot s \in \G_m(K) = K^\times$.  Since $X$ is Jacobson, the $U_x$ for $x \in |X|$ cover $X$.  For every $x \in X^{(1)}$ choose an $\tilde{x} \in |X|$ such that $x \in U_{\tilde{x}}$ and set $s_x = s_{\tilde{x}}$ and $U_x = U_{\tilde{x}}$.  These define a Cartier divisor as $(s_x^{-1}\cdot s)\cdot(s_y^{-1}\cdot s)^{-1} = s_x^{-1}\cdot s_y \mapsto 1 - 1 = 0 \in \Z$, so one has $(s_x^{-1}\cdot s)\cdot(s_y^{-1}\cdot s)^{-1} \in \G_m(U_x \cap U_y)$ by the exactness of $1 \to \G_m \to \mathscr{Y} \to \Z \to 0$, and $[(U_x,s_x^{-1})_x] = \Lcal$ since
\[
    \Gamma(U_x,\Ocal_X((U_x,s_x^{-1})) = \{f \in K : fs_x^{-1} \in \Gamma(U_x,\Ocal_X)\} \stackrel{!}{=} \Gamma(U_x,\Lcal).
\]
One has to compare the line bundle $\Lcal$ with the $\G_m$-torsor $\mathscr{Y}$.  Now one calculates
\begin{align*}
    l_e(1) &= l_e(s) \quad\text{note that $l_e(\G_m(K)) = l(\G_m(K)) = 0$ and $s \mapsto 1$}\\
           & = l_e((s_x^{-1}\cdot s)) \quad\text{since $l_e({\textstyle\prod}_{x\in X^{(1)}}\mathscr{Y}(\Ocal_{X,x})) = 0$}\\
           & = l((s_x^{-1}\cdot s)) \quad\text{since $s_x^{-1}\cdot s \in \G_m(K)$ and $l_e$ extends $l$, note that $(s_x^{-1}\cdot s)_x \in \G_m(\A_K)$}\\
           &= -\log{q}\cdot\sum_{x \in X^{(1)}}\deg_\iota{x}\cdot v_x(s_x^{-1}\cdot s) \quad\text{by definition of $l$}.
\end{align*}

On the other hand, by the above description of $e$, since the $(U_x,s_x^{-1}\cdot s)$ define a Cartier divisor on $X$ with associate line bundle isomorphic to $\Lcal$, one has for $\deg: \CH^d(X) \to \Z$
\[
    \deg(\Lcal \cap \Ocal_X(1)^{d-1}) = \sum_{x \in X^{(1)}}\deg_\iota{x}\cdot v_x(s_x^{-1}\cdot s)
\]
since $\deg_{\iota}{x} = \deg(\overline{\{\iota(x)\}} \cap H^{d-1})$ for a generic hyperplane $H \hookrightarrow \PP^N_k$ and $\Ocal_X(1) = [H] \in \CH^1(X) = \Pic(X)$.

Combining the formulae gives the claim.
\end{proof}

Applying~\cref{lemma:le1} to the above situation $a \in A(K), a^t \in A^t(K)$ gives us
\begin{align*}
    h(a,a^t) &= l_{a \vee a^t}(1) \quad\text{by~\eqref{eq:haatlaveeat1}}\\
                &= -\log{q}\cdot\deg([a \vee a^t] \cap \Ocal_X(1)^{d-1}) \quad\text{by~\cref{lemma:le1}}\\
                &= -\log{q}\cdot \qu{a,a^t} \quad\text{by~\eqref{eq:AcalXExt1X}}.
\end{align*}
(Note that $\iota$ and $\Ocal_X(1)$ occur in $l$ and thus in $h$.)  This finishes the proof of~\cref{lemma:YonedaandBlochpairing}.

\subsection{Comparison of the generalised Bloch pairing with the generalised Néron-Tate height pairing} \label{subsec:NTheight}

Let $K_v$ be the (completion) of $K = k(X)$ at $v \in X^{(1)}$, which is a local field.

Let $\Delta$ be a divisor on $A$ defined over $K_v$ algebraically equivalent to $0$ (this corresponds to $\Ext_X^1(\Acal,\G_m) = \Acal^t(K_v) = A^t(K_v) = \PPic^0_{A/k}(K_v)$).  The divisor $\Delta$ corresponds to an extension
\begin{align} \label{eq:GmTorseur}
    1 \to \G_m \to \mathscr{X}_\Delta \to \Acal \to 0
\end{align}
in $\Ext_X^1(\Acal,\G_m)$.  Let $\Lcal_\Delta$ be the line bundle associated to $\Delta$.  Then $\mathscr{X}_\Delta = V(\Lcal_\Delta) \setminus \{0\}$ with $\Lcal_\Delta = \Ocal_\Acal(\Delta)$ as a $\G_m$-torsor.  The extension~\eqref{eq:GmTorseur} only depends on the linear equivalence class of $\Delta$.

Restricting to $K_v$, the extension~\eqref{eq:GmTorseur} is split as a torsor over $A \setminus |\Delta|$ (since a line bundle associated to a divisor $\Delta$ is trivial on $X \setminus \Delta$) by $\sigma_{\Delta,v}: A \setminus |\Delta| \to \mathscr{X}_{\Delta,K_v}$ with $\sigma_\Delta$ canonical up to translation by $\G_m(K_v)$ (since the choice of $\sigma_{\Delta,v}$ is the same as the choice of a rational section of $\Lcal_\Delta$).  Let $Z_{\Delta,K_v}$ be the group of zero cycles $\mathfrak{A} = \sum n_i(p_i)$, $p_i \in A(K_v)$, on $A$ defined over $K_v$ such that $\sum n_i \deg{p_i} = 0$ and $\supp\mathfrak{A} \subseteq A \setminus |\Delta|$.  We get a homomorphism $\sigma_{\Delta,v}: Z_{\Delta,K_v} \to \mathscr{X}_\Delta(K_v)$ (since $\mathscr{X}_\Delta$ is a group scheme).

We now prove a local analogue of~\cref{lemma:laextends}.
\begin{lemma} \label[lemma]{lemma:locallaextends}
There is a commutative diagram with exact rows and columns:
\[\begin{tikzcd}
    & 1 \arrow[d] & 0 \arrow[d] & & \\
    1 \arrow[r] & \Ocal_{K_v}^\times \arrow[r] \arrow[d] & \mathscr{X}_\Delta(\Ocal_{K_v}) \arrow[r] \arrow[d] & A(K_v) \arrow[d,equal] \arrow[r] & 0\\
    1 \arrow[r] & K_v^\times \arrow[d,"l_v"] \arrow[r] & \mathscr{X}_\Delta(K_v) \arrow[d,"\psi_{\Delta{,}v}",dashrightarrow] \arrow[r] & A(K_v) \arrow[r] & 0\\
    & \Z \arrow[d] \arrow[r,dashed,equal] & \Z \arrow[d] & & \\
    & 0 & 0 & &
\end{tikzcd}\]
\end{lemma}
\begin{proof}
The map labelled $l_v$ is the valuation map.  The short exact sequence in the middle row is~\eqref{eq:GmTorseur} evaluated at $K_v$, and the short exact sequence in the upper row is~\eqref{eq:GmTorseur} evaluated at $\Ocal_{K_v}$:  One is left showing that $\mathscr{X}^1_\Delta \twoheadrightarrow A(K_v)$ is surjective.  But this follows from the long exact sequence associated to the short exact sequence of sheaves on $\Ocal_{K_v}$
\[
    1 \to \G_m \to \mathscr{X} \to \Acal \to 0
\]
and Hilbert's theorem 90: $\H^1(\Spec{\Ocal_{K_v}},\G_m) = 0$ since $\Ocal_{K_v}$ is a local ring.  Further, one has $\Acal(K_v) = \Acal(\Ocal_{K_v}) = A(K_v)$ by the valuative criterion of properness and the Néron mapping property.
\end{proof}

Now let $\psi_{\Delta,v}: \mathscr{X}_\Delta(K_v) \to \Z$ be the map defined in the previous lemma.  For $\mathfrak{A} \in Z_{\Delta,K_v}$ define
\[
    \qu{\Delta,\mathfrak{A}}_v := \psi_{\Delta,v}\sigma_{\Delta,v}(\mathfrak{A}).
\]

\begin{theorem}
Let $K = k(X)$.  Let $a \in A(K)$ and $a^t \in A^t(K)$.  Let $\Delta$ resp.\ $\mathfrak{A}$ be a divisor algebraically equivalent to $0$ defined over $K$ resp.\ a zero cycle of degree $0$ over $K$ on $A$ such that $[\Delta] = a^t$ resp.\ $\mathfrak{A}$ maps to $a$.  Assume $\supp\Delta$ and $\supp\mathfrak{A}$ disjoint.  Then
    \[
        \qu{a,a^t} = \log{q}\cdot\sum_{v \in X^{(1)}}\qu{\Delta,\mathfrak{A}}_v
    \]
with $\qu{a,a^t}$ defined as in~\cref{def:BlochPairing}.
\end{theorem}
\begin{proof}
Let
\[
    1 \to \G_m \to \mathscr{X}_\Delta \to \Acal \to 0
\]
be the $\G_m$-torsor represented by $a^\vee$, and $\sigma_{\Delta,v}: Z_{\Delta,K} \to \mathscr{X}_\Delta(K)$ be as in the local case.  One has to show that the map
\begin{align*}
    l_{a^t}: \mathscr{X}_\Delta(\A_K) \to \log{q}\cdot\Z
\end{align*}
(of the above definition in~\cref{lemma:laextends}; note that $\mathscr{X}_\Delta(\A_K) \hookrightarrow \mathscr{X}(\A_K)$) coincides with the sum of the local maps
\[
    \psi_{\Delta,v}: \mathscr{X}_\Delta(K_v) \to \Z
\]
multiplied by $\log{q}\cdot\deg_\iota{v}$ for $v \in X^{(1)}$ defined above.

Consider the commutative diagram
\[\begin{tikzcd}
    & \G_m^1 \arrow[d,"\sum_v \deg_\iota{v}\cdot v(\cdot)"] \arrow[r] & \mathscr{X}_\Delta(\A_K)/\prod_v\mathscr{X}_{\Delta}(\Ocal_{K_v}) \arrow[d,"\sum_v \deg_\iota{v}\cdot\psi_{\Delta{,}v}"] \arrow[r] & \mathscr{X}_\Delta(\A_K)/\mathscr{X}_\Delta^1 \arrow[d,"l_{a^t}"] \arrow[r] &  0 \\
    0 \arrow[r] & \ker(\sum) \arrow[r] & \bigoplus_{v \in X^{(1)}}\Z \arrow[r,"\log{q}\cdot\sum"] & \log{q}\cdot\Z \arrow[r] & 0.
\end{tikzcd}\]

One has $\prod_v\mathscr{X}_{\Delta}(\Ocal_{K_v}) \subseteq \mathscr{X}_\Delta^1$ and $\G_m^1 \subseteq \mathscr{X}_\Delta(\A_K)$ by the commutative diagram~\eqref{lemma:locallaextends} in~\cref{lemma:laextends}, so the exactness of the upper row follows.  The exactness of the lower row is clear.

The left square commutes obviously.  The right square commutes since by~\cref{lemma:laextends} the extension of $l$ to $\mathscr{X}(\A_K)$ is unique and $\sum_v \deg_\iota{v}\cdot\psi_{\Delta,v}$ is well-defined (since $\psi_{\Delta,v}$ vanishes on $\mathscr{X}_\Delta(\Ocal_{K_v})$ and in the adele ring, almost all components lie in $\mathscr{X}_\Delta(\Ocal_{K_v})$) and restricts to $l: \G_m(\A_K) \to \RR$ since $\psi_{\Delta,v}$ restricts to $l_v$ by~\cref{lemma:locallaextends}.

Now let $x \in \mathscr{X}_\Delta(\A_K)/\mathscr{X}_\Delta^1$ with $l_{a^t}(x) = h$.  Lift it to $\tilde{x} \in \mathscr{X}_\Delta(\A_K)/\prod_v\mathscr{X}_{\Delta}(\Ocal_{K_v})$.  Then $\log{q}\cdot\sum_v \deg_\iota{v} \cdot \psi_{\Delta,v}(\tilde{x}) = h$ by commutativity of the right square.  If one chooses another lift, their difference comes from $d \in \G_m^1$, which has height $0$, so $\log{q}\cdot\sum_v \deg_\iota{v}\cdot \psi_{\Delta,v}(\tilde{x})$ only depends on $x$.
\end{proof}

\begin{proposition}
The local pairings $\qu{\Delta,\mathfrak{A}}_v$ coincide with the local Néron height pairings $\qu{\Delta,\mathfrak{A}}_{\mathrm{N\acute{e}ron},v}$.
\end{proposition}
\begin{proof}
This follows from Néron's axiomatic characterisation in~\cite[p.~304\,f., Theorem~9.5.11]{BombieriGubler}, which holds for $\qu{\Delta,\mathfrak{A}}_v$ by the same argument as in~\cite[p.~73\,ff., (2.11)--(2.15)]{BlochNote}.
\end{proof}

\begin{corollary} \label[corollary]{cor:BlochpairingandNeronTatepairing}
The generalised Bloch pairing coincides with the canonical Néron-Tate height pairing.
\end{corollary}
\begin{proof}
This is clear since the local Néron-Tate height pairings sum up to the canonical Néron-Tate height pairing, see~\cite[p.~307, Corollary~9.5.14]{BombieriGubler}.
\end{proof}

\subsection{Conclusion} \label{subsec:Conclusion}

We are now ready to combine the main results of subsections~\ref{subsect:Yoneda}, \ref{subsec:Bloch} and~\ref{subsec:NTheight} into

\begin{proposition} \label[proposition]{thm:bigdiagramPairings}
	Let $\ell$ be invertible on $X$ and assume $\Sha(\Acal/X)[\ell^\infty]$ is finite.  Then there is a commutative diagram 
	{\footnotesize
		\begin{equation}\begin{tikzcd}
		\H^1(X, T_\ell\Acal)_\Tors  \times \H^{2d-1}(X, T_\ell(\Acal^t)(d-1))_\Tors \arrow[drr,"(0)",phantom]\arrow[rr,"\cup"] && \H^{2d}(X,\Z_\ell(d))_\Tors \arrow[r,"\pr_1^*"]\arrow[d,equal]& \H^{2d}(\overline{X},\Z_\ell(d)) \arrow[r,"\iso"]\arrow[d,equal] &\Z_\ell \arrow[d,equal]\\
		\H^1(X, T_\ell\Acal)_\Tors  \times \H^1(X, T_\ell(\Acal^t))_\Tors \arrow[dr,"(1)",phantom]\arrow[u,"\id \times (\cup\eta^{d-1})",hookrightarrow] \arrow[r,"\cup"] &  \H^2(X,\Z_\ell(1))_\Tors \arrow[drrr,"(2)",phantom]\arrow[r,"(\cup\eta)^{d-1}"]& \H^{2d}(X,\Z_\ell(d))_\Tors \arrow[r,"\pr_1^*"]& \H^{2d}(\overline{X},\Z_\ell(d)) \arrow[r,"\iso"] &\Z_\ell \arrow[d,equal]\\
		\Acal(X)_\Tors \otimes_\Z \Z_\ell \times \Acal^t(X)_\Tors \otimes_\Z \Z_\ell \arrow[drrr,"(3)",phantom]\arrow[u,"(\delta{,}\delta^t)" left,"\iso" right] \arrow[r,"\vee"] & \H^1(X,\G_m)_\Tors \otimes_\Z \Z_\ell \arrow[u,hookrightarrow]\arrow[rr,"\cap\Ocal_X(1)^{d-1}"]&& \CH^d(X)_\Tors \otimes_\Z \Z_\ell \arrow[d]\arrow[r,"\deg"]&\Z_\ell \arrow[d,equal]\\
		\Acal(X)_\Tors \otimes_\Z \Z_\ell \times \Acal^t(X)_\Tors \otimes_\Z \Z_\ell \arrow[u,equal] \arrow[rrr,"\hat{h}_\Pcal(\cdot{,}\cdot)"]&&&\Z \otimes_\Z \Z_\ell \arrow[r,"\iso"] &\Z_\ell.
		\end{tikzcd}\label{eq:BigDigram}\end{equation}}Here, $\eta \in \H^2(X, \Z_\ell(1))$ is the cycle class associated to $\Ocal_X(1) \in \Pic(X) = \CH^1(X)$ ($X$ is regular) by $\Pic(X) \to \H^2(X, \Z_\ell(1))$ (this map comes from the Kummer sequence, see~\cite[p.~271, Proposition~VI.10.1]{MilneEtaleCohomology}), where $\Ocal_X(1) = \iota^*\Ocal_{\PP^N_K}(1)$ for the closed immersion $\iota: X \hookrightarrow \PP^N_K$ which defines the structure of a generalised global field on the function field $K = k(X)$ of $X$.  Further,
	\[
	\vee: \Acal(X)_\Tors \otimes_\Z \Z_\ell \times \Acal^t(X)_\Tors \otimes_\Z \Z_\ell = \Acal(X)_\Tors \otimes_\Z \Z_\ell \times \Ext_X^1(\Acal,\G_m)_\Tors \otimes_\Z \Z_\ell \to \H^1(X,\G_m)_\Tors \otimes_\Z \Z_\ell
	\]
	is the Yoneda Ext-pairing (the equality $\Acal^t(X) = \Ext_X^1(\Acal,\G_m)$ comes from the Barsotti-Weil formula).  The pairing in the lower row is the generalised Néron-Tate canonical height pairing divided by $-\log{q}$.  The left vertical isomorphism $(\delta,\delta^t)$ comes from~\eqref{eq:deltaboundarymapabelianscheme}, and the injection $\cup\eta^{d-1}$ from~\cref{thm:integralhardLefschetz}.
\end{proposition}
\begin{proof}
Diagram~(0) commutes by associativity of the $\cup$-product.
Diagram~(1) commutes by~\cref{lemma:commutativityof2firstpart} and~(2) by~\cref{lemma:cupproductandintersectionproduct}.  Diagram~(3) commutes by~\cref{lemma:YonedaandBlochpairing} and by~\cref{cor:BlochpairingandNeronTatepairing}.
\end{proof}

The above proposition and the definition of the integral hard Lefschetz defect (\cref{def:integralhardLefschetzdefect}) yield the following statement, which is the main theorem of this section:
\begin{theorem} \label[theorem]{thm:comparisonofpairings}
	The \emph{cohomological pairing} $\qu{\cdot,\cdot}_\ell$ from~\cref{thm:BSDI} equals the \emph{generalised Bloch pairing} (see~\cref{def:BlochPairing}) and the \emph{canonical Néron-Tate height pairing} (see~\cref{def:generalisedNeronTatecanonicalheightpairing}) up to multiplication by the integral hard Lefschetz defect (see~\cref{def:integralhardLefschetzdefect}).
\end{theorem}

\begin{remark} \label[remark]{rem:HeightPairingAndCohPairing}
	Note that the cohomological pairing $\qu{\cdot,\cdot}_\ell$ does not depend on an embedding $\iota: X \hookrightarrow \PP^N_k$, but all other pairings in~\eqref{eq:BigDigram} depend on a line bundle $\theta$ or cohomology class $\eta \in \H^2(X,\Z_\ell(1))$, which manifests in the integral hard Lefschetz defect in the commutative square~(0).  The two choices, in the integral hard Lefschetz defect in the commutative square~(0) and in the other pairings, cancel.
\end{remark}

Here is an example where the integral hard Lefschetz morphism is an isomorphism:
\begin{theorem} \label[theorem]{thm:integralHardLefschetzDefectForAbelianVarieties}
	Let $A/k$ be an Abelian variety of dimension $d$ over an algebraically closed field of characteristic $\neq \ell$ with principal polarisation associated to $\Lcal \in \Pic(A)$.  Denote by $\theta \in \H^2(A,\Z_\ell(1))$ the image of $\Lcal$ under the homomorphism $\Pic(A) \to \H^2(A,\Z_\ell(1))$.  Then the integral hard Lefschetz morphism $(\cup\theta)^{d-1}: \H^{1}(A,\Z_\ell) \to \H^{2d-1}(A,\Z_\ell(d-1))$ is an isomorphism.
\end{theorem}
\begin{proof}
	Using that $\theta$ is a principal polarisation, write $\theta = \sum_{i=1}^d e_i \wedge e_i'$ in a symplectic basis (with respect to the Weil pairing $\wedge: T_\ell{A} \times T_\ell(A^t) \to \Z_\ell(1)$; using the principal polarisation $A \isoto A^t$, the Weil pairing becomes a symplectic pairing $T_\ell A \times T_\ell A \to \Z_\ell(1)$ by~\cite[p.~132, Lemma~16.2\,(e)]{MilneAbelianVarieties}) and use that the cohomology ring $\H^*(A,\Z_\ell) = \bigwedge^*\H^1(A,\Z_\ell)$ is an exterior algebra.
	
	By~\cite[p.~130]{MilneAbelianVarieties}, one has $\H^*(A,\Z_\ell) = (\bigwedge^*T_\ell{A})^\vee$ (here we use that the ground field is algebraically closed).
	
	Note that, via the identifications of the cohomology ring with the exterior algebra, proving that $(\cup\theta)^{d-1}$ is an isomorphism is equivalent to showing that this morphism sends a basis of $\bigwedge^1T_\ell{A}$ to a basis of $\bigwedge^{2d-1}T_\ell{A}$.  A basis of $\bigwedge^1T_\ell{A}$ is $e_1, e_1', \ldots, e_d, e_d'$, and a basis of $\bigwedge^{2d-1}T_\ell{A}$ is (a hat denotes the omission of a term)
	\[
	e_1 \wedge e_1' \wedge \ldots \widehat{e_i} \wedge e_i' \wedge \ldots \wedge e_d \wedge e_d'
	\]
	and the same for $e_i'$ instead of $e_i$.  Now,
	\[
	\theta^{d-1} = \sum_{i=1}^d(e_1 \wedge e_1' \wedge \ldots \widehat{e_i \wedge e_i'} \wedge \ldots \wedge e_d \wedge e_d').
	\]
	Thus,
	\[
	e_i' \wedge \theta^{d-1} = e_1 \wedge e_1' \wedge \ldots \widehat{e_i} \wedge e_i' \wedge \ldots \wedge e_d \wedge e_d'
	\]
	and the same for $e_i$, gives a basis of $\bigwedge^{2d-1}T_\ell{A}$.
\end{proof}

\begin{corollary} \label[corollary]{cor:integralHardLefschetzDefectConstantAV}
	Let $\Acal = B \times_k X$ be a constant Abelian scheme with $X = A$ a principally polarised Abelian variety of dimension $d$ over an algebraically closed field $k$.  Then the integral hard Lefschetz morphism $(\cup\theta)^{d-1}: \H^{1}(\Acal,\Z_\ell) \to \H^{2d-1}(\Acal,\Z_\ell(d-1))$ is an isomorphism.
\end{corollary}
\begin{proof}
	Note that $\H^1(A,T_\ell\Acal) = \H^1(A, \Z_\ell) \times T_\ell B$ by~\cref{lemma:BasicPropertiesOfConstantAbelianSchemes}\,2 and the projection formula.
\end{proof}

\begin{corollary}
	Let $\Acal = B \times_k X$ be a constant Abelian scheme over $X$ with $X = A$ a principally polarised Abelian variety of dimension $d$ over finite field $k$.  Then over the maximal $\ell$-extension $k_{\ell^\infty}$, the integral hard Lefschetz morphism $(\cup\theta)^{d-1}: \H^{1}(\Acal_{k_{\ell^\infty}},\Z_\ell) \to \H^{2d-1}(\Acal_{k_{\ell^\infty}},\Z_\ell(d-1))$ is an isomorphism.
	
	Furthermore, for some finite $\ell$-extension $K/k$ the integral hard Lefschetz morphism $(\cup\theta)^{d-1}: \H^{1}(\Acal_K,\Z_\ell) \to \H^{2d-1}(\Acal_K,\Z_\ell(d-1))$ is an isomorphism.
\end{corollary}
\begin{proof}
	By~\cite[p.~259, Proposition~5.9.2\,(iii)]{LeiFuEtaleCohomologyTheory}, the integral hard Lefschetz homomorphism over $\overline{k}$ is the direct limit over all $K/k$ finite.  The transition morphisms are injective since $\cores_{K/k} \circ \res_{K/k} = [K:k]$ is injective, and isomorphisms for $\ell \nmid [K:k]$ since then multiplication by $[K:k]$ is an isomorphism, in particular surjective.  It follows that the integral hard Lefschetz morphism is an isomorphism over the maximal $\ell$-extension $k_{\ell^\infty}$ of $k$.
	
	Since the integral hard Lefschetz morphism over $k_{\ell^\infty}$ is the filtered direct limit over the base changes of $(\cup\theta)^{d-1}$ of the finite $\ell$-extensions of $k$ and since $\H^{2d-1}(\Acal_{k_{\ell^\infty}},\Z_\ell(d-1))$ is a finitely generated $\Z_\ell$-module, there is a finite $\ell$-extension $K/k$ such that $(\cup\theta)^{d-1}: \H^{1}(\Acal_K,\Z_\ell) \hookrightarrow \H^{2d-1}(\Acal_K,\Z_\ell(d-1))$ is surjective~\cite[p.~46, (5.2.3))]{EGAI}, hence an isomorphism.
\end{proof}

\section{The determinant of the pairing $(\cdot,\cdot)_\ell$} \label{subsection:pairingroundbrackets}
\begin{lemma}
Assume $\Sha(\Acal/X)[\ell^\infty]$ is finite.  Then one has a commutative diagram with exact columns
\begin{equation} \label{eq:bigdiagramforroundpairing}
\begin{gathered}
\begin{tikzcd}
        & 0 \arrow[d] & 0 \arrow[d] & 0 \arrow[d] & \\
& (\Acal^t(X)\otimes\Z_\ell)_\Tors \arrow[d]\arrow[r,"\iso"] & \H^1(X,T_\ell\Acal^t)_\Tors \arrow[d]\arrow[r,"(\cup\eta_{T_\ell})^{d-1}",hookrightarrow] & \H^{2d-1}(X,T_\ell(\Acal^t)(d-1))_\Tors \arrow[d]& \\
        & \Acal^t(X) \otimes \Q_\ell \arrow[d]\arrow[r,"\iso"] & \H^1(X,V_\ell\Acal^t) \arrow[d]\arrow[r,"(\cup\eta)^{d-1}" above,"\iso" below] & \H^{2d-1}(X,V_\ell(\Acal^t)(d-1))\arrow[d] & \\
        & \Acal^t(X) \otimes \Q_\ell/\Z_\ell \arrow[d]\arrow[r,"\iso"] & \H^1(X,\Acal^t[\ell^\infty])_\divv \arrow[d]\arrow[r,"(\cup\overline{\eta})^{d-1}",twoheadrightarrow] & \H^{2d-1}(X,\Acal^t[\ell^\infty](d-1))_\divv \arrow[d] \\
        & 0 & 0 & 0 &
\end{tikzcd}
\end{gathered}
\end{equation}
with the cokernel of $\H^1(X,T_\ell\Acal^t)_\Tors \hookrightarrow \H^{2d-1}(X,T_\ell(\Acal^t)(d-1))_\Tors$ being finite.
\end{lemma}
\begin{proof}
The upper left arrow is an isomorphism by~\cref{lemma:ShaFinite}.  For the lower left arrow being an isomorphism:  By~\cref{lemma:AcalExactSequence}, one has a short exact sequence
\[
    0 \to \Acal(X) \otimes \Q_\ell/\Z_\ell \to \H^1(X,\Acal[\ell^\infty]) \to \H^1(X,\Acal)[\ell^\infty] \to 0.
\]
Since $\Acal(X) \otimes \Q_\ell/\Z_\ell$ is divisible, one gets an inclusion $\Acal(X) \otimes \Q_\ell/\Z_\ell \hookrightarrow \H^1(X,\Acal[\ell^\infty])_\divv$.  Since $\Sha(\Acal/X)[\ell^\infty] = \H^1(X,\Acal)[\ell^\infty]$ is finite, if an element from $\H^1(X,\Acal[\ell^\infty])_\divv$ is mapped to $\H^1(X,\Acal)[\ell^\infty]$, it has finite order and is divisible, so it is $0$, hence it comes from $\Acal(X) \otimes \Q_\ell/\Z_\ell$.

The upper and middle right arrows are induced by the integral hard Lefschetz theorem~\cref{thm:integralhardLefschetz} (injective) and the hard Lefschetz theorem~\cref{thm:hardLefschetzFiniteGroundField} (isomorphism), respectively, and the lower one by functoriality of the $\coker$-functor.  So the lower one surjective by the snake lemma.

For the exactness of the columns:  Left column:  This column arises from tensoring
\[
    0 \to \Z_\ell \to \Q_\ell \to \Q_\ell/\Z_\ell \to 0
\]
with $\Acal^t(X)_\Tors \iso \Z^{\rk{\Acal^t(X)}}$ over $\Z$.  (By the theorem of Mordell-Weil~\cref{thm:MordellWeil} and the Néron mapping property~\cref{thm:NeronModel}, $\Acal(X)$ is a finitely generated Abelian group).  Middle and right column:  This follows from~\cref{lemma:lesGarbensequenz}.
\end{proof}

\begin{lemma} \label[lemma]{lemma:H2dminusinjectivewithfinitecokernel}
The homomorphisms induced by the commutative diagram~\eqref{eq:bigdiagramforroundpairing}
\begin{align*}
    \Hom(\H^{2d-1}(X,T_\ell(\Acal^t)(d-1))_\Tors, \Z_\ell) &\to \Hom((\Acal^t(X)\otimes\Z_\ell)_\Tors, \Z_\ell)\quad\text{and}\\
    \Hom(\H^{2d-1}(X,\Acal^t[\ell^\infty](d-1)), \Q_\ell/\Z_\ell)_\divv &\to \Hom(\Acal^t(X) \otimes \Q_\ell/\Z_\ell, \Q_\ell/\Z_\ell)
\end{align*}
are injective with finite cokernels of the same order (even isomorphic).
\end{lemma}
\begin{proof}
Write
\begin{equation} \label{eq:diagAstrichB}
\begin{gathered}
\begin{tikzcd}
    0 \arrow[r] & A' \arrow[r]\arrow[d,"f",hookrightarrow] & A \arrow[r]\arrow[d,"\iso"] & A'' \arrow[r]\arrow[d,"g",twoheadrightarrow] & 0 \\
    0 \arrow[r] & B' \arrow[r] & B \arrow[r] & B'' \arrow[r] & 0
\end{tikzcd}
\end{gathered}
\end{equation}
in short for two right columns of the big diagram~\eqref{eq:bigdiagramforroundpairing}: $A' = \H^1(X,T_\ell\Acal^t)_\Tors$, $A = \H^1(X,V_\ell\Acal^t)$, $A'' = \H^1(X,\Acal^t[\ell^\infty])_\divv$ for the middle column and $B', B, B''$ for the corresponding groups in the right column.

The snake lemma gives us $\ker(g) \isoto \coker(f)$ since the middle vertical arrow in~\eqref{eq:diagAstrichB} is an isomorphism.

Applying $\Hom(-,\Z_\ell)$ to the short exact sequence $0 \to A' \to B' \to \coker(f) \to 0$ gives
\[
    0 \to \Hom(\coker(f),\Z_\ell) \to \Hom(B',\Z_\ell) \to \Hom(A',\Z_\ell) \to \Ext^1(\coker(f),\Z_\ell) \to \Ext^1(B',\Z_\ell).
\]
Since $\coker(f)$ is finite, the first term vanishes and $\Ext^1(\coker(f),\Z_\ell) \iso \coker(f)$, and since $B'$ is torsion-free and finitely generated, hence projective, the last term vanishes.  So $\Hom(f,\Z_\ell)$ is injective with finite cokernel isomorphic to $\coker(f)$.

Applying the exact functor $\Hom(-,\Q_\ell/\Z_\ell)$ ($\Q_\ell/\Z_\ell$ is divisible, hence injective) to the short exact sequence $0 \to \ker(g) \to A'' \to B'' \to 0$ gives
\[
    0 \to \Hom(B'',\Q_\ell/\Z_\ell) \to \Hom(A'',\Q_\ell/\Z_\ell) \to \Hom(\ker(g),\Q_\ell/\Z_\ell) \to 0
\]
and $\Hom(\ker(g),\Q_\ell/\Z_\ell) \iso \ker(g)$ since $\ker(g) \isoto \coker(f)$ is a finite $\ell$-primary group.  So $\Hom(g,\Q_\ell/\Z_\ell)$ is injective with finite cokernel isomorphic to $\ker(g)$.
\end{proof}

\begin{lemma} \label[lemma]{lemma:H2undH2d-1iso}
One has an isomorphism
\begin{equation} \label{eq:H2H2d-1iso}
    \H^2(X,T_\ell\Acal)_\Tors \isoto \Hom(\H^{2d-1}(X,\Acal^t[\ell^\infty](d-1))_\divv, \Q_\ell/\Z_\ell)
\end{equation}
induced by the cup product.
\end{lemma}
\begin{proof}
Poincaré duality for the absolute situation~\cite[p.~183, Corollary~V.2.3]{MilneEtaleCohomology} (easily generalised to higher dimensions) gives non-degenerate pairings of finite groups for all $n \in \N$
\[
    \H^2(X,\Acal[\ell^n]) \times \H^{2d-1}(X,\Acal^t[\ell^n](d-1)) \to \Q_\ell/\Z_\ell.
\]
This is the same as isomorphisms
\[
    \H^2(X,\Acal[\ell^n]) \isoto \Hom(\H^{2d-1}(X,\Acal^t[\ell^n](d-1)), \Q_\ell/\Z_\ell),
\]
and passing to the projective limit gives us an isomorphism
\[
    \H^2(X,T_\ell\Acal) \isoto \Hom(\H^{2d-1}(X,\Acal^t[\ell^\infty](d-1)), \Q_\ell/\Z_\ell).
\]
Write $M = \H^2(X,T_\ell\Acal)$ and $N = \H^{2d-1}(X,\Acal^t[\ell^\infty](d-1))$, so one has $M \isoto N^D$.  These are finitely and cofinitely generated, respectively.  One has
\[
    M_\Tors = N^D/\varinjlim_n N^D[\ell^n] = N^D/(\varprojlim_n N/\ell^n)^D = N^D/\hat{N}^D
\]
since $0 \to N_\divv \to N \stackrel{h}{\to} \hat{N}$ is exact with $\hat{N}$ the $\ell$-adic completion of $N$.  As $N \iso (\Q_\ell/\Z_\ell)^r \oplus T$ with $T$ finite, one has $\hat{N} \iso T$ (since the $\ell$-adic completion of the divisible group $\Q_\ell/\Z_\ell$ is trivial) and $h$ surjective.  Dualising gives $0 \to \hat{N}^D \to N^D \to N_\divv^D \to 0$, so $N^D/\hat{N}^D = N_\divv^D$.  Summing up, we get $M_\Tors = N_\divv^D$.
\end{proof}

\begin{theorem} \label[theorem]{thm:rundeKlammerPaarungDeterminante1}
Assume $\Sha(\Acal/X)[\ell^\infty]$ is finite.  Then one has $\det(\cdot,\cdot)_\ell = 1$ for the pairing
\[
	(\cdot,\cdot)_\ell: \H^2(X, T_\ell\Acal)_\Tors \times \H^{2d-1}(X, T_\ell(\Acal^t)(d-1))_\Tors \to \H^{2d+1}(X,\Z_\ell(d)) = \Z_\ell
\]
from~\eqref{eq:Regulator runde Klammer}.
\end{theorem}
\begin{proof}
Consider the commutative diagram
\begin{equation} \label{eq:Hom2d-1commdiag}
\begin{gathered}
\begin{tikzcd}
    \Hom(\H^{2d-1}(X,T_\ell(\Acal^t)(d-1))_\Tors, \Z_\ell) \arrow[r,hookrightarrow] &\Hom((\Acal^t(X)\otimes\Z_\ell)_\Tors, \Z_\ell) \arrow[dd,"\iso"]\\
    \H^2(X,T_\ell\Acal)_\Tors \arrow[u] \arrow[d,"\iso"] & \\
    \Hom(\H^{2d-1}(X,\Acal^t[\ell^\infty](d-1))_\divv, \Q_\ell/\Z_\ell) \arrow[r,hookrightarrow] & \Hom(\Acal^t(X) \otimes \Q_\ell/\Z_\ell, \Q_\ell/\Z_\ell)
\end{tikzcd}
\end{gathered}
\end{equation}
with the lower left isomorphism by~\eqref{eq:H2H2d-1iso}.  The horizontal maps are injective with cokernels finite of the same order by~\cref{lemma:H2dminusinjectivewithfinitecokernel}.

The right vertical map is an isomorphism:  A homomorphism $h: \Z_\ell \to \Z_\ell$ induces a morphism $\Q_\ell \to \Q_\ell$ by tensoring with $\Q$ and hence a morphism between the cokernels $\Q_\ell/\Z_\ell \to \Q_\ell/\Z_\ell$.  This is an isomorphism:  By the Mordell-Weil theorem and the Néron mapping property~\cref{thm:NeronModel}, $(\Acal^t(X)\otimes\Z_\ell)_\Tors \iso \Z_\ell^{\rk{\Acal^t(X)}}$ and $\Acal^t(X) \otimes \Q_\ell/\Z_\ell \iso (\Q_\ell/\Z_\ell)^{\rk \Acal^t(X)}$, and $\Hom(\Q_\ell/\Z_\ell,\Q_\ell/\Z_\ell) = (\Q_\ell/\Z_\ell)^D = (\varinjlim_n \frac{1}{\ell^n}\Z/\Z)^D = \varprojlim_n \Z/\ell^n\Z = \Z_\ell$.

It follows from Poincaré duality for the absolute situation~\cite[p.~183, Corollary~V.2.3]{MilneEtaleCohomology} that one has a non-degenerate pairing
\[
    (\cdot,\cdot)_\ell: \H^2(X,T_\ell\Acal)_\Tors \times \H^{2d-1}(X,T_\ell(\Acal^t)(d-1))_\Tors \stackrel{\cup}{\to} \H^{2d+1}(X,\Z_\ell(d)) = \Z_\ell,
\]
so the upper left vertical arrow in~\eqref{eq:Hom2d-1commdiag}
\[
    \H^2(X,T_\ell\Acal)_\Tors \to \Hom(\H^{2d-1}(X,T_\ell(\Acal^t)(d-1))_\Tors, \Z_\ell)
\]
is injective with cokernel of order $\det(\cdot,\cdot)_\ell$.  By comparison of the terms in the commutative diagram~\eqref{eq:Hom2d-1commdiag} and using that the horizontal morphisms are injective with cokernels of the same order, see~\cref{lemma:H2dminusinjectivewithfinitecokernel}, it follows that $\det(\cdot,\cdot)_\ell = 1$.
\end{proof}

\section{Proof of the conjecture for constant Abelian schemes} \label{sec:IsoconstantAbelianScheme}

\subsection{The case of a basis of arbitrary dimension}

\begin{lemma} \label[lemma]{lemma:BasicPropertiesOfConstantAbelianSchemes}
Let $A$ be an Abelian variety over a finite field $k$, $X/k$ be a variety and $\Acal = A \times_k X$ be a constant Abelian scheme over $X$.

1.  There is an isomorphism $\Acal[m] \isoto A[m] \times_k X$ of finite flat group schemes resp.\ of constructible sheaves (for $\Char{k} \nmid m$) on $X$.

2.  There is an isomorphism $T_\ell\Acal = (T_\ell A) \times_k X$ of $\ell$-adic sheaves on $X$ for $\ell \neq p$.

3.  There is an isomorphism of Abelian groups
    \[
        \Acal(X) = \Mor_X(X,\Acal) \isoto \Mor_k(X, A), (f: X \to \Acal) \mapsto \pr_1 \circ f,
    \]
    and under this isomorphism $\Acal(X)_\tors$ corresponds to the subset of constant morphisms
    \[
        \Acal(X)_\tors \isoto \{f: X \to A \mid f(X) = \{a\}\} = \Hom_k(k,A) = A(k).
    \]
\end{lemma}
\begin{proof}
1.  Consider the fibre product diagram
\[\begin{tikzcd}
    A[m] \arrow[r] \arrow[d,hookrightarrow] & \Spec{k} \arrow[d,"0", hookrightarrow] \\
    A    \arrow[r,"{[m]}"]        & A
\end{tikzcd}\]
and apply $- \times_k X$.

2.  This follows from 1 by passing to the inverse limit over $m = \ell^n, n \in \N$.

3.  The inverse is given by $(f: X \to A) \mapsto ((f, \id_X): X \to A \times_k X = \Acal)$.

    For the second statement:  If $f: X \to A$ takes on the constant value $a$, $(f, \id_X)$ has finite order $\ord{a}$ in $A(k)$ since $k$ and thus $A(k)$ is finite.  Conversely, if $f: X \to \Acal$ has finite order $n$, the image of $\pr_1 \circ f$ lies in the discrete set of $n$-torsion points (since $\pr_1: A \times_k X \to A$ is a morphism of group schemes), so is constant because $X$ is connected.
\end{proof}

\begin{corollary} \label[corollary]{cor:ExactSequenceOfAcal}
Assume $X$ has a $k$-rational point $x_0$.  Then there is a commutative diagram with exact rows
\[\begin{tikzcd}
    0 \arrow[r] & A(k) \arrow[r] \arrow[d,"\iso"] & \Acal(X) \arrow[r] \arrow[d,equal] & \Hom_k(\Alb_{X/k}, A) \arrow[r] \arrow[d,"\iso"]   & 0\\
    0 \arrow[r] & \Acal(X)_\tors \arrow[r]     & \Acal(X) \arrow[r]            & \Acal(X)_\Tors \arrow[r]                      & 0,
\end{tikzcd}\]
and
\[
    \rk_\Z\Acal(X) = r(f_A,f_{\Alb_{X/k}}) = \dim_{\Q_\ell}\Hom_{\Q_\ell[G_k]}(V_\ell A,V_\ell\Alb_{X/k}),
\]
with $f_A,f_B$ the characteristic polynomials of the Frobenius on $A,B/k$ and
\[
	r(f_A,f_B) = \sum_{P \in \Q[T]\text{ irreducible}}v_P(f_A)v_P(f_B)\deg{P}
\]
(see~\cite[p.~138]{TateEndomorphisms}).
\end{corollary}
\begin{proof}
The lower row is trivially exact.  By the universal property of the Albanese variety (use that $X$ has a $k$-rational point $x_0$), one has $\{f \in \Mor_k(X, A) \mid f(x_0) = 0\} = \Hom_k(\Alb_{X/k}, A)$.  Thus the upper row is exact.  The left hand vertical arrow is an isomorphism because of~\cref{lemma:BasicPropertiesOfConstantAbelianSchemes}\,3.  Now the five lemma implies that the right hand vertical arrow is an isomorphism since it is a well-defined homomorphism: Precompose $f: \Alb_{X/k} \to A$ with the Abel-Jacobi map $\phi: X \to \Alb_{X/k}$ associated to $x_0$.

The equality for the rank follows from~\cite[p.~139, equation~(5) and Theorem~1\,(a)]{TateEndomorphisms}.
\end{proof}

\begin{example} \label[example]{ex:RankOfConstantAbelianVarietyOverP1}
The rank of the Mordell-Weil group of a constant Abelian variety over a projective space is $0$, since there are no non-constant $k$-morphisms $\mathbf{P}^n_k \to A$, see~\cite[p.~107, Corollary~3.9]{MilneAbelianVarieties}.
\end{example}

\begin{lemma} \label[lemma]{lemma:InvariantsAndTensorProduct}
Let $M$ and $N$ be torsion-free finitely generated $\Z_\ell$-modules, resp.\ continuous $\Z_\ell[\Gamma]$-modules.  Then one has
\begin{align*}
  M \otimes_{\Z_\ell} N &= \Hom_{\Z_\ell\text{-}\Mod}(M^\vee,N) \\
  (M \otimes_{\Z_\ell} N)^\Gamma &= \Hom_{\Z_\ell[\Gamma]\text{-}\Mod}(M^\vee,N) \\
\end{align*}
\end{lemma}
\begin{proof}
For the first equality, see~\cite[p.~628, Corollary~XVI.5.5]{LangAlgebra}.   Note that finitely generated torsion-free modules over a principal ideal domain are free.
	
The second equality follows from $\Hom_R(M,N)^\Gamma = \Hom_{R[\Gamma]}(M,N)$ for any commutative ring $R$ with $1$, group $\Gamma$ and $R$-modules $M,N$ and the first equality and using $M^{\Gamma} = M^\Z$ for $M$ a discrete $\Z_\ell[\Gamma]$-module since $\Z \subset \Gamma$ is dense.
\end{proof}

\begin{lemma} \label[lemma]{lemma:projection formula}
Let $\Acal = A\times_kX$ be a constant Abelian scheme.  Then one has $\H^i(\overline{X}, T_\ell\Acal) = \H^i(\overline{X}, \Z_\ell) \otimes_{\Z_\ell} T_\ell A$ as $\ell$-adic sheaves on the étale site of $k$.
\end{lemma}
\begin{proof}
This follows from~\cref{lemma:BasicPropertiesOfConstantAbelianSchemes}\,2 and the projection formula.
\end{proof}

\begin{theorem} \label[theorem]{thm:AlbDualToPic}
Let $X/k$ be a smooth projective geometrically connected variety with a $k$-rational point.  Then the reduced Picard variety $(\PPic_{X/k}^0)_{\mathrm{red}}$ is dual to $\Alb_{X/k}$ and $\PPic_{X/k}^0$ is reduced if and only if $\dim \PPic_{X/k}^0 = \dim_k \H^1_{\Zar}(X,\Ocal_X)$.
\end{theorem}
\begin{proof}
By~\cite[Proposition~A.6\,(i)]{MochizukiTopics} or~\cite[p.~289\,f., Remark~9.5.25]{FGAExplained}, $(\PPic_{X/k}^0)_{\mathrm{red}}$ is dual to $\Alb_{X/k}$.  By~\cite[p.~283, Corollary~9.5.13]{FGAExplained}, the Picard variety is reduced (and then smooth and an Abelian scheme) iff equality holds in $\dim \PPic_{X/k}^0 \leq \dim_k \H^1_{\Zar}(X,\Ocal_X)$.
\end{proof}

\begin{remark} \label[remark]{rem:DefectOfSmoothness}
The integer $\alpha(X) := \dim_k \H^1_{\Zar}(X,\Ocal_X) - \dim \PPic_{X/k}^0 \geq 0$ is called the \defn{defect of smoothness}.
\end{remark}

\begin{example} \label[example]{ex:DefectOfSmoothness}
One has $\alpha(X) = 0$ iff the Picard scheme of $X/k$ is smooth (since a group variety is smooth iff it is reduced) iff the dimension of $\H^1_{\Zar}(\overline{X},\Ocal_{\overline{X}})$ as a vector space over $\overline{k}$ equals the dimension of the Albanese variety of $\overline{X}/\overline{k}$~\cite[p.~94, Remarks to Theorem~1]{Milne-Tate-Shafarevic}:

This holds true for K3 surfaces since $\H^1_{\Zar}(\overline{X},\Ocal_{\overline{X}}) = 0$ by~\cite[p.~1, Definition~1.1]{HuybrechtsK3Surfaces}.  In characteristic $0$, this is always the case~\cite[p.~94, Remarks to Theorem~1]{Milne-Tate-Shafarevic}.  For examples of non-reduced Picard schemes of smooth projective surfaces in positive characteristic see~\cite{LiedtkeNonReducedPicardSchemes}.
\end{example}

\begin{lemma} \label[lemma]{lemma:WeilPairingAndHomomorphisms}
Let $f: A \to B$ a homomorphism of an Abelian varieties and $e_A: T_\ell A \times T_\ell A^t \to \Z_\ell(1)$ and $e_B: T_\ell B \times T_\ell B^t \to \Z_\ell(1)$ be the perfect Weil pairings from~\cref{thm:WeilPairing}.  Then
\[
    e_B(f(a),b) = e_A(a,f^t(b))
\]
for all $a \in T_\ell A$ and $b \in T_\ell B^t$, i.\,e.\ the diagram
\[\begin{tikzcd}
    T_\ell A \arrow[r,"\times",phantom] \arrow[d,"f"] & T_\ell A^t \arrow[r,"{e_A}"] & \Z_\ell(1) \arrow[d,equal]\\
    T_\ell B \arrow[r,"\times",phantom] & T_\ell B^t \arrow[u,"{f^t}"] \arrow[r,"{e_B}"] & \Z_\ell(1)
\end{tikzcd}\]
commutes.
\end{lemma}
\begin{proof}
See~\cite[p.~186, (I)]{MumfordAbelianVarieties}.
\end{proof}

\begin{corollary} \label[corollary]{cor:TraceOfDualMorphism}
Let $f: A \to A$ be an endomorphism of an Abelian variety $A$.  Then
\[
    \Tr(f) = \Tr_{T_\ell(A)}(f) = \Tr_{T_\ell(A^t)}(f^t) = \Tr(f^t).
\]
\end{corollary}
\begin{proof}
Choosing an isomorphism $\Z_\ell(1) \iso \Z_\ell$,\footnote{a heresy!~\cite[p.~194, l.~$-6$]{zbMATH03319220}} dualising the diagram in~\cref{lemma:WeilPairingAndHomomorphisms} and using that the Weil pairing is perfect by~\cref{thm:WeilPairing}, gives us a commutative diagram
\[\begin{tikzcd}
    T_\ell A \arrow[r,"f"] \arrow[d,"\iso"] & T_\ell A \arrow[d,"\iso"] \\
    (T_\ell A^t)^\vee \arrow[r,"(f^t)^\vee"] & (T_\ell A^t)^\vee.
\end{tikzcd}\]
Now use that dualising does not change the trace.

The trace of an endomorphism can be calculated on $\ell$-adic Tate modules by~\cite[p.~125, Proposition~12.9]{MilneAbelianVarieties}.
\end{proof}

\begin{lemma} \label[lemma]{lemma:bigCommutativeDiagram}
	Let $X/k$ be a smooth projective geometrically connected variety of dimension $d$ with Albanese variety $A$ associated to a base point $x_0 \in X(k)$ such that $\PPic_{X/k}$ is reduced.  Consider the following diagram of finitely generated free $\Z_\ell$-modules:
	{\scriptsize
		\begin{equation}\begin{tikzcd}
		\H^1(X, T_\ell\Acal)_\Tors  \times \H^{2d-1}(X, T_\ell(\Acal^t)(d-1))_\Tors \arrow[d,"\pr_1^*"] \arrow[dr,"\text{(1)}",phantom]\arrow[r,"\cup"] & \H^{2d}(X,\Z_\ell(d)) \to \H^{2d}(\overline{X},\Z_\ell(d)) \arrow[r,"\iso"] &\Z_\ell \arrow[d,equal]\\
		\H^1(\overline{X}, T_\ell\Acal)^\Gamma  \times \H^{2d-1}(\overline{X}, T_\ell(\Acal^t)(d-1))^\Gamma \arrow[dr,"\text{(2)}",phantom]\arrow[r,"\cup"] & \H^{2d}(\overline{X},\Z_\ell(d)) \arrow[r,"\iso"] &\Z_\ell \arrow[d,equal]\\
		\big(\H^1(\overline{X}, \Z_\ell(1)) \otimes_{\Z_\ell} (T_\ell A)(-1)\big)^\Gamma  \times \big(\H^{2d-1}(\overline{X}, \Z_\ell(d-1)) \otimes_{\Z_\ell} T_\ell(A^t)\big)^\Gamma \arrow[u,"\iso"] \arrow[d,"\iso" left] \arrow[dr,"\text{(3)}",phantom]\arrow[r] & \H^{2d}(\overline{X},\Z_\ell(d)) \otimes_{\Z_\ell} \Z_\ell\arrow[u,equal] \arrow[r,"\iso"] &\Z_\ell \arrow[d,equal]\\
		\big(T_\ell\PPic^0_{X/k} \otimes_{\Z_\ell} (T_\ell A)(-1)\big)^\Gamma  \times \big((T_\ell\PPic^0_{X/k})^\vee \otimes_{\Z_\ell} T_\ell(A^t)\big)^\Gamma \arrow[d,"\iso" left] \arrow[dr,"\text{(4)}",phantom] \arrow[r]& \End(T_\ell\PPic^0_{X/k}) \otimes_{\Z_\ell} \Z_\ell \arrow[r,"\Tr"] &\Z_\ell \arrow[d,equal]\\
		\Hom_{\Z_\ell[\Gamma]\text{-}\Mod}\big(\left((T_\ell A)(-1)\right)^\vee, T_\ell\PPic^0_{X/k}\big)  \times \Hom_{\Z_\ell[\Gamma]\text{-}\Mod}\big(T_\ell\PPic^0_{X/k}, T_\ell(A^t)\big) \arrow[d,"\iso" left] \arrow[dr,"\text{(5)}",phantom]\arrow[r] & \Hom_{\Z_\ell[\Gamma]\text{-}\Mod}\big(\left((T_\ell A)(-1)\right)^\vee, T_\ell(A^t)\big) \arrow[r] &\Z_\ell \arrow[d,equal]\\
		\Hom_{\Z_\ell[\Gamma]\text{-}\Mod}\big(T_\ell A^t, T_\ell\PPic^0_{X/k}\big)  \times \Hom_{\Z_\ell[\Gamma]\text{-}\Mod}\big(T_\ell\PPic^0_{X/k}, T_\ell A^t\big)  \arrow[dr,"\text{(6)}",phantom] \arrow[r,"\circ"]& \End_{\Z_\ell[\Gamma]\text{-}\Mod}(T_\ell A^t) \arrow[r,"\Tr"] &\Z_\ell \arrow[d,equal]\\
		\Hom_k(A^t, \PPic^0_{X/k}) \otimes_\Z \Z_\ell \times \Hom_k(\PPic^0_{X/k}, A^t) \otimes_\Z \Z_\ell \arrow[u,"\iso"] \arrow[dr,"\text{(7)}",phantom] \arrow[r,"\circ"]& \End_k(A^t) \otimes_\Z \Z_\ell \arrow[u,"\iso"]\arrow[r,"\Tr"] &\Z_\ell \arrow[d,equal]\\
		\Hom_k(\Alb_{X/k}, A) \otimes_\Z \Z_\ell \times \Hom_k(A, \Alb_{X/k}) \otimes_\Z \Z_\ell \arrow[u,"\iso" left,"(-)^t" right]\arrow[r,"\circ"] & \End_k(\Alb_{X/k}) \otimes_\Z \Z_\ell \arrow[r,"\Tr"] &\Z_\ell
		\end{tikzcd}\label{eq:bigDiagram}\end{equation}
	}
	We claim that the above diagram~\eqref{eq:bigDiagram} commutes and that the left-hand vertical arrows are indeed isomorphisms.
\end{lemma}
\begin{proof}
	First we will justify that the left-hand vertical arrows in diagram~\eqref{eq:bigDiagram} that are claimed to be isomorphisms are indeed so.

	For the vertical isomorphisms in the first factor in the left column of~\eqref{eq:bigDiagram}:  One has
	\begin{align*}
	\H^1(X, T_\ell\Acal)_\Tors  &= \H^1(\overline{X}, T_\ell\Acal)^\Gamma \quad\text{by~\eqref{eq:sesHS1}} \\
	&= \left(\H^1(\overline{X}, \Z_\ell(1)) \otimes_{\Z_\ell} (T_\ell A)(-1)\right)^\Gamma \quad\text{by~\cref{lemma:projection formula}} \\
	&= \big(T_\ell\PPic^0_{X/k} \otimes_{\Z_\ell} (T_\ell A)(-1)\big)^\Gamma \quad\text{by the Kummer sequence}\\
	&= \Hom_{\Z_\ell[\Gamma]\text{-}\Mod}\big(\left((T_\ell A)(-1)\right)^\vee, T_\ell\PPic^0_{X/k}\big) \quad\text{by~\cref{lemma:InvariantsAndTensorProduct}}  \\
	&= \Hom_{\Z_\ell[\Gamma]\text{-}\Mod}\big(\Hom_{\Z_\ell\text{-}\Mod}(T_\ell(A^t), \Z_\ell)^\vee, T_\ell\PPic^0_{X/k}\big) \quad\text{by~\eqref{eq:HomTellAZell}} \\
	&= \Hom_{\Z_\ell[\Gamma]\text{-}\Mod}\big(T_\ell(A^t), T_\ell\PPic^0_{X/k}\big) \\
	&= \Hom_k(A^t, \PPic^0_{X/k}) \otimes_\Z \Z_\ell \quad\text{by the Tate conjecture~\cite{TateEndomorphisms}}\\
	&= \Hom_k(\Alb_{X/k},A) \otimes_\Z \Z_\ell \quad\text{since the functor $(-)^t$ is an autoduality.}
	\end{align*}
	Note that $\H^1(\overline{X}, T_\ell\Acal)^\Gamma$ is torsion-free since $\H^1(\overline{X}, T_\ell\Acal)$ is so, and this holds because of the Künneth formula and since $\H^1(\overline{X},\Z_\ell(1)) = T_\ell\PPic^0_{X/k}$ is torsion-free by~\cref{lemma:PropertiesOfTateModule}\,\ref{lemma:Tate-module-is-torsion-free}.  Therefore, in~\eqref{eq:sesHS1}, $\ker\alpha = \H^0(\overline{X},T_\ell\Acal)_\Gamma$ is the whole torsion subgroup of $\H^1(X,T_\ell\Acal)$.

	$T_\ell\Acal$ has weight $-1$ by~\cref{prop:RiundTateModul} and $T_\ell(\Acal^t)(d-1)$ has weight $-1 - 2(d-1) = -2d + 1$ and from~\eqref{eq:sesHS}, we have a commutative diagram with exact rows
	\[
	\begin{tikzcd}
	0 \arrow[r] &\H^{2d-1}(\overline{X}, T_\ell(\Acal^t)(d-1))_\Gamma \arrow[r,"\beta"] &\H^{2d}(X, T_\ell(\Acal^t)(d-1)) \arrow[r] &\H^{2d}(\overline{X}, T_\ell(\Acal^t)(d-1))^\Gamma \arrow[r] &0\\
	0 \arrow[r] &\H^{2d-2}(\overline{X}, T_\ell(\Acal^t)(d-1))_\Gamma \arrow[r] &\H^{2d-1}(X, T_\ell(\Acal^t)(d-1)) \arrow[r,"\alpha"] &\H^{2d-1}(\overline{X}, T_\ell(\Acal^t)(d-1))^\Gamma \arrow[llu,"f",dashrightarrow] \arrow[r] &0,
	\end{tikzcd}
	\]
	where only the four groups connected by $f$, $\alpha$ and $\beta$ can be infinite by~\cref{lemma:alphafandbeta} and as in~\eqref{eq:sesHS1}.
	
	The perfect Poincaré duality pairing
	\begin{align} \label{eq:PoincareDualityAndPic}
	\H^1(\overline{X}, \Z_\ell(1)) \times \H^{2d-1}(\overline{X}, \Z_\ell(d-1)) \to \H^{2d}(\overline{X},\Z_\ell(d)) \isoto \Z_\ell
	\end{align}
	identifies $\H^{2d-1}(\overline{X}, \Z_\ell(d-1))$ with $(T_\ell\PPic^0_{X/k})^\vee$.
	
	For the vertical isomorphisms in the second factor in the left column of~\eqref{eq:bigDiagram}:  One has
	\begin{align*}
	\H^{2d-1}(X, T_\ell(\Acal^t)(d-1))_\Tors &= \H^{2d-1}(\overline{X}, T_\ell(\Acal^t)(d-1))^\Gamma \quad\text{by~\eqref{eq:sesHS1}} \\
	&= \left(\H^{2d-1}(\overline{X}, \Z_\ell(d-1)) \otimes_{\Z_\ell} T_\ell(A^t)\right)^\Gamma \quad\text{by~\cref{lemma:projection formula}} \\
	&= \big((T_\ell\PPic^0_{X/k})^\vee \otimes_{\Z_\ell} T_\ell(A^t)\big)^\Gamma \quad\text{by~\eqref{eq:PoincareDualityAndPic}}\\
	&= \Hom_{\Z_\ell[\Gamma]\text{-}\Mod}\big(T_\ell\PPic^0_{X/k}, T_\ell(A^t)\big) \quad\text{by~\cref{lemma:InvariantsAndTensorProduct}}\\
	&= \Hom_k(\PPic^0_{X/k}, A^t) \otimes_\Z \Z_\ell \quad\text{by the Tate conjecture~\cite{TateEndomorphisms}}\\
	&= \Hom_k(A,\Alb_{X/k}) \otimes_\Z \Z_\ell \quad\text{the functor $(-)^t$ is an autoduality.}
	\end{align*}
	
	Now we will prove that the diagram commutes:
	
	$(1)$ commutes since $\cup$-product commutes with restrictions.
	
	$(2)$ commutes because of the associativity of the $\cup$-product.
	
	$(3)$ commutes since, in general, one has a commutative diagram of finitely generated free modules over a ring $R$
	\[\begin{tikzcd}
	A \times B \arrow[d,"\iso"] \arrow[r,"\qu{\cdot{,}\cdot}"] & R \arrow[d,equal] \\
	C \times C^\vee          \arrow[r]                    & R
	\end{tikzcd}\]
	identifying $B$ with the dual of $C \iso A$ with a perfect pairing $\qu{\cdot,\cdot}$ and the canonical pairing $C \times C^\vee \to R, (c,\phi) \mapsto \phi(c)$:  Choose a basis $(a_i)$ of $A$ and the dual basis $(b_i)$ of $B$; these are mapped to the bases $(c_i)$ and $(c_i')$ of $C$ and $C^\vee$.  Then, under the top horizontal map, $\qu{a_i, b_j} = \delta_{ij}$ with the Kronecker symbol $\delta_{ij}$, and under the bottom horizontal map $(c_i, c_j') = \delta_{ij}$.
	
	$(4)$ commutes since, in general, one has using~\cref{lemma:InvariantsAndTensorProduct} a commutative diagram of finitely generated free modules over a ring $R$
	\[\begin{tikzcd}
	(M \otimes_R N^\vee) \times (M^\vee \otimes_R N) \arrow[d,"\iso"]  \arrow[r]          & \End_{R\text{-}\Mod}(M) \otimes_R \End_{R\text{-}\Mod}(N) \arrow[rr,"\Tr_M \otimes_R \Tr_N"] && R \arrow[d,equal] \\
	\Hom_{R\text{-}\Mod}(N, M) \times \Hom_{R\text{-}\Mod}(M, N)                  \arrow[r,"\circ"]  & \End_{R\text{-}\Mod}(N) \arrow[rr,"\Tr_N"] && R.
	\end{tikzcd}\]
	For proving this, choose bases $(a_i)$ of $M$ and $(b_i)$ of $N$ and their dual bases $(a_i')$ of $M^\vee$ and $(b_i')$ of $N^\vee$.  The element $(a_i \otimes b_j', a_k' \otimes b_l)$ of $(M \otimes_K N^\vee) \times (M^\vee \otimes_K N)$ is sent by the upper horizontal arrows to $\delta_{ik}\delta{jl}$, and by the left vertical arrow to $(b_m \mapsto b_j'(b_m)a_i, a_n \mapsto a_k'(a_n)b_l)$.  The latter element is mapped by the lower left horizontal arrow to $b_m \mapsto a_k'(b_j'(b_m)a_i)b_l$ and the trace of this endomorphism is $\delta_{jl}\delta_{ki}$.  Therefore, the diagram commutes.
	
	$(5)$ commutes because of precomposing with the isomorphism $(T_\ell\Acal(-1))^\vee \isoto T_\ell(\Acal^t)$ coming from the perfect Weil pairing~\cref{thm:WeilPairing}.
	
	$(6)$ commutes because of~\cite[p.~186\,f., Theorem~3]{LangAV}.
	
	$(7)$ commutes since $\PPic^0_{X/k}$ is dual to $\Alb_{X/k}$ by~\cref{thm:AlbDualToPic} since $\PPic_{X/k}$ is reduced and because of
	\begin{align*}
		\Tr(\beta \circ \alpha) &= \Tr((\beta \circ \alpha)^t) \quad\text{by~\cref{cor:TraceOfDualMorphism}}\\
								&= \Tr(\alpha^t \circ \beta^t)\\
								&= \Tr(\beta^t \circ \alpha^t) \quad\text{by~\cite[p.~187, Corollary~1]{LangAV}.}
	\end{align*}
\end{proof}

\begin{theorem}[The cohomological and the trace pairing] \label[theorem]{thm:HeightAndTracePairing}
Let $X/k$ be a smooth projective geometrically connected variety of dimension $d$ with Albanese variety $A$ associated to a base point $x_0 \in X(k)$ such that $\PPic_{X/k}$ is reduced.  Denote the constant Abelian scheme $B\times_kX/X$ by $\Acal/X$.  Then the trace pairing
\[
    \Hom_k(A,B) \times \Hom_k(B,A) \stackrel{\circ}{\to} \End_k(A) \stackrel{\Tr}{\to} \Z
\]
tensored with $\Z_\ell$ equals the cohomological pairing from~\eqref{eq:Regulator spitze Klammer}
\[
    \qu{\cdot,\cdot}_\ell: \H^1(X, T_\ell\Acal)_\Tors \times \H^{2d-1}(X, T_\ell(\Acal^t)(d-1))_\Tors \to \H^{2d}(X,\Z_\ell(d)) \stackrel{\pr_1^*}{\to} \H^{2d}(\overline{X},\Z_\ell(d)) = \Z_\ell,
\]
and this equals by~\cref{thm:comparisonofpairings} the Néron-Tate canonical height pairing up to the integral hard Lefschetz defect (see~\cref{def:integralhardLefschetzdefect}).
\end{theorem}
\begin{proof}
First note that the Kummer sequence for $\G_m$ on $X$ gives us a short exact sequence
\[
  1 = \G_m(\overline{X})/\ell^n \to \H^1(\overline{X},\mu_{\ell^n}) \to \Pic(\overline{X})[\ell^n] \to 0,
\]
the first equality since $\G_m(\overline{X})/\ell^n = \overline{k}^\times/\ell^n = 1$ since $X/k$ is proper and geometrically integral, and passing to the inverse limit over $n$, an isomorphism $\H^1(\overline{X},\Z_\ell(1)) = T_\ell\Pic(\overline{X}) = T_\ell\PPic^0_{X/k}$, the latter equality since $T_\ell\NS(\overline{X}) = 0$ since the Néron-Severi group is finitely generated by the theorem of the base~\cite[p.~215, Theorem~V.3.25]{MilneEtaleCohomology}.

Now use~\cref{lemma:bigCommutativeDiagram}.
\end{proof}

\begin{example}
	In particular, if the characteristic polynomials of the Frobenius on $\PPic^0_{X/k}$ and $A^t$ are coprime, then $\Hom_k(A^t, \PPic^0_{X/k}) = 0 = \Hom_k(\PPic^0_{X/k}, A^t)$ and the discriminants of the parings $\qu{\cdot,\cdot}_\ell$ and $(\cdot,\cdot)_\ell$ from~\eqref{eq:Regulator spitze Klammer} and~\eqref{eq:Regulator runde Klammer} are equal to $1$.
\end{example}

\begin{theorem} \label[theorem]{thm:MilnesMainTheorem}
Let $k = \F_q$, $q = p^n$ be a finite field and $X/k$ a smooth projective and geometrically connected variety and assume $\overline{X} = X \times_k \overline{k}$ satisfies

(a) the Néron-Severi group of $\overline{X}$ is torsion-free and

(b) the dimension of $\H^1_{\Zar}(\overline{X},\Ocal_{\overline{X}})$ as a vector space over $\overline{k}$ equals the dimension of the Albanese variety of $\overline{X}/\overline{k}$.

If $B/k$ is an Abelian variety, then $\H^1(X,B)$ is finite and its order satisfies the relation
\[
    q^{gd}\prod_{a_i \neq b_j}\Big(1-\frac{a_i}{b_j}\Big) = \card{\H^1(X,B)}\card{\det\qu{\alpha_i,\beta_j}},
\]
where $A/k$ is the Albanese variety of $X/k$, $g$ and $d$ are the dimensions of $A$ and $B$, respectively, $(a_i)_{i=1}^{2g}$ and $(b_j)_{j=1}^{2d}$ are the roots of the characteristic polynomials of the Frobenius of $A/k$ and $B/k$, $(\alpha_i)_{i=1}^r$ and $(\beta_i)_{i=1}^r$ are bases for $\Hom_k(A,B)$ and $\Hom_k(B,A)$, and $\qu{\alpha_i,\beta_j}$ is the trace of the endomorphism $\beta_j\alpha_i$ of $A$.
\end{theorem}
\begin{proof}
See~\cite[p.~98, Theorem~2]{Milne-Tate-Shafarevic}.
\end{proof}

\begin{remark} \label[remark]{rem:ShaConstant}
Note that $\Hom_k(A,B)$ and $\Hom_k(B,A)$ are free $\Z$-modules of the same rank $r = r(f_A, f_B) \leq 4gd$ by~\cite[p.~139, Theorem~1\,(a)]{TateEndomorphisms}, with $f_A$ and $f_B$ the characteristic polynomials of the Frobenius of $A/k$ and $B/k$.  (Another argument for them having the same rank is that the category of Abelian varieties up to isogeny is semi-simple, decomposing $A$ and $B$ into simple factors.)  Furthermore, $\H^1(X,B) = \H^1(X,B\times_kX) = \Sha(B\times_kX/X)$ since for $U \to X$, one has $B(U) = (B\times_kX)(U)$ by the universal property of the fibre product.
\end{remark}

\begin{example} \label[example]{ex:aandb}
	(a) and (b) in~\cref{thm:MilnesMainTheorem} are satisfied for $X = A$ an Abelian variety, a K3 surface or a curve: (a) because of~\cite[p.~178, Corollary~2]{MumfordAbelianVarieties} and~\cite[p.~385\,ff., Chapter~17]{HuybrechtsK3Surfaces}, and (b) for curves and Abelian varieties since $A^t = \PPic^0_{A/k}$ is an Abelian variety, in particular smooth and reduced, and by~\cref{ex:DefectOfSmoothness} for K3 surfaces.  See also~\cref{thm:AlbDualToPic}, \cref{rem:DefectOfSmoothness} and~\cref{ex:DefectOfSmoothness}.
\end{example}

\begin{lemma} \label[lemma]{Lemma: permutation of eigenvalues of the Frobenius}
Let $k = \F_q$ be a finite field, $\ell$ invertible in $k$ and $A/k$ be an Abelian variety of dimension $g$.  Denote the eigenvalues of the Frobenius $\Frob_q$ on $V_\ell A$ by $(\alpha_i)_{i=1}^{2g}$.  Then $\alpha_i \mapsto q/\alpha_i$ is a bijection.
\end{lemma}
\begin{proof}
The Weil pairing (\cref{thm:WeilPairing}) induces a perfect Galois equivariant pairing
\[
    V_\ell A \times V_\ell A^t \to \Q_\ell(1),
\]
and, choosing a polarisation $f: A \to A^t$, by~\cref{lemma:isogeny induces iso on Tate modules}, we also have by precomposing a perfect Galois equivariant pairing
\[
    \qu{\cdot,\cdot}: V_\ell A \times V_\ell A \to \Q_\ell(1).
\]
Now let $v_i$ be an eigenvector of $\Frob_q$ on $V_\ell A$ with eigenvalue $\alpha_i$.  Then there is exactly one eigenvector $v_j$ of $\Frob_q$ on $V_\ell A$ such that $\qu{v_i,v_j} = 1 \neq 0$ (otherwise, since the pairing $\qu{\cdot,\cdot}$ is perfect, we would have $\qu{v_i,v_j} = 0$ for all eigenvectors $v_j$, but there is a basis of eigenvectors on the Tate module since the Frobenius acts semi-simply).  Now, since the pairing is Galois equivariant, $q = \Frob_q(1) = \Frob_q\qu{v_i,v_j} = \qu{\Frob_qv_i,\Frob_qv_j} = \qu{\alpha_iv_i, \alpha_jv_j} = \alpha_i\alpha_j\qu{v_i,v_j} = \alpha_i\alpha_j1$, and the statement follows.
\end{proof}

\begin{definition}
Define the \defn{regulator} $R(\Acal/X)$ of $\Acal/X$ as $|\det(\qu{\cdot,\cdot})|$.
\end{definition}

By~\cref{rem:ShaConstant}, we get
\begin{corollary} \label[corollary]{cor:MilnesMainTheorem}
In the situation of~\cref{thm:MilnesMainTheorem}, one has
\[
    q^{gd}\prod_{a_i \neq b_j}\Big(1-\frac{a_i}{b_j}\Big) = \card{\Sha(B\times_kX/X)} R(\Acal/X).
\]
\end{corollary}

\begin{definition}  \label[definition]{def:LFunctionMilne}
Define the \defn{$L$-function} of $B\times_kX/X$ as the $L$-function of the Chow motive
\[
    h^1(B) \otimes (h^0(X) \oplus h^1(X)) = h^1(B) \oplus (h^1(B) \otimes h^1(X)),
\]
namely
\[
    L(B\times_kX/X,s) = \frac{L(h^1(B) \otimes h^1(X),s)}{L(h^1(B),s)}.
\]
\end{definition}
Here, the Künneth projectors are algebraic by~\cite[p.~217, Corollary~3.2]{DeningerMurre}. 

\begin{theorem} \label[theorem]{thm:Lfnsagree}
The two $L$-functions~\cref{def:Li} and~\cref{def:LFunctionMilne} agree for constant Abelian schemes.
\end{theorem}
\begin{proof}
One has $V_\ell B = \H^1(\overline{B},\Q_\ell)^\vee$ by~\cref{thm:Tate module and etale cohomology}, $(V_\ell B)^\vee = (V_\ell B^t)(-1)$, $V_\ell(B) \iso V_\ell(B^t)$ by~\cref{lemma:isogeny induces iso on Tate modules} and the existence of a polarisation~\cite{MilneAbelianVarieties}, p.~113, Theorem~7.1, $\H^i(\overline{X}, V_\ell\Acal) = \H^i(\overline{X}, \Q_\ell) \otimes V_\ell B$ by~\cref{lemma:projection formula} and $V_\ell\Acal = (V_\ell B) \times_k X$ by~\cref{lemma:BasicPropertiesOfConstantAbelianSchemes} since $\Acal/X$ is constant.  Using this, one gets
\begin{align*}
    L(h^i(X) \otimes h^1(B), t) &= \det(1 - \Frob_q^{-1}t \mid \H^i(\overline{X}, \Q_\ell) \otimes \H^1(\overline{B}, \Q_\ell)) \\
                                &= \det(1 - \Frob_q^{-1}t \mid \H^i(\overline{X}, \Q_\ell) \otimes V_\ell(B^t)(-1)) \\
                                &= \det(1 - \Frob_q^{-1}t \mid \H^i(\overline{X}, \Q_\ell) \otimes V_\ell(B)(-1)) \\
                                &= \det(1 - \Frob_q^{-1}t \mid \H^i(\overline{X}, (V_\ell B) \times_k X)(-1)) \\
                                &= \det(1 - \Frob_q^{-1}q^{-1}t \mid \H^i(\overline{X}, V_\ell\Acal)) \\
                                & = L_i(\Acal/X,q^{-1}t).
\end{align*}
Now conclude using $h^1(B) = h^0(X) \otimes h^1(B)$ since $X$ is connected.
\end{proof}

\begin{remark}
Note that
\begin{align*}
    \ord_{t=1}L(\Acal/X,t) = \ord_{s=1}L(\Acal/X,q^{-1}q^s).
\end{align*}
\end{remark}

\begin{remark} \label[remark]{rem:MotivationForLfunction}
Now let us explain how we came up with this definition of the $L$-function.  We omit the characteristic polynomials $L_i(\Acal/X,t)$ in higher dimensions $i > 1$ since otherwise cardinalities of cohomology groups would turn up in the special $L$-value which we have no interpretation for (as in the case $i = 0$ and the cardinality of the $\ell$-torsion of the Mordell-Weil group, or in the case $i = 1$ and the cardinality of the $\ell$-torsion of the Tate-Shafarevich group).  In the case of a curve $C$ as a basis, our definition is the same as the classical definition of the $L$-function up to an $L_2(t)$-factor.  This factor contributes basically only a factor $\card{\Acal^t(X)[\ell^\infty]_\tors}$ in the denominator.  In the classical curve case $\dim{X} = 1$, the $L$-function can also be represented as a product over all closed points $x \in |X|$ of Euler factors.
%
\end{remark}

We expand
\begin{align*}
    L(B\times_kX/X, s) &= \frac{L(h^1(B) \otimes h^1(X),s)}{L(h^1(B),s)} \\
                        &= \frac{\prod_{j=1}^{2d}\prod_{i=1}^{2g}(1-a_ib_jq^{-s})}{\prod_{j=1}^{2d}(1-b_jq^{-s})}.
\end{align*}

By~\cref{Lemma: permutation of eigenvalues of the Frobenius}, one has for the numerator
\begin{equation} \label{eq:Numerator}
    \prod_{j=1}^{2d}\prod_{i=1}^{2g}(1-a_ib_jq^{-s}) = \prod_{j=1}^{2d}\prod_{i=1}^{2g}\Big(1-\frac{a_i}{b_j}q^{1-s}\Big),
\end{equation}
and the denominator has no zeros at $s=1$ by the Riemann hypothesis (the eigenvalues of the Frobenius $(b_j)$ on $h^1(B)$ have absolute value $q^{1/2}$).  Therefore
\[
    \ord_{s=1}L(B\times_kX/X, s) = r(f_A, f_B)
\]
is equal to the number $r(f_A, f_B)$ of pairs $(i,j)$ such that $a_i = b_j$, which equals by~\cite[p.~139, Theorem~1\,(a)]{TateEndomorphisms} the rank $r$ of $(B\times_kX)(X)$:
\begin{align*}
    r(f_A, f_B) &= \rk_\Z\Hom_k(A,B) \\
                &= \rk_\Z\Hom_k(X,B) \quad\text{by the universal property of the Albanese variety}\\
                &= \rk_\Z\Hom_X(X, B \times_k X),
\end{align*}
see~\cref{cor:ExactSequenceOfAcal}.

\begin{lemma} \label[lemma]{lemma:denominator}
The denominator evaluated at $s=1$ equals
\[
    \prod_{j=1}^{2d}(1-b_jq^{-1}) = \frac{\card{(B\times_kX)(X)_\tors}}{q^d}.
\]
\end{lemma}
\begin{proof}
\begin{align*}
    \prod_{j=1}^{2d}(1-b_jq^{-1}) &= \prod_{j=1}^{2d}\Big(1-\frac{1}{b_j}\Big) \quad\text{by~\cref{Lemma: permutation of eigenvalues of the Frobenius}} \\
                                    &= \prod_{j=1}^{2d}\frac{b_j-1}{b_j} \\
                                    &= \prod_{j=1}^{2d}\frac{1-b_j}{b_j} \quad\text{since $2d$ is even}\\
                                    &= \frac{\deg(\id_B-\Frob_q)}{q^d} \quad\text{by~\cref{Lemma: permutation of eigenvalues of the Frobenius} and~\cite[p.~186\,f., Theorem~3]{LangAV}}\\
                                    &= \frac{\card{B(\F_q)}}{q^d} \quad\text{since $\id_B - \Frob_q$ is separable} \\
                                    &= \frac{\card{(B\times_kX)(X)_\tors}}{q^d} \quad\text{by~\cref{lemma:BasicPropertiesOfConstantAbelianSchemes}\,3}. \qedhere
\end{align*}
\end{proof}

\begin{remark}
Note that, if $X/k$ is a smooth \emph{curve}, $(B\times_kX)(X) = B(K)$ with $K = k(X)$ the function field of $X$ by the valuative criterion for properness since $X/k$ is a smooth curve and $B/k$ is proper.  For general $X$, setting $\Acal = B \times_k X$ and $K = k(X)$ the function field, $(B\times_kX)(X) = \Acal(X) = A(K)$ also holds true because of the Néron mapping property.
\end{remark}

\begin{remark}
One has $\card{(B\times_kX)(X)_\tors} = \card{B(k)} = \card{B^t(k)} = \card{(B\times_kX)^t(X)_\tors}$ by~\cref{lemma:BasicPropertiesOfConstantAbelianSchemes}\,3 and~\cref{lemma:isogenous Abelian varieties}.
\end{remark}


Putting everything together, one has
\begin{align*}
    \lim_{s \to 1}\frac{L(\Acal/X, s)}{(s-1)^r} &= \frac{q^d (\log q)^r}{\card{\Acal(X)_\tors}} \prod_{a_i \neq b_j}\Big(1-\frac{a_i}{b_j}\Big) \quad\text{by~\cref{lemma:denominator} and~\eqref{eq:Numerator}} \\
                                                    &= q^{(g-1)d} (\log q)^r\frac{\card{\Sha(\Acal/X)} \cdot R(\Acal/X)}{\card{\Acal(X)_\tors}} \quad\text{by~\cref{cor:MilnesMainTheorem}.}
\end{align*}

\begin{theorem} \label[theorem]{thm:BSDII}
In the situation of~\cref{thm:MilnesMainTheorem}, one has:

\noindent 1.  The Tate-Shafarevich group $\Sha(\Acal/X)$ is finite.

\noindent 2.  The vanishing order equals the Mordell-Weil rank $r$: $\ord\nolimits_{s=1}L(\Acal/X,s) = \rk_\Z\Acal(X) = \rk_\Z A(K)$.

\noindent 3.  There is the equality for the leading Taylor coefficient
    \[
        L^*(\Acal/X,1) = q^{(g-1)d} (\log q)^r\frac{\card{\Sha(\Acal/X)} \cdot R(\Acal/X)}{\card{\Acal(X)_\tors}}.
    \]
\end{theorem}

Combining~\cref{thm:BSDI} and~\cref{thm:BSDII} and using~\cref{thm:Lfnsagree}, one can identify the remaining two expressions in~\cref{thm:BSDI}:
\begin{corollary}
In the situation of~\cref{thm:MilnesMainTheorem}, in~\cref{thm:BSDI} resp.~\cref{lemma:ShaFinite}, all equalities hold and one has
\begin{align*}
    \left|\det(\cdot,\cdot)_\ell\right|_\ell^{-1}   &= 1,\\
    \card{\H^2(\overline{X},T_\ell\Acal)^\Gamma}         &= 1.
\end{align*}
\end{corollary}

\begin{remark} \label[remark]{rem:H2vanishing}
For \emph{constant} Abelian schemes $\Acal = A \times_k X$ (under the assumption (a) above that $\NS(\overline{X})$ is torsion-free), one has $\card{\H^2(\overline{X},T_\ell\Acal)^\Gamma} = 1$:\footnote{this factor turns up in the Birch-Swinnerton-Dyer formula for the special $L$-value~\cref{thm:BSDI}}  The long exact sequence associated to the Kummer sequence yields the exactness of
\[
    0 \to \H^1(\overline{X},\G_m)/\ell^n \to \H^2(\overline{X},\mu_{\ell^n}) \to \H^2(\overline{X},\G_m)[\ell^n] \to 0.
\]
Combining with the exactness of
\[
    0 \to \Pic^0(\overline{X}) \to \Pic(\overline{X}) \to \NS(\overline{X}) \to 0
\]
and the divisibility of $\Pic^0(\overline{X})$ (since multiplication by $\ell^n$ on an Abelian variety is an isogeny, hence surjective, by~\cite[p.~115, Theorem~8.2]{MilneAbelianVarieties}), hence $\H^1(\overline{X},\G_m)/\ell^n = \Pic(\overline{X})/\ell^n = \NS(\overline{X})/\ell^n$, and passage to the inverse limit $\varprojlim_n$ gives us
\[
    0 \to \NS(\overline{X}) \otimes_\Z \Z_\ell \to \H^2(\overline{X},\Z_\ell(1)) \to T_\ell\H^2(\overline{X},\G_m) \to 0
\]
since the $\NS(\overline{X})/\ell^n$ are finite by~\cite[p.~215, Theorem~V.3.25]{MilneEtaleCohomology}, so they satisfy the Mittag-Leffler condition.  As $\NS(\overline{X})$ is torsion-free (by assumption (a) above) and $T_\ell\H^2(\overline{X},\G_m)$ too (as a Tate module), it follows that $\H^2(\overline{X},\Z_\ell(1))$ is torsion-free, so also
\[
    \H^2(\overline{X},T_\ell\Acal)^\Gamma = \H^2(\overline{X},\pi^*(T_\ell A))^\Gamma = \H^2(\overline{X},\pi^*(T_\ell A(-1)) \otimes \Z_\ell(1))^\Gamma = (\H^2(\overline{X},\Z_\ell(1)) \otimes_{\Z_\ell} T_\ell A(-1))^\Gamma
\]
by~\cref{lemma:BasicPropertiesOfConstantAbelianSchemes}\,2 (here we are using that $\Acal/X$ is constant) and the projection formula for $\pi: \overline{X} \to k$ (similar to~\cref{lemma:TateTwistherausziehen}), so
\[
    \card{\H^2(\overline{X},T_\ell\Acal)^\Gamma} = 1
\]
since this group is finite by~\cref{cor:weightnotequalzerofinite} (having weight $2-1=1 \neq 0$ by~\cref{thm:YogaOfWeights} and~\cref{prop:RiundTateModul}) and torsion-free (as a subgroup of a tensor product of torsion-free finite rank groups).
\end{remark}

\subsection{The case of a curve as a basis}
Let $X/k$ be a smooth projective geometrically connected \emph{curve} with function field $K = k(X)$, base point $x_0 \in X(k)$, Albanese variety $A$, Abel-Jacobi map $\phi: X \to A$ with canonical principal polarisation $c: A \isoto A^t$, and $B/k$ be an Abelian variety.

Let
\[
\qu{\cdot,\cdot}: \Hom_k(A,B) \times \Hom_k(B,A) \to \Z, (\alpha,\beta) \mapsto \qu{\alpha,\beta} := \Tr(\beta \circ \alpha: A \to A) \in \Z
\]
be the trace pairing, the trace being taken as an endomorphism of $A$ as in~\cite{LangAV}.  By~\cite[p.~186\,f., Theorem~3]{LangAV}, this equals the trace taken as an endomorphism of the Tate module $T_\ell A$ or $\H^1(\overline{A},\Z_\ell)$ (they are dual to each other by~\cref{thm:Tate module and etale cohomology}, so for the trace, it does not matter which one we are taking).

We now show that our trace pairing is equivalent to the usual Néron-Tate height pairing on curves and is thus a sensible generalisation to the case of a higher dimensional base.

\begin{lemma} \label[lemma]{prop:adjunctionPoincareBundle}
	Let $X, Y$ be Abelian varieties over a field $k$ and $f \in \Hom_k(X,Y)$.  Then
	\[
	(f \times \id_{Y^t})^*\Pcal_Y \iso (\id_X \times f^t)^*\Pcal_X
	\]
	in $\Pic(X \times_k Y^t)$.
\end{lemma}
\begin{proof}
	By the universal property of the Poincaré bundle $\Pcal_X$ applied to $(f \times \id_{\Pcal_y})^*\Pcal_Y$, there exists a unique map $\hat{f}: X^t \to Y^t$ such that
	\begin{align} \label{eq:universal property of Poincare bundle}
	(f \times \id_{Y^t})^*\Pcal_Y \iso (\id_X \times \hat{f})^*\Pcal_X.
	\end{align}
	It remains to show that $\hat{f} = f^t$.
	
	Let $T/k$ be a variety and $\Lcal \in \Pic^0(Y \times_k T)$ arbitrary.  By the universal property of the Poincaré bundle $\Pcal_Y$, there exists $g: T \to Y^t$ such that $\Lcal = (\id_Y \times g)^*\Pcal_Y$.  We want to show $\hat{f}_*: Y^t(T) \to X^t(T), g \mapsto \hat{f}g$ equals $f^t: \Pic^0(Y \times_k T) \to \Pic^0(X \times_k T), \Lcal \mapsto f^*\Lcal$.  But we have
	\begin{align*}
	f^t(\Lcal)           &= (f \times \id_T)^*\Lcal \\
	&= (f \times \id_T)^*(\id_Y \times g)^*\Pcal_Y \\
	&= (f \times g)^*\Pcal_Y \\
	&= (\id_X \times g)^*(f \times \id_{Y^t})^*\Pcal_Y \\
	&= (\id_X \times g)^*(\id_X \times \hat{f})^*\Pcal_X    \quad\text{by~\eqref{eq:universal property of Poincare bundle}} \\
	&= (\id_X \times \hat{f}g)^*\Pcal_X \\
	&= \hat{f}_*(\Lcal)
	\end{align*}
	for any $\Lcal \in \Pic^0(Y \times_k T)$.
\end{proof}

\begin{lemma} \label[lemma]{prop:ThetaDivisor}
	Let $\theta^-$ be the class of $[-1]^*\Theta$ with the Theta divisor as in~\cite[p.~272, Remark~8.10.8]{BombieriGubler} and $\delta_1 \in \Pic(X \times_k A)$ as in~\cite[p.~278, l.~$-4$]{BombieriGubler} the Poincaré class.  Let $\phi$ be the Abel-Jacobi map and $\phi_{\theta^-}$ as in~\cite[p.~252, Theorem~8.5.1]{BombieriGubler}.  Let $c_A = m^*\theta^- - \pr_1^*\theta^- - \pr_2^*\theta^- \in \Pic(A \times_k A)$ with $m: A \times_k A \to A$ the addition morphism and $\pr_i: A \times_k A \to A$ the projections.  Then
	\begin{equation} \label{eq:ThetaUndPoincareDivisor}
	(\phi \times \id_A)^*c_A = -\delta_1
	\end{equation}
	and
	\begin{equation} \label{eq:PoincareUndThetaDivisor}
	(\id_A \times \phi_{\theta^-})^*\Pcal_A = c_A.
	\end{equation}
\end{lemma}
\begin{proof}
	See~\cite[p.~279, Propositions~8.10.19 and~8.10.20]{BombieriGubler}.
\end{proof}

\begin{theorem}[The trace and the height pairing for curves] \label[theorem]{thm:TraceAndHeightPairingForCurves}
	Let $X/k$ be a smooth projective geometrically connected \emph{curve} with Albanese variety $A$.  Then the trace pairing
	\[
	\Hom_k(A,B) \times \Hom_k(B,A) \stackrel{\circ}{\to} \End_k(A) \stackrel{\Tr}{\to} \Z, (\alpha,\beta) \mapsto \qu{\alpha,\beta}
	\]
	equals the following height pairing
	\begin{align*}
	(\alpha,\beta)_{\mathrm{ht}}  &:= \deg_X(-(\gamma(\alpha),\gamma'(\beta))^*\Pcal_B) = \deg_X(-(\alpha\phi, \beta^t c\phi)^*\Pcal_B),
	\end{align*}
	where $\phi: X \to A$ is the Abel-Jacobi map associated to a rational point of $X$, $c: A \isoto A^t$ is the canonical principal polarisation associated to the theta divisor, and $\gamma(\alpha), \gamma'(\beta)$ are the following compositions
	\begin{align*}
	\gamma(\alpha)&: X \stackrel{\phi}{\to} A \stackrel{\alpha}{\to} B, \\
	\gamma'(\beta)&: X \stackrel{\phi}{\to} A \stackrel{c}{\to} A^t \stackrel{\beta^t}{\to} B^t,
	\end{align*}
	and $(\alpha,\beta)_{\mathrm{ht}}$ is equal to the usual Néron-Tate canonical height pairing up to a sign.
\end{theorem}
\begin{proof}
	By~\cite[p.~100]{Milne-Tate-Shafarevic}, we have
	\[
	\qu{\alpha,\beta} = \deg_X((\id_X,\beta\alpha\phi)^*\delta_1),
	\]
	where $\delta_1 \in \Pic(X \times_k A)$ is a divisorial correspondence such that
	\[
	(\id_X \times \phi)^*\delta_1 = \Delta_X - \{x_0\}\times X - X \times \{x_0\}
	\]
	with the diagonal $\Delta_X \hookrightarrow X \times_k X$, see~\cite[p.~279, Proposition~8.10.18]{BombieriGubler}.
	
	Note the property~\cref{prop:ThetaDivisor} of the Theta divisor $\Theta$ of the Jacobian $A$ of $C$ on $A$ (which is defined in~\cite[p.~272, Remark~8.10.8]{BombieriGubler}) and let $\Theta^- = [-1]^*\Theta$ with $\theta$ and $\theta^-$ denoting the respecting divisor class.  The Theta divisor induces the canonical principal polarisation $\phi_\theta = c: A \isoto A^t$.  Therefore
	\begin{align*}
	(\alpha\phi \times \beta^t c\phi)^*\Pcal_B &= (\alpha\phi \times c\phi)^*(\id_X \times \beta^t)^*\Pcal_B \\
	&= (\alpha\phi \times c\phi)^*(\beta \times \id_{A^t})^*\Pcal_A \quad\text{by~\cref{prop:adjunctionPoincareBundle}} \\
	&= (\beta\alpha\phi \times c\phi)^*\Pcal_A \\
	&= (\beta\alpha\phi \times \phi_\theta\phi)^*\Pcal_A \\
	&= (\beta\alpha\phi \times \phi)^*(\id_A \times \phi_\theta)^*\Pcal_A  \\
	&= (\beta\alpha\phi \times \phi)^*c_A \quad\text{by~\eqref{eq:PoincareUndThetaDivisor}} \\
	&= (\phi \times \beta\alpha\phi)^*c_A \quad\text{by symmetry of $c_A$} \\
	&= -(\id_X \times \beta\alpha\phi)^*\delta_1 \quad\text{by~\eqref{eq:ThetaUndPoincareDivisor}}
	\end{align*}
	Summing up, one has
	\begin{align*}
	(\alpha,\beta)_{\mathrm{ht}} &= \deg_X(-(\id_X,\beta\alpha\phi)^*\delta_1) \\
	&= -\qu{\alpha,\beta}.
	\end{align*}
	
	By~\cite[p.~72, Théorème~5.4]{MoretBaillyMetriquesPermises}, the latter pairing equals the Néron-Tate canonical height pairing.
\end{proof}

\section{Proof of the conjecture for special Abelian schemes} \label{sec:Verification}

We assume in this section that all varieties have a base point.  This assumption is needed for the existence of the Albanese variety in~\cref{prop:PicardSchemeOfProduct}.

\subsection{Picard and Neron-Severi groups of products}

\begin{proposition}[Picard scheme of a product] \label[proposition]{prop:PicardSchemeOfProduct}
Let $X, Y$ be smooth proper varieties over a field $k$ with a $k$-rational point.  Then there is an exact sequence of $k$-group schemes
\[
    0 \to \PPic_{X/k} \times_k \PPic_{Y/k} \to \PPic_{X \times_k Y/k} \to \underline{\mathrm{Hom}}_k(\Alb_{X/k}, \PPic^0_{Y/k}),
\]
which is short exact on geometric points.
\end{proposition}
\begin{proof}
See~\cite{187445}.
\end{proof}

\begin{corollary} \label[corollary]{cor:Pic0OfAProduct}
Let $X, Y$ be smooth proper varieties over an algebraically field $k$ with a $k$-rational point.  If $\PPic^0_{X/k}$ and $\PPic^0_{Y/k}$ are reduced, so is $\PPic^0_{X \times_k Y/k} = \PPic^0_{X/k} \times_k \PPic^0_{Y/k}$.
\end{corollary}
\begin{proof}
One has $\PPic^0_{X \times_k Y/k} = \PPic^0_{X/k} \times_k \PPic^0_{Y/k}$ from the exact sequence in~\cref{prop:PicardSchemeOfProduct} by taking the connected component of $0$ and since the $\underline{\mathrm{Hom}}$-scheme is discrete.  Now use that the fibre product of reduced varieties over an algebraically closed field is reduced~\cite[p.~135, Proposition~5.49]{Goertz-Wedhorn}.
\end{proof}

\begin{corollary} \label[corollary]{cor:NSOfAProductIsFree}
Let $X, Y$ be smooth proper varieties over an algebraically closed field $k$ with a $k$-rational point.  If $\NS(X)$ and $\NS(Y)$ are free, so is $\NS(X \times_k Y)$.
\end{corollary}
\begin{proof}
By~\cref{prop:PicardSchemeOfProduct} and~\cref{cor:Pic0OfAProduct}, there is a commutative diagram with exact rows
\[\begin{tikzcd}
    0 \arrow[r] & \Pic^0(X) \times \Pic^0(Y) \arrow[d,hookrightarrow]\arrow[r,"\iso"] & \Pic^0(X \times_k Y) \arrow[r]\arrow[d,hookrightarrow] & 0 \ar[d] \\
    0 \arrow[r] & \Pic(X) \times \Pic(Y) \arrow[r] & \Pic(X \times_k Y) \arrow[r] & \Hom_k(\PPic^0_{X/k}, \PPic^0_{Y/k}) \arrow[r] & 0.
\end{tikzcd}\]
The snake lemma gives us a short exact sequence
\[
    0 \to \NS(X) \times \NS(Y) \to \NS(X \times_k Y) \to \Hom_k(\PPic^0_{X/k}, \PPic^0_{Y/k}) \to 0.
\]
Now use that $\Hom_k(A,B)$ for Abelian varieties $A,B$ over a field $k$ is a finitely generated free Abelian group, see~\cite[p.~122, Lemma~12.2]{MilneAbelianVarieties}.
\end{proof}

\subsection{Preliminaries on étale fundamental groups}

\begin{lemma} \label[lemma]{lemma:EtaleCoveringOfAProduct}
Let $X_i$, $i = 1,\ldots,n$ be connected proper varieties over an algebraically closed field $k$.  If $\tilde{X}$ is an étale covering of $X_1 \times_k \ldots \times_k X_n$, there are étale coverings $\tilde{X}_i$ of $X_i$ and an étale covering $\tilde{X}_1 \times_k \ldots \times_k \tilde{X}_n \to \tilde{X}$.
\end{lemma}
\begin{proof}
By~\cite[p.~203\,f., Corollaire~X.1.7]{SGA1}, the étale fundamental group of a product of connected proper varieties over an algebraically closed field is the product of the étale fundamental groups of its factors.  Now use that for an open subgroup $H \leq G$ of a profinite group $G = G_1 \times \ldots \times G_n$ contains an open subgroup $H_1 \times \ldots \times H_n$ of $G$ with $H_i \leq G_i$ open.  (One can take $H_i = G_i \cap H$.)
\end{proof}

\begin{proposition} \label[proposition]{prop:ConnectedFiniteEtaleCovering}
Let $G/S$ be finite étale over $S$ connected.  Then there is a connected finite étale covering $S'/S$ of degree dividing $\deg(G/S)!$ such that $G \times_S S'/S'$ is constant.  
\end{proposition}
\begin{proof}
Choose a geometric point $s$ of $S$.  Let $X$ be the $\pi_1^\et(S,s)$-set corresponding to $G/S$, and let $H \subseteq \pi_1^\et(S,s)$ be the subgroup corresponding to the elements that act as the identity on $X$, the kernel of $\pi_1^\et(S,s) \to \Aut(X)$.  Let $S'$ be the finite étale covering corresponding to the $\pi_1^\et(S,s)$-set $\pi_1^\et(S,s)/H$, which is connected as $\pi_1^\et(S,s)$ acts transitively on $\pi_1^\et(S,s)/H$.  The scheme $G \times_S S'/S'$ is constant by~\cite[p.~113, Corollaire~V.6.5]{SGA1} applied to the functor $- \times_S S': \FEt/S \to \FEt/S'$ of Galois categories. 

Note that $|\Aut(X)| = |X|! = \deg(G/S)!$, so $\deg(S'/S) = [\pi_1^\et(S,s):H] \mid \deg(G/S)!$.
\end{proof}

\subsection{Isoconstant Abelian schemes}

\begin{theorem} \label[theorem]{thm:EllipticCurvesIsoconstant}
Let $k$ be a field of characteristic $p$ and $S/k$ be proper, reduced and connected.  Let $\Acal/S$ be a relative elliptic curve or a principally polarised Abelian scheme with constant isomorphism type of $\Acal[p]$.  Then there is a connected finite étale covering $S'/S$ such that $\Acal \times_S S'/S'$ is constant.
\end{theorem}
\begin{proof}
If $\Acal/S$ is a relative elliptic curve:  Choose $N \geq 3$ such that $N$ is invertible on $S$.  Since $\Ecal[N]/S$ is finite étale, by~\cref{prop:ConnectedFiniteEtaleCovering} there is a connected finite étale covering $S'/S$ such that there is an $S'$-isomorphism $\Ecal[N] \times_S S' \iso (\Z/N)^2$.  Since the fine ($N \geq 3$) moduli space $Y(N)$ of elliptic curves with full level-$N$ structure is affine by~\cite[p.~117, Corollary~4.7.2]{KatzMazur} and $S'$ is reduced and connected, by the coherence theorem, the morphism $S' \to Y(N)$ classifying $(\Ecal \times_S S', \Ecal[N] \times_S S')$ factors over a finite extension field $k'$ of $k$.  Hence $\Ecal \times_S S' \iso \Ecal^{\mathrm{univ}} \times_{Y(N)} \Spec(k')$ is constant.

If $\Acal/S$ is a principally polarised Abelian scheme with constant isomorphism type of $\Acal[p]$:  Use the same argument and use that there is a level-$n$ structure for some $n \geq 3$ not divisible by $p$ after finite étale base extension and that the Ekedahl-Oort stratification of the moduli space $\Acal_{g,1,n} \otimes \F_p$ for $p \nmid n$ is quasi-affine~\cite[p.~348, Theorem~1.2]{OortTexelProceedings}.
\end{proof}

\begin{lemma} \label[lemma]{lemma:SpreadingOutHomomorphismsOfAbelianSchemes}
Let $X$ be a normal Noetherian integral scheme with function field $K = k(X)$, $\Acal$ and $\Bcal$ Abelian schemes over $X$ and $L/K$ be a separable field extension.  Given a homomorphism $f_L \in \Hom_X(\Acal_L,\Bcal_L)$, there exists a finite étale covering $X'/X$ with function field $L'$ with $L \supseteq L' \supseteq K$ and an extension of $f_{L'}$ to $f_{X'} \in \Hom_{X'}(\Acal_{X'},\Bcal_{X'})$.
\end{lemma}
\begin{proof}
Since $X$ is normal Noetherian integral, the Abelian schemes $\Acal,\Bcal$ are projective over $X$ by~\cite[p.~161, Théorème~XI.1.4]{RaynaudSLN}.  Since $X$ is Noetherian and $\Acal,\Bcal$ are also flat over $X$, by~\cite[p.~133, Theorem~5.23]{FGAExplained}, there exists the $\underline{\Hom}$-scheme $\underline{\Hom}_X(\Acal,\Bcal)$ over $X$, which is an open subscheme of the Hilbert scheme $\mathrm{Hilb}_{\Acal \times_X \Bcal/X}$, which is separated and locally of finite presentation over $X$.  Since for a discrete valuation ring $R$ with quotient field $\Quot(R)$, arguing as in~\cite[p.~15, proof of Proposition~1.2/8]{BLR}, there is for $f_{\Quot(R)}: \Acal_{\Quot(R)} \to \Bcal_{\Quot(R)}$ a unique (by separatedness) extension to $f_R: \Acal_R \to \Bcal_R$, the connected components of $\Hom_X(\Acal,\Bcal)$ are proper over $X$.  

By the infinitesimal lifting criterion for unramified morphisms, $\Hom_X(\Acal,\Bcal) \to X$ is also unramified:  Let $(R,\mathfrak{m})$ be a local Artinian ring with residue field $k$.  Then $\Hom_R(\Acal_R,\Bcal_R) \hookrightarrow \Hom_k(\Acal_k,\Bcal_k)$ is injective since $\Spec(R)$ consists of a single point:  Namely, if $f: \Acal_R \to \Acal_R$ maps to $f_k = 0$, $f = 0$ by the rigidity lemma~\cite[p.~115, Theorem~6.1\,1)]{MumfordGIT}.  Hence any component of $\Hom_X(\Acal,\Bcal)$ that is dominant over $X$ is finite (by Zariski's main theorem, since it is proper and quasi-finite) and étale over $X$ (since $X$ is integral and normal, hence geometrically unibranch, so dominant, finite and unramified implies étale by~\cite[p.~157, Théorème~18.10.1]{EGAIV4}).
\end{proof}

For the definition of a supersingular Abelian variety see~\cite[p.~113, Definition~4.1]{OortSubvarietiesOfModuliSpaces}.  A supersingular Abelian scheme is an Abelian schemes with all fibres supersingular Abelian varieties, equivalently (for an integral base) if the generic fibre is supersingular (this follows from~\cref{thm:SupersingularAbelianSchemeIsogenousToProductOfSupersingularEllipticCurves}).

\begin{theorem}[supersingular Abelian schemes] \label[theorem]{thm:SupersingularAbelianSchemeIsogenousToProductOfSupersingularEllipticCurves}
Let $X$ be a normal Noetherian integral scheme of characteristic $p > 0$ and $\Acal/X$ be an Abelian scheme with supersingular generic fibre.  Then there exists a finite étale covering $X'/X$, a supersingular elliptic curve $E/\F_p$ and an isogeny $(E \times_{\F_p} X')^g \to \Acal \times_X X'$.
\end{theorem}
\begin{proof}
Let $K = k(X)$ be the function field of $X$.  By~\cite[p.~113, Theorem~4.2]{OortSubvarietiesOfModuliSpaces}, $\Acal_{\overline{K}}$ is isogenous to $E_{\overline{K}}^g$ with $E_{\overline{K}}/\overline{K}$ any (!) supersingular elliptic curve (any two supersingular elliptic curves over an algebraically closed field are isogenous, see~\cite[p.~113]{OortSubvarietiesOfModuliSpaces}).  Note that for any prime $p$, there exists a supersingular elliptic curve over $\F_p$, see~\cite[p.~148\,f., Theorem~V.4.1\,(c)]{SilvermanAEC} for $p > 2$ and the text before this theorem for $p = 2$.  By~\cite[p.~146, Corollary~20.4\,(b)]{MilneAbelianVarieties} applied to the primary field extension $\overline{K}/K^{\mathrm{sep}}$, there is a separable field extension $L/K$ and an isogeny $E_L^g \to \Acal_L$.  Since $E/\F_p$ extends to $E \times_{\F_p} X$ over $X$, the claim follows from~\cref{lemma:SpreadingOutHomomorphismsOfAbelianSchemes}.
\end{proof}

\begin{definition}
We call an Abelian scheme $\Acal/X$ \defn{$\ell'$-isoconstant} if there is a proper, surjective, generically étale $\ell'$-morphism of regular schemes $f: X' \to X$ (an $\ell'$-alteration) such that $\Acal \times_X X'$ is constant.
\end{definition}

The following theorem about descent of finiteness of the Tate-Shafarevich group together with~\cref{thm:MilnesMainTheorem} implies~\cref{mainthm:BSDforIsoconstant} from the introduction.

\begin{theorem}[invariance of finiteness of $\Sha$ under alterations] \label[theorem]{thm:isotrivial}
Let $\ell$ be a prime invertible on $X$.  Let $f: X' \to X$ be a proper, surjective, generically étale $\ell'$-morphism of regular schemes.  If $\Acal$ is an Abelian scheme on $X$ such that the $\ell^\infty$-torsion of the Tate-Shafarevich group $\Sha(\Acal'/X')$ of $\Acal' := f^*\Acal = \Acal \times_X X'$ is finite, then the $\ell^\infty$-torsion of the Tate-Shafarevich group $\Sha(\Acal/X)$ is finite.
\end{theorem}
\begin{proof}
See~\cite[p.~238, Theorem~4.29]{KellerSha}.
\end{proof}

\begin{corollary} \label[corollary]{cor:BSDoverCurves}
Let $X$ be a product of smooth proper curves, Abelian varieties and K3 surfaces over a finite field of characteristic $p$.  Now let $\Acal$ be an Abelian $X$-scheme belonging to one of the following three classes:
\begin{enumerate}
	\item a relative elliptic curve
	\item an Abelian scheme such that the isomorphism type of $\Acal[p]$ is constant
	\item an Abelian scheme with supersingular generic fibre
\end{enumerate}
Then the prime-to-$p$ part of our analogue of the conjecture of Birch and Swinnerton-Dyer holds for $\Acal/X$ and, if $\Acal/X$ is a relative elliptic curve, $\Br(\Acal)[\text{non-$p$}]$ is finite.  If $X$ is a curve, the full conjecture of Birch and Swinnerton-Dyer holds for $\Acal/X$.  Furthermore, the Tate conjecture holds in dimension $1$ for $\Acal$.
\end{corollary}
\begin{proof}
The conditions~(a) and~(b) from~\cref{thm:MilnesMainTheorem} are satisfied for $S'$ in~\cref{thm:EllipticCurvesIsoconstant} by~\cref{ex:aandb} if the base scheme is a curve or an Abelian variety as a finite étale constant connected covering of a curve or an Abelian variety is again a curve or an Abelian variety, respectively:  For curves, this is clear, and for Abelian varieties see~\cite[p.~155, Theorem of Serre-Lang]{MumfordAbelianVarieties}.  So one has~(a) and~(b) for a product from~\cref{cor:Pic0OfAProduct} and~\cref{cor:NSOfAProductIsFree}.  A K3 surface $X/k$ has $\pi_1^\et(X) = \pi_1^\et(k)$ by~\cite[p.~131, proof of Theorem~1.1]{HuybrechtsK3Surfaces} and the homotopy exact sequence $1 \to \pi_1^\et(X \times_k k^{\sep}) \to \pi_1^\et(X) \to \pi_1^\et(k) \to 1$.  Therefore, a connected étale covering of $X$ is of the form $X \times_k K$ with $K/k$ a finite separable field extension.  Since $\H^1_\Zar(X,\Ocal_X) = 0$, also $\H^1(X \times_k K,\Ocal_{X \times_k K}) = 0$ by~\cite[p.~189, Corollary~5.2.27]{Liu2006}.  Furthermore, $\Omega^2_{X \times_k K} = \Ocal_{X \times_k K}$ by~\cite[p.~271, Proposition~6.1.24\,(a)]{Liu2006}.  Now apply~\cref{thm:isotrivial} to the étale covering from~\cref{lemma:EtaleCoveringOfAProduct} to get~(a) and~(b) for the covering.

For an Abelian scheme with supersingular generic fibre use the same argument together with~\cref{thm:SupersingularAbelianSchemeIsogenousToProductOfSupersingularEllipticCurves} and isogeny invariance of the finiteness of the Tate-Shafarevich group~\cite[p.~240, Theorem~4.31]{KellerSha}.


Note that $\Acal/X$ is $\ell'$-isoconstant for some $\ell \neq \Char(k)$, and then we can use (a) $\implies$ (b) from~\cref{thm:BSDI} to get independence from $\ell$.  Using~\cite[p.~286, Theorem~4.8]{Bauer}, this proves the conjecture of Birch and Swinnerton-Dyer for elliptic curves with good reduction everywhere over $1$-dimensional global function fields.

The finiteness of the prime-to-$p$ part of the Brauer group of the absolute variety $\Ecal$ over an Abelian variety $X$ follows from the finiteness of $\Br(X)[\text{non-$p$}]$~\cite{ZarhinBrauerGroupAbelianVarietyOverFiniteField} and~\cite[p.~237, Theorem~4.27]{KellerSha}.  For $X$ a curve, see the proof of~\cref{cor:BrauerGroupFinite}.  For $X$ a K3 surface, see~\cite[p.~11405, Theorem~1.3]{SkorobogatovZarhinFinitenessOfBrauerGroupOfK3Surfaces} and~\cite[p.~1, Theorem~1.1]{FinitenessOfBrauerGroupOfK3SurfacesInChar2} and note that the Brauer group of a finite field is trivial.

The Tate conjecture holds in dimension $1$ since the Kummer sequence gives an exact sequence
\[
    0 \to \Pic(\Ecal) \otimes_\Z \Z_\ell \to \H^1(\Ecal, \Z_\ell(1)) \to T_\ell\Br(\Ecal) \to 0
\]
and $\Br(\Ecal)[\ell^\infty]$ is finite, so $T_\ell\Br(\Ecal) = 0$ by~\cref{lemma:PropertiesOfTateModule}\,\ref{lemma:Tate-module-of-finite-group-is-trivial}.
\end{proof}

The $p$-part will be covered in a forthcoming article~\cite{KellerpPartTateShafarevichAndBrauer}.  There, we prove that the Brauer group of an Abelian variety over a finite field is finite (including the $p$-part), descent of finiteness of the $p^\infty$-torsion of the Tate-Shafarevich group under alterations, and isogeny invariance of finiteness of the $p^\infty$-torsion of the Tate-Shafarevich group.

\begin{corollary} \label[corollary]{cor:BrauerGroupFinite}
Let $C/\F_q$ be a smooth proper geometrically connected curve and $\Ecal/C$ be a relative elliptic curve.  Then $\Br(\Ecal) = \Sha(\Ecal/C)$ is finite and of square order, and the Tate conjecture holds for $\Ecal$.
\end{corollary}
\begin{proof}
This follows from~\cite[p.~237, Theorem~4.27]{KellerSha} and~\cref{cor:BSDoverCurves}, and since $\Br(C) = 0$ by class field theory, see~\cite[p.~137, Remark~I.A.15 and p.~131, Theorem~I.A.7]{MilneADT} and the Albert-Brauer-Hasse-Noether theorem~\cite[p.~437, Theorem~8.1.17]{CohomologyOfNumberFields}.

The statement about the square order follows from~\cite{LiuBrauer}.  The Tate conjecture in dimensions other than $1$ is trivial for a surface.
\end{proof}


\section{Reduction to the case of a surface or a curve as a basis} \label{sec:Reduction}

\begin{theorem} \label[theorem]{thm:ReductionToSurface}
	If the analogue of the conjecture of Birch-Swinnerton-Dyer holds for a prime $\ell$ invertible on the base and for all Abelian schemes over all smooth projective geometrically integral surfaces, then it holds over arbitrary dimensional bases.
	
	More precisely, if there is a sequence $S \hookrightarrow \ldots \hookrightarrow X$ of ample smooth projective geometrically integral hypersurface sections with a surface $S$ and the conjecture holds for $\Acal/S$, then it holds for $\Acal/X$.
\end{theorem}

The basic idea is using ample hypersurface sections, Poincaré duality, the affine Lefschetz theorem and that the conjecture of Birch and Swinnerton-Dyer depends only on $\Sha(\Acal/X) = \Het^1(X,\Acal)$ in cohomological degree $1$.

\begin{proof}
Let $Y \hookrightarrow X$ be an ample smooth geometrically connected hypersurface section (this exists by Poonen's Bertini theorem for varieties over finite fields~\cite[Proposition~2.7]{PoonenBertini}) with (necessarily) affine complement $U \hookrightarrow X$.  Base changing to $\overline{k}$ and writing $\overline{X} = X \times_k \overline{k}$ etc., one has by~\cite[p.~94, Remark~III.1.30]{MilneEtaleCohomology} a long exact sequence
\begin{align} \label{eq:longexactsequence}
	\ldots \to \H^i_c(\overline{U}, \Acal[\ell^n]) \to \H^i(\overline{X}, \Acal[\ell^n]) \to \H^i(\overline{Y}, \Acal[\ell^n]) \to \H^{i+1}_c(\overline{U}, \Acal[\ell^n]) \to \ldots
\end{align}
(Note that $\H^i_c(\overline{X}, \Fcal) = \H^i(\overline{X}, \Fcal)$ since $\overline{X}/\overline{k}$ is proper, and likewise for $\overline{Y}$.)

Since $\Acal[\ell^n]/X$ is étale, Poincaré duality~\cite[p.~276, Corollary~VI.11.2]{MilneEtaleCohomology} gives us
\[
	\H^i_c(\overline{U},\Acal[\ell^n]) = \H^{2d-i}(\overline{U},(\Acal[\ell^n])^\vee(d)).
\]
(Note that the varieties live over a separably closed field.)  By the affine Lefschetz theorem~\cite[p.~253, Theorem~VI.7.2]{MilneEtaleCohomology}, one has $\H^{2d-i}(\overline{U},(\Acal[\ell^n])^\vee(d)) = 0$ for $2d-i > d$, i.\,e.\ for $i < d$.  Analogously, $\H^{i+1}_c(\overline{U}, \Acal[\ell^n]) = 0$ for $i+1 < d$.  Plugging this into~\eqref{eq:longexactsequence}, one gets an isomorphism
\begin{equation} \label{eq:HibarXbarYiso}
	\H^i(\overline{X}, \Acal[\ell^n]) \isoto \H^i(\overline{Y}, \Acal[\ell^n])
\end{equation}
for $i+1 < d$. Inductively, it follows that the cohomology groups of $\overline{X}$ in dimension $i = 0,1$ are isomorphic to the cohomology groups of a smooth projective geometrically integral \emph{surface} ($d = 2$) $S/k$.

Since $\cd_\ell(k) = 1$, the Hochschild-Serre spectral sequence degenerates on the $E_2$-page giving exact sequences
\[
	0 \to \H^{i-1}(\overline{X},\Acal[\ell^n])_\Gamma \to \H^i(X,\Acal[\ell^n]) \to \H^i(\overline{X},\Acal[\ell^n])^\Gamma \to 0
\]
and similar for $S$, which implies isomorphisms $\H^i(X,\Acal[\ell^n]) \isoto \H^i(S,\Acal[\ell^n])$ for $i = 0,1$ by the $5$-lemma and~\eqref{eq:HibarXbarYiso}.

It follows that there is a commutative diagram with exact rows
\[\begin{tikzcd}
	 0 \arrow[r] & \Acal(X)/\ell^n \arrow[r]\arrow[d,hookrightarrow] & \H^1(X,\Acal[\ell^n]) \arrow[r]\arrow[d,"\iso"] & \Sha(\Acal/X)[\ell^n] \arrow[r]\arrow[d,twoheadrightarrow] & 0\\
	 0 \arrow[r] & \Acal(S)/\ell^n \arrow[r] & \H^1(S,\Acal[\ell^n]) \arrow[r] & \Sha(\Acal/S)[\ell^n] \arrow[r] & 0.
\end{tikzcd}\]
Passing to the inverse limit $\varprojlim_n$ and using $\varprojlim_n^1\Acal(X)/\ell^n = 0$ (and similar for $S$) because the $\Acal(X)/\ell^n$ are finite by the (weak) Mordell-Weil theorem~\cref{thm:MordellWeil} and the Néron mapping property $\Acal(X) = A(K)$~\cref{thm:NeronModel}, one has a commutative diagram with exact rows
\begin{equation}\begin{tikzcd}\label{eq:ReduktionaufS}
	0 \arrow[r] & \Acal(X)\otimes_\Z\Z_\ell \arrow[r]\arrow[d,hookrightarrow] & \H^1(X,T_\ell\Acal) \arrow[r]\arrow[d,"\iso"] & T_\ell\Sha(\Acal/X) \arrow[r]\arrow[d,twoheadrightarrow] & 0\\
	0 \arrow[r] & \Acal(S)\otimes_\Z\Z_\ell \arrow[r] & \H^1(S,T_\ell\Acal) \arrow[r] & T_\ell\Sha(\Acal/S) \arrow[r] & 0.
\end{tikzcd}\end{equation}
By the snake lemma,
\begin{equation} \label{eq:kercokerAXAY}
	\ker\big(T_\ell\Sha(\Acal/X) \twoheadrightarrow T_\ell\Sha(\Acal/S)\big) = \coker\big(\Acal(X)\otimes_\Z\Z_\ell \hookrightarrow \Acal(S)\otimes_\Z\Z_\ell\big)
\end{equation}
is a finitely generated free $\Z_\ell$-module (since $T_\ell\Sha(\Acal/X)$ is), so $T_\ell\Sha(\Acal/X) \isoto T_\ell\Sha(\Acal/S)$ iff $\rk \Acal(X) = \rk \Acal(S)$.

\begin{proposition} \label[proposition]{prop:AXAS}
	Let $X$ be a smooth projective geometrically integral variety over a finite field  of characteristic $p$.  Let $Y \hookrightarrow X$ be an ample smooth projective geometrically integral hypersurface section with $\dim{Y} \geq 2$ and affine complement $U$.  Let $\Acal/X$ be an Abelian scheme.  Then the restriction morphism $\Acal(X) \to \Acal(Y)$ is an isomorphism (away from $p$).
\end{proposition}
\begin{proof}
	By~\cite[p.~94, Remark~III.1.30]{MilneEtaleCohomology}, there is an exact sequence
	\[
		0 \to \H^0_c(U,\Acal) \to \H^0(X,\Acal) \to \H^0(Y,\Acal) \to \H^1_c(U,\Acal).
	\]
	The injectivity of $\Acal(X) \to \Acal(Y)$ follows from~\eqref{eq:ReduktionaufS} (or $\H^0_c(U,\Acal) = 0$ since $U$ is affine).  For the surjectivity of $\Acal(X) \to \Acal(Y)$ away from $p$, it suffices to show that $\H^1_c(U,\Acal)[\text{non-$p$}] = 0$ (or at least that $\H^1_c(U,\Acal)[\text{non-$p$}]$ is finite/torsion since the cokernel is torsion-free away from $p$ by~\eqref{eq:kercokerAXAY}).  The Kummer exact sequence $0 \to \Acal[\ell^n] \to \Acal \to \Acal \to 0$ with $\ell$ invertible on $U$ induces an exact sequence
	\[
		\H^1_c(U,\Acal[\ell^n]) \to \H^1_c(U,\Acal) \stackrel{\ell^n}{\to} \H^1_c(U,\Acal) \to \H^2_c(U,\Acal[\ell^n]).
	\]
	Since $\H^1_c(U,\Acal[\ell^n]) = 0 = \H^2_c(U,\Acal[\ell^n])$ because of $\dim{U} > 2$ as above by Poincaré duality and the affine Lefschetz theorem, $\H^1_c(U,\Acal)$ is $\ell$-divisible and $\ell$-torsion free.  The exact sequence~\cite[p.~94, Remark~III.1.30]{MilneEtaleCohomology}
	\[
		\Acal(X) \to \Acal(Y) \to \H^1_c(U,\Acal) \to \Sha(\Acal/X) \to \Sha(\Acal/Y)
	\]
	shows, since the Mordell-Weil groups are finitely generated Abelian groups by the theorem of Mordell-Weil~\cref{thm:MordellWeil} and the Néron mapping property $\Acal(X) = A(K)$~\cref{thm:NeronModel} and the $\ell$-primary components of the (torsion) Tate-Shafarevich groups are cofinitely generated Abelian groups by~\cref{lemma:ShaEndlichenEllKorang}, that
	\[
		\H^1_c(U,\Acal)[\text{non-$p$}] \iso \bigoplus_{\ell \neq p}(F_\ell \oplus (\Q_\ell/\Z_\ell)^{n_\ell}) \oplus \Z^n
	\]
	with $F_\ell$ finite Abelian $\ell$-groups and $n, n_\ell \in \N$.  It follows from $\H^1_c(U,\Acal)/\ell = 0$ that $n = 0$ and then from $\H^1_c(U,\Acal)[\ell] = 0$ that $\H^1_c(U,\Acal)[\text{non-$p$}] = 0$.
\end{proof}

It also follows from $\H^i(\overline{X},\Acal[\ell^n]) \isoto \H^i(\overline{S},\Acal[\ell^n])$ for $i = 0,1$ and~\cref{def:Li} that $L(\Acal/X,s) = L(\Acal/S,s)$, so if the conjecture of Birch and Swinnerton-Dyer holds for $\Acal/S$, $\rk \Acal(X) = \rk \Acal(S)$ by~\cref{prop:AXAS} and $\Acal(X)\otimes_\Z\Z_\ell \isoto \Acal(S)\otimes_\Z\Z_\ell$. Hence, the analogue of the conjecture of Birch and Swinnerton-Dyer for $\Acal/X$ is equivalent to the conjecture for $\Acal/S$.
\end{proof}

\begin{theorem}
	If there is a smooth projective ample geometrically integral curve $C \hookrightarrow S$ with $\rk \Acal(S) = \rk \Acal(C)$, the analogue of the conjecture of Birch and Swinnerton-Dyer for $\Acal/S$ is equivalent to the classical conjecture for $\Acal/C$.
\end{theorem}
\begin{proof}
For an ample smooth projective geometrically integral \emph{curve} hypersurface section $C \hookrightarrow S$, one has still $\Acal(S)[\ell^n] \isoto \Acal(C)[\ell^n]$ and at least an injection $\H^1(S,\Acal[\ell^n]) \hookrightarrow \H^1(C,\Acal[\ell^n])$ for all $n \geq 0$ and $\H^1(S,T_\ell\Acal) \hookrightarrow \H^1(C,T_\ell\Acal)$.  Arguing in the same way as above using the commutative diagram with exact rows
\[\begin{tikzcd}
0 \arrow[r] & \Acal(S)\otimes_\Z\Z_\ell \arrow[r]\arrow[d,hookrightarrow] & \H^1(S,T_\ell\Acal) \arrow[r]\arrow[d,hookrightarrow] & T_\ell\Sha(\Acal/S) \arrow[r]\arrow[d] & 0\\
0 \arrow[r] & \Acal(C)\otimes_\Z\Z_\ell \arrow[r] & \H^1(C,T_\ell\Acal) \arrow[r] & T_\ell\Sha(\Acal/C) \arrow[r] & 0
\end{tikzcd}\]
and the snake lemma
\[
	\ker\big(T_\ell\Sha(\Acal/S) \to T_\ell\Sha(\Acal/C)\big) \hookrightarrow \coker\big(\Acal(S)\otimes_\Z\Z_\ell \hookrightarrow \Acal(C)\otimes_\Z\Z_\ell\big)
\]
with $T_\ell\Sha(\Acal/S)$ and hence the kernel being torsion-free, if the conjecture of Birch and Swinnerton-Dyer holds for $\Acal/C$ \emph{and} $\rk \Acal(S) = \rk \Acal(C)$, the analogue of the conjecture of Birch and Swinnerton-Dyer holds for $\Acal/S$.
\end{proof}

\begin{remark}
	So the question arises if there is always such a $C \hookrightarrow S \hookrightarrow \ldots \hookrightarrow X$ with $\rk \Acal(S) = \rk \Acal(C)$, see~\cite[Theorem~1.2\,(ii)]{GraberStarr} and Proposition~1.5\,(iii) (over uncountable fields).
	
	One always has the inequality $\rk \Acal(S) \leq \rk \Acal(C)$, so the analogue of the conjecture of Birch and Swinnerton-Dyer for $\Acal/X$ holds if there is such a $C \hookrightarrow X$ with $\rk \Acal(C) = 0$, e.\,g.\ $C \iso \mathbf{P}^1_k$ and $\Acal/C$ isoconstant, e.\,g.\ if $\Acal/C$ is a relative elliptic curve.
\end{remark}

\paragraph{Acknowledgements.}
I thank the anonymous referee for significantly improving the article, my advisor Uwe Jannsen and Maarten Derickx, Patrick Forré, Ulrich Görtz, Walter Gubler, Peter Jossen, Moritz Kerz, Klaus Künnemann, Frans Oort, Michael Stoll, Tamás Szamuely, Georg Tamme and, from mathoverflow, abx, ACL, Angelo, anon, Martin Bright, Holger Partsch, Kestutis Cesnavicius, Torsten Ekedahl, Laurent Moret-Bailly, nfdc23, Jason Starr, ulrich and xuhan; Yigeng Zhao for proofreading; finally the Studienstiftung des deutschen Volkes for financial and ideational support.

\phantomsection 
\addcontentsline{toc}{section}{\refname}
{\small
	\bibliographystyle{tkalpha3}
	\bibliography{BSD}
}

{\small \textsc{Timo Keller, Mathematisches Institut, Universität Bayreuth, 95440 Bayreuth, Germany}

{\small \emph{E-Mail address: firstname.lastname@uni-bayreuth.de}}

\end{document}